\documentclass[10pt]{amsart}
\usepackage{amssymb,amsmath,amsthm,amscd}
\usepackage{leftidx,mathrsfs}
\usepackage{graphicx,bbm}
\usepackage{color}
\usepackage{microtype,caption}
\usepackage{hyperref}
\usepackage[usenames,dvipsnames,svgnames,table]{xcolor}
\usepackage{tikz,yfonts,mathabx}
\usepackage{comment}
\usepackage{caption}
\usepackage{subcaption}
\usepackage{xfrac,faktor}
\usepackage{epstopdf} 
\usepackage{float}
\usepackage{tikz-cd}
\usepackage[all]{xy}
\usepackage{pgfplots}

\newtheorem{lemma}{Lemma}[section]
\newtheorem{theorem}[lemma]{Theorem}
\newtheorem{corollary}[lemma]{Corollary}
\newtheorem{proposition}[lemma]{Proposition}

\newtheorem{question}[lemma]{Question}
\theoremstyle{definition}
\newtheorem{definition}[lemma]{Definition}
\theoremstyle{remark}
\newtheorem{remark}[lemma]{Remark}

\theoremstyle{Theorem}
\newtheorem{thmx}{Theorem}
 
\pgfplotsset{compat=1.18}

\newcommand{\C}{\mathbb{C}}
\newcommand{\D}{\mathbb{D}}

\newcommand{\R}{\mathbb{R}}
\newcommand{\Z}{\mathbb{Z}}

\newcommand{\pslc}{{PSL_2 (\mathbb{C})}}

\newcommand\FF{{\mathcal F}}

\newcommand\LL{{\mathcal L}}
\newcommand\MM{{\mathcal M}}

\newcommand\PP{{\mathcal P}}

\newcommand\PMF{{\PP\kern-2pt\MM\FF}}

\newcommand\PML{{\PP\kern-2pt\MM\LL}}

\newcommand{\cC}{\mathcal{C}}

\newcommand{\cM}{\mathcal{M}}

\newcommand{\cS}{\mathcal{S}}
\newcommand{\cT}{\mathcal{T}}
\newcommand{\cK}{\mathcal{K}}

\DeclareMathOperator{\Int}{int}
\DeclareMathOperator{\Cl}{cl}

\renewcommand{\epsilon}{\varepsilon}
\renewcommand{\phi}{\varphi}
\newcommand{\mate}{\bot \!\! \! {\bot}}

\makeatletter
\DeclareFontFamily{U}{tipa}{}
\DeclareFontShape{U}{tipa}{m}{n}{<->tipa10}{}
\newcommand{\arc@char}{{\usefont{U}{tipa}{m}{n}\symbol{62}}}

\newcommand{\arc}[1]{\mathpalette\arc@arc{#1}}

\newcommand{\arc@arc}[2]{
  \sbox0{$\m@th#1#2$}
  \vbox{
    \hbox{\resizebox{\wd0}{\height}{\arc@char}}
    \nointerlineskip
    \box0
  }
}
\makeatother

\numberwithin{equation}{section}

\date{\today}

\begin{document}

\title[Matings, correspondences, and a Bers slice]{Matings, holomorphic correspondences, and\\ a Bers slice}

\author{Mahan Mj}
\address{School of Mathematics, Tata Institute of Fundamental Research, Mumbai-400005, India}
\email{mahan@math.tifr.res.in, mahan.mj@gmail.com}

\author{Sabyasachi Mukherjee}
\address{School of Mathematics, Tata Institute of Fundamental Research, Mumbai-400005, India}
\email{sabya@math.tifr.res.in, mukherjee.sabya86@gmail.com}

\subjclass[2010]{30C10, 30C20, 30F60, 37F05, 37F10, 37F31, 37F32, 37F34, 37F44 (primary), 37C85, 57M50, 20F65 (secondary)}
\keywords{Fuchsian group, Bowen-Series map, rational map, simultaneous uniformization, conformal mating, algebraic correspondence, Teichm{\"u}ller space, Bers slice}

\thanks{Both authors were  supported by  the Department of Atomic Energy, Government of India, under project no.12-R\&D-TFR-5.01-0500 as also  by an endowment of the Infosys Foundation.
	MM was also supported in part by a DST JC Bose Fellowship. SM was supported in part by SERB research project grant MTR/2022/000248.}

\begin{abstract}
There are two frameworks for mating Kleinian groups with rational maps on the Riemann sphere: the algebraic correspondence framework due to Bullett-Penrose-Lomonaco \cite{BP94,BL20a} and the simultaneous uniformization mating framework of \cite{MM1}. The current paper unifies and generalizes these two frameworks. To achieve this, we extend the mating framework of \cite{MM1} to genus zero hyperbolic orbifolds with at most one orbifold point of order $\nu \geq 3$ and at most one orbifold point of order two. We give an explicit description of the resulting conformal matings in terms of uniformizing rational maps. Using these rational maps, we construct correspondences that are matings of such hyperbolic orbifold groups (including punctured spheres and Hecke groups) with polynomials in real-symmetric hyperbolic components. We also define an algebraic parameter space of correspondences and construct an analog of a Bers slice of the above orbifolds in this parameter space.
\end{abstract}

\maketitle

\setcounter{tocdepth}{1}
\tableofcontents

\section{Introduction}\label{intro_sec}
Fatou \cite{Fatou29} observed an empirical similarity between the behavior of two complex  one-dimensional dynamical systems: one coming from iteration of polynomials, the other from Kleinian groups. This was developed into a systematic dictionary by Sullivan \cite{sullivan-dict} (see also \cite{McM94,ctm-classification,ctm-renorm,MS98,pilgrim,lyubich-minsky} etc.). Fatou's original suggestion \cite{Fatou29} of developing a unified framework for treating these two  kinds of dynamical systems in terms of \emph{correspondences} (multi-valued maps with holomorphic local branches) was pursued by Bullett and his co-authors in \cite{BP94,B00,BH00,BH07,BL20a,BL20b,BL22}. A new \emph{conformal matings} framework based on orbit-equivalence  was developed by the authors recently 
\cite{MM1} adapting the theme of Bers' simultaneous uniformization (in the context of Kleinian groups, \cite{Ber60}) and mating (in the context of polynomial and rational dynamics, \cite{douady-mating,hubbard-mating}). The conformal matings framework of
\cite{MM1} (see \cite{sullivan_survey} for a brief account of this framework) furnished new examples of mateable groups; however, two fundamental questions remained unanswered:

\begin{question}\label{qn_main}
\noindent\begin{enumerate}
	\item Identify the class of analytic functions obtained via the mating process of 
	\cite{MM1}.
	\item Is there a relationship between the Bullett-Penrose-Lomonaco correspondences 
	of \cite{BP94,B00,BL20a,BL20b,BL22} and the matings in   \cite{MM1}?
	\end{enumerate}
\end{question}

A primary aim of this paper is to answer both these questions by 
\begin{enumerate}
	\item characterizing the class of analytic functions obtained via the mating process of 
	\cite{MM1}, and
	\item establishing an equivalence between the two notions of matings coming from correspondences and simultaneous uniformization.
\end{enumerate}

\subsection*{The class of orbifolds}
Before we state the main theorems of the paper, let us describe the general class of orbifolds (equivalently, Fuchsian groups) that are the principal players in the game.
The family of correspondences most extensively studied by Bullett and his collaborators exhibit matings of the modular group $\mathrm{PSL}_2(\Z)$ and quadratic polynomial/rational maps. On the other hand, the  conformal matings framework of \cite{MM1} applies to Bowen-Series maps of Fuchsian punctured sphere groups, possibly with an order two elliptic element. In this paper, we  work with the following collection 
 of finite volume hyperbolic orbifolds that includes both these as special cases:
 
 \smallskip
 
 \noindent $\cS:=$ hyperbolic orbifolds  of genus zero with
	\begin{enumerate}
	\item at least one puncture,
	\item at most one order two orbifold point,
	\item at most one order $\nu\geq 3$ orbifold point.
	\end{enumerate}

\subsection*{Going up/going down and conformal matings} 
It should be pointed out at the outset that the modular group does not fit into the conformal matings framework of \cite{MM1} as the existence of an order three orbifold point forces its Bowen-Series map to be discontinuous. To circumvent this obstacle, one can pass to a $\nu-$fold cyclic cover $\widetilde{\Sigma}$ of  $\Sigma\in\cS$ such that the Bowen-Series map $A_{\widetilde{\Sigma}}^{\mathrm{BS}}$ of the Fuchsian group uniformizing $\widetilde{\Sigma}$ (equipped with suitable fundamental domains) only has controlled discontinuities. Remarkably, all these points of discontinuity disappear when one passes to appropriate factors of these Bowen-Series maps. Heuristically, passing to a factor dynamical system (going down) can be thought of as the dual of passing to a cyclic cover of $\Sigma$ (going up). This gives rise to continuous \emph{factor Bowen-Series maps} $A_{\widetilde{\Sigma}}^{\mathrm{fBS}}$ (see Figures~\ref{going_up_down_fig} and~\ref{going_up_down_2_fig}). This construction is detailed in Section \ref{factor_bs_sec}.	
\begin{figure}[H]
\[
\begin{tikzcd}
\widetilde{\Sigma} \arrow{d}[swap]{\mathrm{cover}} \arrow[r,rightsquigarrow] & A_{\widetilde{\Sigma}}^{\mathrm{BS}} \arrow{d}{\mathrm{factor}} \\
\Sigma  &  A_{\widetilde{\Sigma}}^{\mathrm{fBS}}
\end{tikzcd}
\]
\caption{Going up and going down}
\label{going_up_down_fig}
\end{figure}
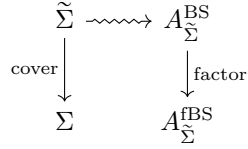

A key feature of a factor Bowen-Series map, one that lies at the heart of the construction of conformal matings, is that its restriction on the unit circle $\mathbb{S}^1$ is topologically conjugate to $z^d\vert_{\mathbb{S}^1}$, where
$$
	d\equiv d(\Sigma):=
   \begin{cases} 
   1-2\nu\cdot\chi_{\mathrm{orb}}(\Sigma)\quad \mathrm{if}\ \Sigma \textrm{ has an order } \nu\geq 3 \textrm{ orbifold point},\\
   1-2\chi_{\mathrm{orb}}(\Sigma)\quad \mathrm{if}\ \Sigma \textrm{ does not have an order } \nu\geq 3 \textrm{ orbifold point}.
   \end{cases}
  $$ 

Our first main theorem extends the conformal mating construction of \cite{MM1} to genus zero orbifolds in the above class.
We direct the reader to Section~\ref{conf_mating_sec} for the precise notion of a conformal mating. For now, it suffices to say that a conformal mating is a map $F: \overline{\Omega} \to \widehat{\mathbb{C}}$, where ${\Omega} \subset \widehat{\mathbb{C}}$ is open, and the dynamics of $F$ combines naturally the dynamics of a piecewise M\"obius map and a polynomial.

\begin{thmx}[Conformal matings of Factor Bowen-Series maps with polynomials]\label{conf_mating_intro_thm} $ $\\
\noindent Let $\Sigma\in\cS$, and $P$ be a complex polynomial in the principal hyperbolic component $\mathcal{H}_d$ of degree $d$ polynomials. Then the factor Bowen-Series map $A_{\widetilde{\Sigma}}^{\mathrm{fBS}}$ and  the polynomial $P$ are conformally mateable. Moreover, the conformal mating is unique up to M{\"o}bius conjugacy.
\end{thmx}

 See  Theorem~\ref{conf_mat_thm} for the proof  existence of conformal matings.
There are two special cases of Theorem~\ref{conf_mating_intro_thm} that need special mention:
\begin{enumerate}
\item The case where $\Sigma$ has no order $\nu (\geq 3)$ orbifold points. This is treated in Section~\ref{punc_sphere_bs_subsec}, and was dealt within the conformal matings framework of \cite{MM1}.
\item The case where   $\Sigma$ has exactly one cusp, i.e.\ it is the $(2,\nu,\infty)$ orbifold of genus zero. This is treated in Section~\ref{hecke_factor_bs_subsec}. The case $\nu=3$ was extensively studied within the correspondence framework by Bullett and his collaborators starting with \cite{BP94} and culminating in \cite{BL20a,BL20b,BL22}. The case $\nu=4$ was examined
in \cite{BF05}. Theorem~\ref{conf_mating_intro_thm} in combination with Theorem~\ref{corr_mating_intro_thm} unifies and generalizes these examples
to arbitrary $\nu\geq 3$.  A set of 
necessary conditions of a completely different flavor for general $\nu \geq 3$ was given in \cite{B00,BH00}.
 \end{enumerate}

\subsection*{Rational uniformization of conformal matings}
The next result (see Corollary~\ref{main_hyp_comp_mating_class_cor}), which plays the role of a bridge between conformal matings and algebraic correspondences, answers the first part of Question~\ref{qn_main}. The existence of the rational function $R$ in the proposition below is established via a new application of the relationship of anti-holomorphic maps with quadrature domains \cite{LLMM1}. A key fact that leads to this rational uniformization is a crucial property of the factor Bowen-Series map: $A_{\widetilde{\Sigma}}^{\mathrm{fBS}}$ acts via an involution on the  ideal polygon boundary of its domain of definition. See Lemmas~\ref{qd_lem} and \ref{real_sym_schwarz_lem} where this is exploited and made explicit. 

\begin{proposition}[Rational uniformization of conformal matings]\label{rat_unif_conf_mating_intro_prop}
Let $\Sigma\in\cS$, let $P$ be a complex polynomial in the principal hyperbolic component $\mathcal{H}_d$, let $F:\overline{\Omega}\to\widehat{\C}$ be the conformal mating of $A_{\widetilde{\Sigma}}^{\mathrm{fBS}}$ and $P$, and let  $\eta(z)=1/z$.
Then, there exist
\begin{itemize}
\item a Jordan domain $\mathfrak{D}$ with $\eta(\partial\mathfrak{D})=\partial\mathfrak{D}$, and

\item a degree $d+1$ rational map $R$ of $\widehat{\C}$ that maps $\overline{\mathfrak{D}}$ homeomorphically onto $\overline{\Omega}$,
\end{itemize}
such that $F\equiv R\circ\eta\circ (R\vert_{\overline{\mathfrak{D}}})^{-1}$. In particular, we have
\begin{equation}\label{semiconj_eqn}
	F\circ R = R \circ \eta.
\end{equation}
\end{proposition}

The construction of the rational uniformizing map $R$ above is detailed in Section~\ref{conf_mating_char_sec}, especially Sections~\ref{real_sym_mating_subsec} and~\ref{qc_real_sym_mating_subsec}.

\subsection*{From conformal matings to algebraic correspondences}
Thanks to the algebraic description of the conformal mating given in Relation~\eqref{semiconj_eqn} above, one can pull back such a conformal mating by the branches of $R^{-1}$ to obtain an algebraic correspondence $\mathfrak{C}$ on the Riemann sphere $\widehat{\C}$.
The next main theorem of the paper (see Theorem~\ref{corr_main_thm}) gives a positive answer to the second part of Question~\ref{qn_main}. Let  $d = d(\Sigma)$ be as above.

\begin{thmx}[Mating genus zero orbifolds with polynomials as correspondences]\label{corr_mating_intro_thm}
\noindent\item	Let $\Sigma\in\cS$, and $P$ be a complex polynomial in the principal hyperbolic component $\mathcal{H}_d$ of degree $d$ polynomials.   	
	Then, there exists a bi-degree $d$:$d$ algebraic correspondence $\mathfrak{C}$ on the Riemann sphere $\widehat{\C}$ defined by the equation
\begin{equation}
\frac{R(w)-R(1/z)}{w-1/z}=0,
\label{corr_eqn_intro}
\end{equation} 
and a $\mathfrak{C}-$invariant partition $\widehat{\C}=\widetilde{\cT}\sqcup\widetilde{\cK}$ such that the following hold.
	\begin{enumerate}
		\item On $\widetilde{\cT}$, the dynamics of $\mathfrak{C}$ is orbit-equivalent to the action of a  group of conformal automorphisms acting properly discontinuously. Further, $\faktor{\widetilde{\cT}}{\mathfrak{C}}$ is biholomorphic to $\Sigma.$
		\item $\widetilde{\cK}$ can be written as the union of two copies $\widetilde{\cK}_1, \widetilde{\cK}_2$ of $\cK(P)$ (where $\cK(P)$ is the filled Julia set of $P$), such that $\widetilde{\cK}_1$ and $\widetilde{\cK}_2$ intersect in finitely many points. Furthermore, $\mathfrak{C}$ has a forward (respectively, backward) branch carrying $\widetilde{\cK}_1$ (respectively, $\widetilde{\cK}_2$) onto itself with degree $d$, and this branch is conformally conjugate to $P:\mathcal{K}(P)\to \mathcal{K}(P)$. 
		\end{enumerate}
\end{thmx}

We remark that the Relation~\eqref{semiconj_eqn} connects two dynamical planes: one corresponding to the
\emph{conformal mating} or $F-$plane, and one corresponding to the
\emph{correspondence} or $\mathfrak{C}-$plane. The rational map $R$ mediates the connection between these two planes. This is elaborated upon in Section \ref{corr_sec}.

The following diagram summarizes the discussion above in terms of interconnections among the objects that are mated, the resulting conformal matings, and the associated correspondences.

\begin{figure}[H]
	\begin{Large}
		\begin{tikzcd}[column sep=large,row sep=large]
			& \mathrm{Teich}(\Sigma)\times\mathcal{H}_{d} \arrow{dl}[swap]{\substack{\textrm{Combination via}\\ \textrm{orbit equivalence}}}  & \\
			\substack{\cM\\ \\ \textrm{Moduli space of}\\ \textrm{conformal matings}}  \arrow{rr}[swap]{\textrm{Uniformization}\ +\ \textrm{Pullback by rational map}} & & \substack{\cC\\ \\ \textrm{Moduli space of}\\ \textrm{correspondences}} \arrow{ul}[swap]{\substack{\textrm{Conformal class of dynamics}\\ \textrm{ on invariant subsets}}}
		\end{tikzcd}
	\end{Large}
\caption{Flow-chart of interconnections}
\label{intro_fig}
\end{figure}
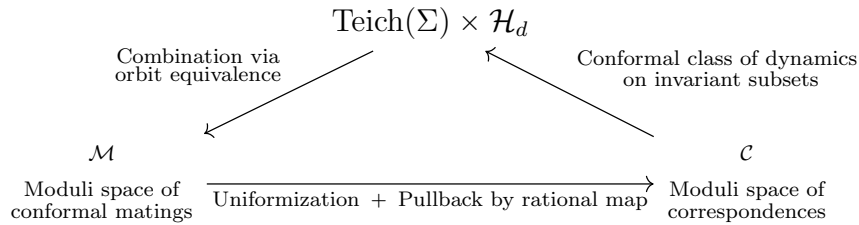

Theorem~\ref{corr_mating_intro_thm}  establishes an exact translation  between the Bullett-Penrose-Lomonaco correspondence framework 
\cite{BP94,B00,BL20a,BL20b,BL22} and the conformal matings framework of   \cite{MM1}. In particular, we obtain a different way of constructing the Bullett-Penrose-Lomonaco correspondences, starting from conformal matings (See Section \ref{bp_corr_sec} for details). The matings framework is complex analytic in nature, as opposed 
to the more algebraic flavor of the correspondence framework. The analytic setup has greater 
flexibility, giving new examples of correspondences that  combine Fuchsian punctured sphere groups (possibly with some elliptic elements) and polynomials.

\subsection*{Correspondences as character variety and a new Bers slice}

We turn now to the final theme of this paper. The existence of the rational map
$R$ allows us to look at the  space of matings algebraically parametrized by the coefficients of $R$. Extending the Sullivan dictionary to the present setup, we have the following:

\begin{itemize}
	\item The algebraic equation~\eqref{corr_eqn_intro} shows that correspondences
	are parametrized by the quasi-projective variety $\mathrm{Rat}_{d+1} (\C)$, the space of rational maps $R$ of degree exactly equal to
	$(d+1)$. 
	In the context of correspondences, $\mathrm{Rat}_{d+1} (\C)$ plays the role of the \emph{representation variety}. The quotient  $\ \faktor{\mathrm{Rat}_{d+1}(\C)}{\sim}$ by the equivalence relation
	$$
	R\sim M_2\circ R\circ M_1,
	$$
	where $R\in \mathrm{Rat}_{d+1}(\C), M_2\in\pslc$, and $M_1$ belongs to the centralizer of $\eta(z)=1/z$ in $\pslc$, plays the role of the \emph{character variety} (see Section~\ref{char_var_sec}).

	\item There is a complex-analytic realization of the Teichm{\"u}ller space of punctured spheres (more generally, genus zero orbifolds  as in Theorem \ref{corr_mating_intro_thm}) within the space  $\ \faktor{\mathrm{Rat}_{d+1}(\C)}{\sim}\ $.  This gives the analog of a \emph{Bers slice} (see Section~\ref{punc_sphere_bers_sec}).
\end{itemize}

Theorem~\ref{punc_sphere_teich_corr_intro_thm} below makes this  precise:

\begin{thmx}[Bers slices of genus zero orbifolds in spaces of correspondences]\label{punc_sphere_teich_corr_intro_thm}
	Let $\Sigma_0\in\cS$ and $d:=d(\Sigma_0)$. Then, the Teichm{\"u}ller space $\mathrm{Teich}(\Sigma_0)$ can be biholomorphically embedded in a space of bi-degree $d$:$d$ algebraic correspondences on $\widehat{\C}$ such that each resulting correspondence is a mating of some $\Sigma\in\mathrm{Teich}(\Sigma_0)$ and the polynomial $z^d$ (in the sense of Theorem~\ref{corr_mating_intro_thm}).
\end{thmx}

\subsection*{Further implications of unification of the mating frameworks}
As mentioned above, the explicit nature of factor Bowen-Series maps and the complex-analytic construction of correspondences via conformal matings facilitates the construction of such objects. In a recent work of the second author with Shaun Bullett, Luna Lomonaco, and Mikhail Lyubich, this strategy was employed to construct correspondences as matings of all parabolic rational maps and Hecke surfaces settling a parabolic version of a conjecture of Bullett and Freiberger (see \cite{BLLM}, \cite[\S 3, p. 3926]{BF03}).

The unification of the two mating frameworks also allows one to study parameter spaces of correspondences in terms of parameter spaces of conformal matings of polynomials and factor Bowen-Series maps. Often, these conformal matings behave like \emph{pinched polynomial-like maps}. In a joint work of the second author with Mikhail Lyubich and Yusheng Luo \cite{LLM24}, such spaces of pinched polynomial-like maps are investigated using puzzle techniques from one variable complex dynamics. This study reveals intimate topological relations between various spaces of correspondences and connectedness loci of complex polynomials.
\smallskip

\noindent\textbf{Notation and convention.} For the convenience of the reader, we set forth some basic notation that will be used throughout.
\begin{itemize}
	\item $\eta(z):= 1/z$, $\eta^{-}(z)=1/\overline{z}$, $\iota(z)=\overline{z}$. 
	\item The topological closure of a set $X\subset\widehat{\C}$ is denoted by $\overline{X}$ or $\Cl{X}$. 
	\item $\D^*:= \widehat{\C}\setminus\overline{\D}$.
	\item $m_d:\R/\Z\to\R/\Z,\ \theta\mapsto d\theta$.
	\item For a meromorphic map $f:U\to\widehat{\C}$, the set of critical points of $f$ is denoted by $\mathrm{crit}(f)$.
	\end{itemize}

\noindent Unless stated otherwise, we will draw non-escaping sets of the maps in consideration as bounded subsets of the plane.
\smallskip

\noindent {\bf Acknowledgments:} 
The authors thank the anonymous referee of the present paper for a diligent and careful reading, and for his/her numerous valuable suggestions that have improved the exposition. We are grateful to Tien-Cuong Dinh for posing to us the first part of Question~\ref{qn_main}. We thank an anonymous referee of \cite{sullivan_survey} for posing the second part of Question~\ref{qn_main}. We also thank Yusheng Luo for helpful conversations. 
This research was supported in part by the International Centre for Theoretical Sciences (ICTS) during the course of the program - ICTS Probabilistic Methods in Negative Curvature (code: ICTS/pmnc--2023/02).

\section{Factor Bowen-Series maps}\label{factor_bs_sec}

We will now study Bowen-Series maps associated with appropriate cyclic covers of genus zero orbifolds. These maps have mild discontinuities. However, one can pass to factors of these Bowen-Series maps such that the factors are continuous. The construction of factor Bowen-Series maps is the first key step in the proofs of our main theorems.

\subsection{Factor Bowen-Series map for a base group}\label{factor_bs_base_subsec}

Let $n, p$ be two positive integers with $np\geq 3$. For $r\in\{1,\cdots,n\}$, denote the counter-clockwise arc $\arc{e^{\frac{2i\pi(r-1)}{n}}, e^{\frac{2i\pi r}{n}}}\subset\mathbb{S}^1$ by $J_r$. Note that $J_1$ is the counter-clockwise arc of $\mathbb{S}^1$ connecting $1$ to $e^{\frac{2i\pi}{n}}$, and the various $J_r$ are obtained by rotating $J_1$ successively by angle $\frac{2\pi}{n}$ about the origin. We set $\omega:=e^{\frac{2i\pi}{n}}$, and $M_\omega:\overline{\D}\to\overline{\D}, z\mapsto \omega z$.  

Further, for $r\in\{1,\cdots,n\}$, consider the chain of $p$ bi-infinite hyperbolic geodesics 
$$
C_{r,s}:=\overline{e^{\frac{2i\pi(r-1)}{n}+\frac{2i\pi(s-1)}{np}}, e^{\frac{2i\pi (r-1)}{n}+\frac{2i\pi s}{np}}},\quad s\in\{1,\cdots,p\}.
$$
For any $r\in\{1,\cdots,n\}$, the geodesic $C_{r,1}$ has its endpoints at $e^{\frac{2i\pi(r-1)}{n}}$ and $e^{\frac{2i\pi (r-1)}{n}+\frac{2i\pi}{np}}$, and the other $C_{r,s}$ are obtained by rotating $C_{r,1}$ successively by angle $\frac{2\pi}{np}$ about the origin (see Figure~\ref{factor_bs_fig}). The geodesics $C_{r,s}$ induce a partition of the arc $J_r$ into $p$ arcs $J_{r,1},\cdots, J_{r,p}$, where $J_{r,s}$ is the arc of $\mathbb{S}^1$ of length $\frac{2\pi}{np}$ connecting the endpoints of $C_{r,s}$.

The bi-infinite geodesics $C_{r,s}$, $r\in\{1,\cdots,n\},\ s\in\{1,\cdots,p\}$, bound a closed ideal $np-$gon (in the topology of $\D$), which we call $\pmb{\Pi}$. We will now introduce M{\"o}bius maps of the disk that pair the sides of $\pmb{\Pi}$. To do so, we will exploit the symmetry $M_\omega$ of $\pmb{\Pi}$. Specifically, we will prescribe the side-pairings for $C_{1,1},\cdots, C_{1,p}$ explicitly, and conjugate these side-pairing transformations by powers of $M_\omega$ to define pairings for the other sides of $\pmb{\Pi}$. Let us denote the diameter of $\mathbb{S}^1$ with endpoints at $\pm e^{\frac{i\pi}{n}}$ by $\ell$. Now observe that the M{\"o}bius map $g_{1,s}$ obtained by post-composing the reflection in $C_{1,s}$ with the reflection in $\ell$ carries $C_{1,s}$ to $C_{1,p+1-s}$. In particular, $g_{1,p+1-s}=g_{1,s}^{-1}$. Note that when $p$ is odd, then $g_{1,\frac{p+1}{2}}$ is an involution with a fixed point on $C_{1,\frac{p+1}{2}}$.

By the Poincar{\'e} polygon theorem, the M{\"o}bius maps 
$$
g_{r,s}:=M_\omega^{r-1}\circ g_{1,s}\circ M_{\omega}^{-(r-1)},\quad r\in\{1,\cdots,n\},\ s\in\{1,\cdots,p\}
$$ 
generate a Fuchsian group $\pmb{\Gamma_{n,p}}$, and $\pmb{\Pi}$ is a closed fundamental domain for the $\pmb{\Gamma_{n,p}}-$action on $\D$. Moreover, $\faktor{\D}{\pmb{\Gamma_{n,p}}}$ is biholomorphic to
\smallskip

\noindent$\bullet$ a sphere with $\frac{np}{2}+1$ punctures for $p$ even, and
\smallskip

\noindent$\bullet$ a sphere with $\frac{n(p-1)}{2}+1$ punctures and $n$ order two orbifold points for $p$ odd.
\smallskip

\begin{remark}\label{reconcile_rem_1}
1) The integer $p$ can be thought of as the number of `pockets' in each sector of angular width $2\pi/n$. On the other hand, the integer $n$ plays the role of $\nu$ appearing in the definition of the class $\cS$ of genus zero orbifolds (see Section~\ref{intro_sec}).

2) When $n\geq 3$, the orbifold $\faktor{\D}{\pmb{\Gamma_{n,p}}}$ is an $n-$fold cyclic cover of a base genus zero orbifold $\pmb{\Sigma}\in\cS$ with $\lfloor p/2\rfloor +1$ punctures, zero/one order two orbifold point depending on the parity of $p$, and an order $n$ orbifold point.
\end{remark}

We now look at the Bowen-Series map $A_{\pmb{\Gamma_{n,p}}}^{\mathrm{BS}}$ equipped with the above fundamental domain $\pmb{\Pi}$ and side-pairing transformations (cf. \cite{bowen_series}). By definition, the map 
$$
A_{\pmb{\Gamma_{n,p}}}^{\mathrm{BS}}:\overline{\D}\setminus\Int{\pmb{\Pi}}\longrightarrow \overline{\D}
$$ 
acts as $g_{r,s}$ on the closure of the hyperbolic half-plane enclosed by the geodesic $C_{r,s}$ and the arc $J_{r,s}$ (see Figure~\ref{factor_bs_fig}).
It is now easily checked that $A_{\pmb{\Gamma_{n,p}}}^{\mathrm{BS}}$ is continuous on $\mathbb{S}^1\setminus\sqrt[n]{1}$, and the left and right-hand limits of $A_{\pmb{\Gamma_{n,p}}}^{\mathrm{BS}}$ at the points of $\sqrt[n]{1}$ lie in the set $\sqrt[n]{1}$. We will now use this fact to pass to a factor of $A_{\pmb{\Gamma_{n,p}}}^{\mathrm{BS}}$ that is continuous everywhere.

\begin{figure}[h!]
\captionsetup{width=0.98\linewidth}
\begin{tikzpicture}
\node[anchor=south west,inner sep=0] at (0.5,0) {\includegraphics[width=0.45\linewidth]{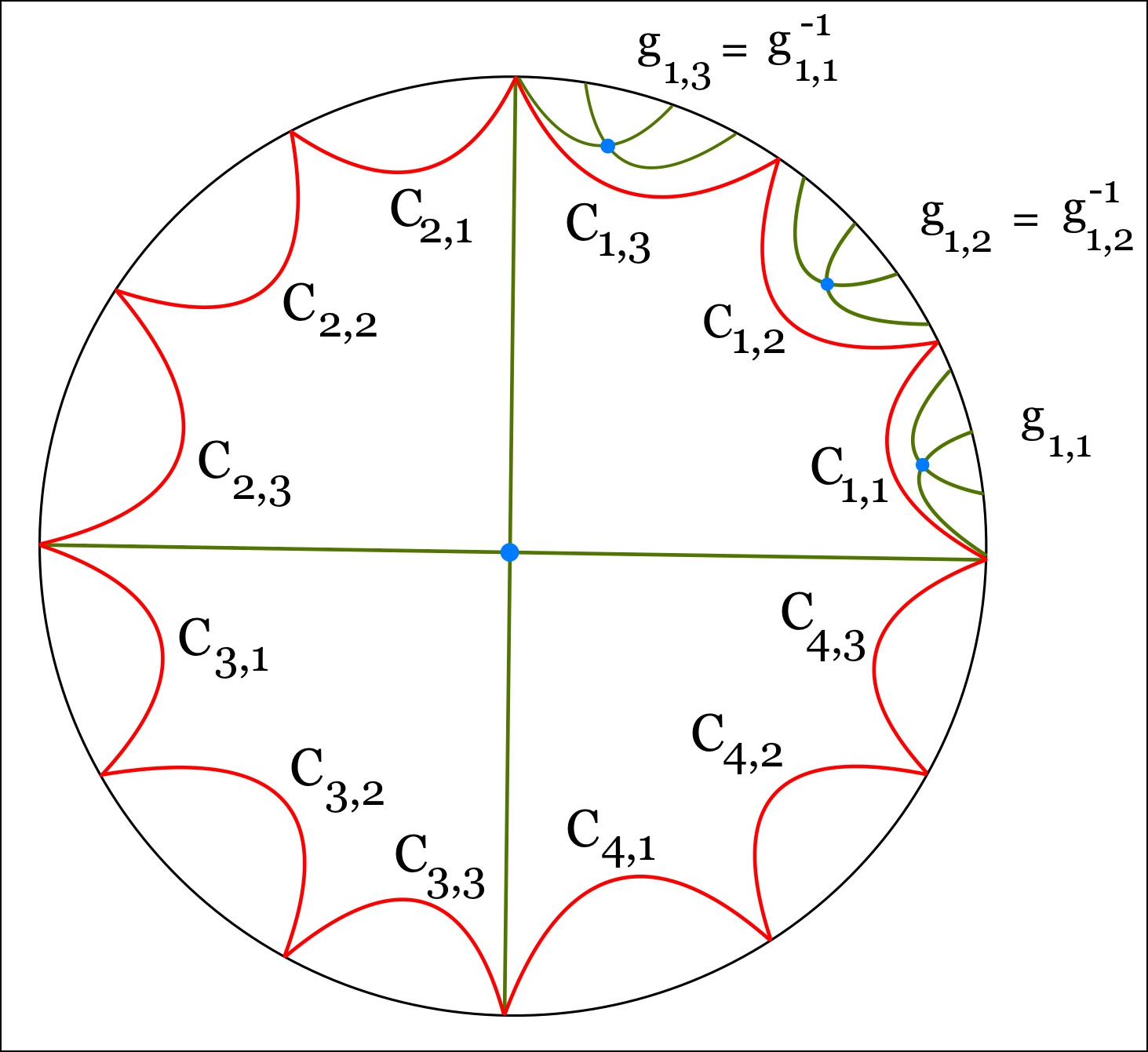}};
\node[anchor=south west,inner sep=0] at (7,0) {\includegraphics[width=0.42\linewidth]{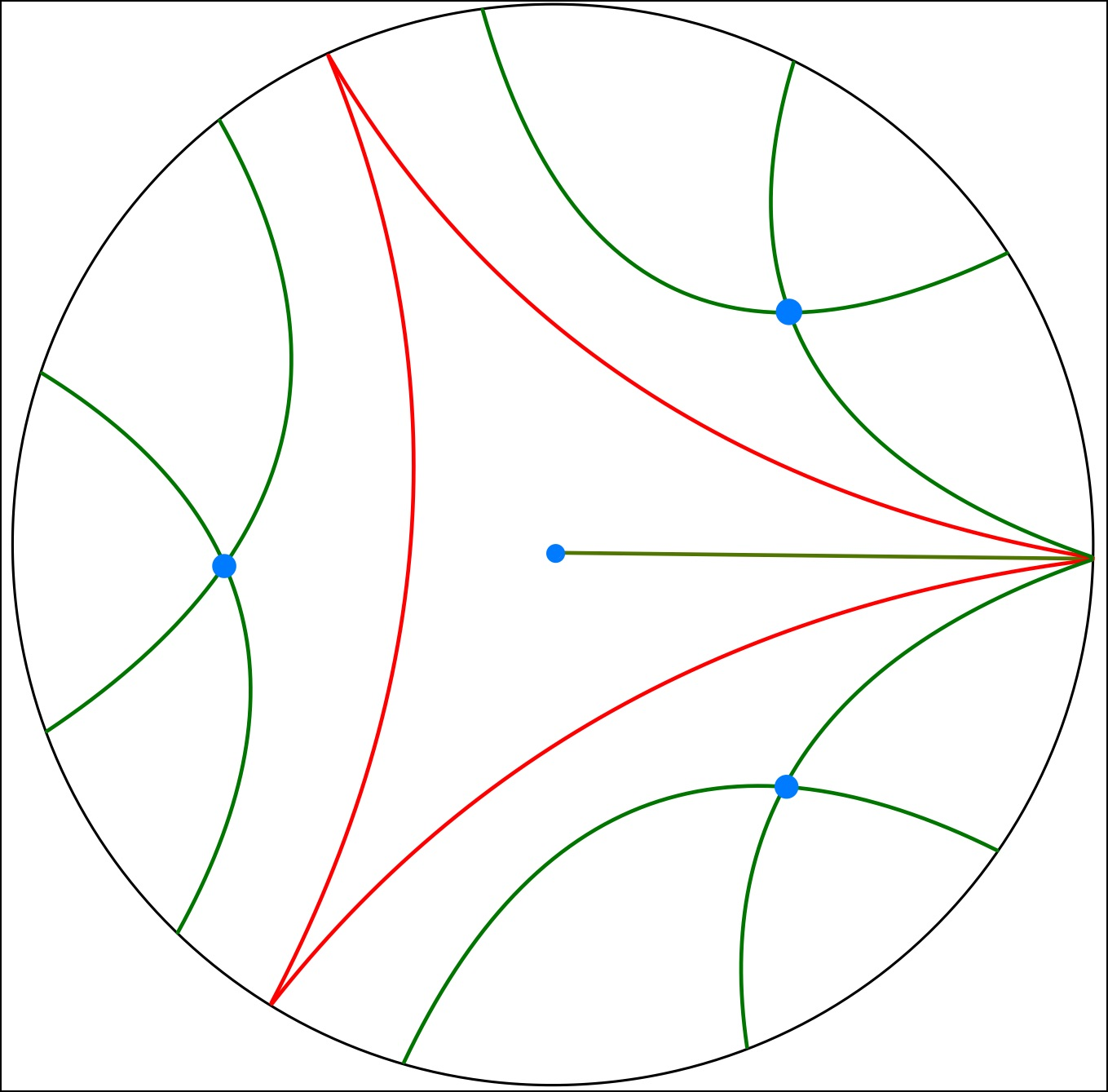}};
\end{tikzpicture}
\caption{Left: The fundamental domain $\pmb{\Pi}$ of $\pmb{\Gamma_{4,3}}$ is the polygon having the geodesics $C_{r,s}$, $r\in\{1,2,3,4\},\ s\in\{1,2,3\},$ as its edges. The Bowen-Series map $A_{\pmb{\Gamma_{4,3}}}^{\mathrm{BS}}$, which commutes with $M_i$, acts on the arcs $J_{1,1}, J_{1,2}, J_{1,3}$ as $g_{1,1}, g_{1,2}, g_{1,3}$. The map $A_{\pmb{\Gamma_{4,3}}}^{\mathrm{BS}}$ is continuous away from the fourth roots of unity. The pre-images of the vertical and horizontal radial lines under $g_{1,s}$ are displayed in green. Identifying the radial lines at angle $0, \pi/2$ under $M_i$ and uniformizing the resulting cone yields the factor Bowen-Series map $A_{\pmb{\Gamma_{4,3}}}^{\mathrm{fBS}}$. Right: The factor Bowen-Series map $A_{\pmb{\Gamma_{4,3}}}^{\mathrm{fBS}}$ is defined outside of the ideal triangle $\pmb{\Pi}/\langle M_\omega\rangle$ with vertices at the third roots of unity, and is a degree $11$ covering of $\mathbb{S}^1$. It maps all the green curves to the radial line at angle $0$, and hence has three critical points each of multiplicity three (at the valence four vertices of the green graph).}
\label{factor_bs_fig}
\end{figure}
	
Consider the bordered (orbifold) Riemann surfaces 
$$
\mathcal{Q}:=\faktor{\overline{\D}}{\langle M_\omega\rangle}\ ,\quad \mathcal{Q}_1:=\faktor{\left(\overline{\D}\setminus \Int{\pmb{\Pi}}\right)}{\langle M_\omega\rangle}\ ,
$$
and note that a closed fundamental domain for the action of $\langle M_\omega \rangle$ on $\overline{\D}$ is given by 
$$
\{z\in\overline{\D}:\ 0\leq\arg{z}\leq\frac{2\pi}{n}\}\cup\{0\}.
$$ 
Thus, $\mathcal{Q}$ is biholomorphic to the surface obtained from the above fundamental domain by identifying the radial line segments $\{r:0\leq r\leq 1\}$ and $\{re^{\frac{2\pi i}{n}}:0\leq r\leq 1\}$ by~$M_\omega$. 

By construction, the Bowen-Series map $A_{\pmb{\Gamma_{n,p}}}^{\mathrm{BS}}$ commutes with $M_\omega$, and hence it can be pushed forward via the quotient map $\mathfrak{q}:\overline{\D}\to\mathcal{Q}$ to define a map 
$$
\mathfrak{q}\circ A_{\pmb{\Gamma_{n,p}}}^{\mathrm{BS}}\circ\mathfrak{q}^{-1}: \mathcal{Q}_1\to\mathcal{Q}.
$$
Note that the map $z\mapsto z^{n}$ yields a conformal isomorphism $\xi$ between the (bordered) surfaces $\mathcal{Q}$ and $\overline{\D}$. Finally, we set 
$$
A_{\pmb{\Gamma_{n,p}}}^{\mathrm{fBS}}:= \xi\circ \left(\mathfrak{q}\circ A_{\pmb{\Gamma_{n,p}}}^{\mathrm{BS}}\circ\mathfrak{q}^{-1}\right) \circ \xi^{-1}: \mathcal{D}_{\pmb{\Gamma_{n,p}}}:=\xi(\mathcal{Q}_1)\to \overline{\D}.
$$
Note that $A_{\pmb{\Gamma_{n,p}}}^{\mathrm{fBS}}:\mathbb{S}^1\to\mathbb{S}^1$ is an orientation-preserving covering map of degree $np-1$ (see Figure~\ref{factor_bs_fig}). By \cite[Lemma~3.7]{LMMN}, the map $A_{\pmb{\Gamma_{n,p}}}^{\mathrm{fBS}}\vert_{\mathbb{S}^1}$ is expansive. Moreover, $A_{\pmb{\Gamma_{n,p}}}^{\mathrm{fBS}}$ has  $p$ critical points at $\xi(\mathfrak{q}(g_{1,s}(0)))$, $s\in\{1,\cdots, p\}$, and each of them has multiplicity $n-1$. All these critical points are mapped to $0$.

\begin{definition}\label{factor_bs_base_def}
We call the map $A_{\pmb{\Gamma_{n,p}}}^{\mathrm{fBS}}: \mathcal{D}_{\pmb{\Gamma_{n,p}}}\to \overline{\D}$ the \emph{factor Bowen-Series} map of $\pmb{\Gamma_{n,p}}$ equipped with the fundamental domain $\pmb{\Pi}$.
\end{definition}

\begin{remark}\label{rmk-Pi}
    It is worth noting that 
    $A_{\pmb{\Gamma_{n,p}}}^{\mathrm{fBS}}$ extends continuously to the boundary $\partial \, \pmb{\Pi}/\langle M_\omega\rangle$ of the ideal triangle $\pmb{\Pi}/\langle M_\omega\rangle$. Further, this action on $\partial \, \pmb{\Pi}/\langle M_\omega\rangle$ is by complex conjugation. In what follows, it will be important that the action of $A_{\pmb{\Gamma_{n,p}}}^{\mathrm{fBS}}$ on $\partial \, \pmb{\Pi}/\langle M_\omega\rangle$ is by an involution.
\end{remark}

\subsection{Deformations and moduli spaces of factor Bowen-Series maps}\label{factor_bs_gen_subsec}

The preferred fundamental domain $\pmb{\Pi}$ of $\pmb{\Gamma_{n,p}}$ can be used to equip every marked Fuchsian group in the Teichm{\"u}ller space $\mathrm{Teich}(\pmb{\Gamma_{n,p}})$ of $\pmb{\Gamma_{n,p}}$ with a preferred fundamental domain. This, in turn, will allow us to define factor Bowen-Series maps for all marked groups in a suitable symmetry locus of $\mathrm{Teich}(\pmb{\Gamma_{n,p}})$.

Recall that the Teichm{\"u}ller space $\mathrm{Teich}(\pmb{\Gamma_{n,p}})$ is the space of M{\"o}bius conjugacy classes of discrete, faithful, strongly type-preserving representations of $\pi_1\left(\faktor{\D}{\pmb{\Gamma_{n,p}}}\right)\cong \pmb{\Gamma_{n,p}}$ into $\mathrm{Aut}(\D)\cong \mathrm{PSL}_2(\R)$. Any such representation $\rho: \pmb{\Gamma_{n,p}}\longrightarrow\Gamma$ is given by $\rho(g)=\psi_\rho\circ g\circ\psi_\rho^{-1},\ g\in\pmb{\Gamma_{n,p}}$, where $\psi_\rho$ is a quasiconformal homeomorphism of $\widehat{\C}$ that preserves $\D$. 
The quasiconformal homeomorphism is not unique. However,
two such quasiconformal homeomorphisms inducing the same representation $\rho$ agree on $\mathbb{S}^1$. We can and will choose a quasiconformal homeomorphism $\psi_\rho$ such that $\psi_\rho(1)=1$, and $\psi_\rho(\pmb{\Pi})$ is a hyperbolic ideal polygon. 
The ideal polygon $\psi_\rho(\pmb{\Pi})$ is a preferred fundamental domain for the marked group $\Gamma$.

We denote by $\mathrm{Teich}^\omega(\pmb{\Gamma_{n,p}})$ the collection of $(\rho: \pmb{\Gamma_{n,p}}\longrightarrow\Gamma)\in\mathrm{Teich}(\pmb{\Gamma_{n,p}})$ that commute with conjugation by $M_\omega$; i.e., 
$$
\rho(M_\omega\circ g\circ M_\omega^{-1})=M_\omega\circ \rho(g)\circ M_\omega^{-1},\ g\in\pmb{\Gamma_{n,p}}.
$$ 
This is equivalent to requiring that the associated quasiconformal map $\psi_\rho$ commutes with $M_\omega$.  

For each $(\rho: \pmb{\Gamma_{n,p}}\longrightarrow\Gamma)\in \mathrm{Teich}^\omega(\pmb{\Gamma_{n,p}})$, consider the Bowen-Series map 
$$
A_{\Gamma}^{\mathrm{BS}} \equiv\psi_\rho\circ A_{\pmb{\Gamma_{n,p}}}^{\mathrm{BS}}\circ\psi_\rho^{-1}:\overline{\D}\setminus \Int{\psi_\rho(\pmb{\Pi})}\to\overline{\D}
$$
associated with the marked group $\Gamma$ equipped with the preferred fundamental domain $\psi_\rho(\pmb{\Pi})$. By definition, $A_\Gamma^{\mathrm{BS}}$ commutes with $M_\omega$, and thus can be pushed forward via the quotient map $\mathfrak{q}:\overline{\D}\to\mathcal{Q}$. As in the previous section, this gives rise to a map
$$
A_{\Gamma}^{\mathrm{fBS}}:= \xi\circ \left(\mathfrak{q}\circ A_{\Gamma}^{\mathrm{BS}} \circ\mathfrak{q}^{-1}\right) \circ \xi^{-1}: \mathcal{D}_{\Gamma}:= \xi(\mathfrak{q}(\overline{\D}\setminus \Int{\psi_\rho(\pmb{\Pi})}))\to \overline{\D},
$$
that is quasiconformally conjugate to $A_{\pmb{\Gamma_{n,p}}}^{\mathrm{fBS}}$.

\begin{definition}\label{factor_bs_gen_def}
For $(\rho: \pmb{\Gamma_{n,p}}\longrightarrow\Gamma)\in \mathrm{Teich}^\omega(\pmb{\Gamma_{n,p}})$ induced by the quasiconformal map $\psi_\rho$, the map $A_{\Gamma}^{\mathrm{fBS}}: \mathcal{D}_{\Gamma}\to \overline{\D}$ is called the \emph{factor Bowen-Series} map of $\Gamma$ equipped with the fundamental domain $\psi_\rho(\pmb{\Pi})$.
\end{definition}

\begin{remark}
For the symmetric base group $\pmb{\Gamma_{n,p}}$, the generators chosen in Section~\ref{factor_bs_base_subsec} are compositions of two anti-M{\"o}bius reflections. However, the (quasiconformally deformed) preferred generators for $(\rho: \pmb{\Gamma_{n,p}}\longrightarrow\Gamma)\in \mathrm{Teich}^\omega(\pmb{\Gamma_{n,p}})$ in general do not admit such symmetries.
\end{remark}
 
The group $\widehat{\pmb{\Gamma}}_{n,p}$ generated by $\pmb{\Gamma_{n,p}}$ and $M_\omega$ is an index $n$ extension of $\pmb{\Gamma_{n,p}}$. Clearly, the set
$$
\widehat{\pmb{\Pi}}:=\{ z\in\pmb{\Pi}:\ 0\leq\arg{z}\leq\frac{2\pi}{n}\}\cup\{0\}
$$
is a closed fundamental domain for the action of $\widehat{\pmb{\Gamma}}_{n,p}$ on $\D$. It follows that $\faktor{\D}{\pmb{\widehat{\Gamma}}_{n,p}}$ is biholomorphic to
\smallskip

\noindent$\bullet$ a sphere with $\frac{p}{2}+1$ punctures and an order $n$ orbifold point for $p$ even, and
\smallskip

\noindent$\bullet$ a sphere with $\frac{p+1}{2}$ punctures, an order two orbifold point and an order $n$ orbifold point for $p$ odd.
\smallskip

The space of factor Bowen-Series maps constructed above is parametrized by $\mathrm{Teich}^\omega(\pmb{\Gamma_{n,p}})$, which in turn can be identified with the Teichm{\"u}ller space $\mathrm{Teich}(\widehat{\pmb{\Gamma}}_{n,p})$. Specifically, for the representation $(\rho: \pmb{\Gamma_{n,p}}\longrightarrow\Gamma)\in \mathrm{Teich}^\omega(\pmb{\Gamma_{n,p}})$, we define 
$$
\widehat{\Gamma}:=\langle\Gamma, M_\omega\rangle,
$$ 
and associate with $\rho$ the representation
$$
(\widehat{\rho}: \widehat{\pmb{\Gamma}}_{n,p}\longrightarrow\widehat{\Gamma})\in \mathrm{Teich}(\widehat{\pmb{\Gamma}}_{n,p}),\quad \textrm{where}\quad \widehat{\rho}\vert_{\pmb{\Gamma_{n,p}}}\equiv\rho\vert_{\pmb{\Gamma_{n,p}}}\quad \textrm{and}\quad  \widehat{\rho}(M_\omega)=M_\omega.
$$ 
Thus,   $\mathrm{Teich}^\omega(\pmb{\Gamma_{n,p}})$ is the same as the Teichm\"uller space of the orbifold group
	$\widehat{\pmb{\Gamma}}_{n,p}=\pmb{\Gamma_{n,p}} \rtimes \langle M_\omega\rangle$.

\begin{remark}\label{reconcile_rem_2}
1) The orbifold $\faktor{\D}{\widehat{\pmb{\Gamma}}_{n,p}}$ is a base genus zero orbifold $\pmb{\Sigma}\in\cS$ with $\lfloor p/2\rfloor +1$ punctures, zero/one order two orbifold point depending on the parity of $p$, and an order $n$ orbifold point when $n\geq 3$ (cf. Remark~\ref{reconcile_rem_1}). Thus, any $\Sigma\in\mathrm{Teich}(\pmb{\Sigma})$ is uniformized by some Fuchsian group $\Gamma\in\mathrm{Teich}(\widehat{\pmb{\Gamma}}_{n,p})$.
\medskip

2) The fact that the chosen fundamental domain and side-pairings of the base group $\pmb{\Gamma_{n,p}}$ admit a $2\pi/n$ rotation symmetry and that the representations $(\rho: \pmb{\Gamma_{n,p}}\longrightarrow\Gamma)\in\mathrm{Teich}^\omega(\pmb{\Gamma_{n,p}})$ respect this symmetry, together imply that the orbifolds $\faktor{\D}{\Gamma}$ have an order $n$ isometry. Quotienting by this isometry yields an $n-$fold (branched) covering $\faktor{\D}{\Gamma}\longrightarrow\faktor{\D}{\widehat{\Gamma}}$.
\end{remark}

\begin{figure}[H]
\[
\begin{tikzcd}
\faktor{\D}{\Gamma} \arrow{d}[swap]{\mathrm{cover}} \arrow[r,rightsquigarrow] & A_{\Gamma}^{\mathrm{BS}} \arrow{d}{\mathrm{factor}} \\
\faktor{\D}{\widehat{\Gamma}}  &  A_{\Gamma}^{\mathrm{fBS}}
\end{tikzcd}
\]
\caption{Covering orbifolds and factor Bowen-Series maps}
\label{going_up_down_2_fig}
\end{figure}

We summarize the main properties of factor Bowen-Series maps below.

\begin{proposition}\label{factor_bs_prop}
\noindent \begin{enumerate}
\item $A_{\Gamma}^{\mathrm{fBS}}:\mathbb{S}^1\to\mathbb{S}^1$ is a piecewise analytic, orientation-preserving, expansive, covering map of degree $np-1$. In particular, it is topologically conjugate to $z^{np-1}\vert_{\mathbb{S}^1}$; i.e., there exists a homeomorphism $\mathfrak{h}_\Gamma: {\mathbb{S}^1} \to {\mathbb{S}^1}$ such that $\mathfrak{h}_\Gamma (z^{np-1}) = \big( A_{\Gamma}^{\mathrm{fBS}}(\mathfrak{h}_\Gamma (z))\big)$.

\item The restriction $A_{\Gamma}^{\mathrm{fBS}}: \left(A_{\Gamma}^{\mathrm{fBS}}\right)^{-1}(\mathcal{D}_{\Gamma})\to \mathcal{D}_{\Gamma}$ is a covering map of degree $np-1$. In particular, it maps each component of $\left(A_{\Gamma}^{\mathrm{fBS}}\right)^{-1}(\Int{\mathcal{D}_{\Gamma}})$ homeomorphically onto some component of $\Int{\mathcal{D}_{\Gamma}}$.

\item The restriction $A_{\Gamma}^{\mathrm{fBS}}: \left(A_{\Gamma}^{\mathrm{fBS}}\right)^{-1}(\overline{\D}\setminus\mathcal{D}_{\Gamma})\to \overline{\D}\setminus\mathcal{D}_{\Gamma}$ has degree $np$. If $n\geq 2$, there are $p$ critical points of $A_{\Gamma}^{\mathrm{fBS}}$, each of multiplicity $n-1$, in $\left(A_{\Gamma}^{\mathrm{fBS}}\right)^{-1}(\overline{\D}\setminus\mathcal{D}_{\Gamma})$.
All these critical points are mapped to the unique critical value $0$ of $A_{\Gamma}^{\mathrm{fBS}}$.

\item The set $\overline{\D}\setminus\mathcal{D}_\Gamma$ is (the interior of) a topological $p-$gon with its $p$ vertices on~$\mathbb{S}^1$. Each such vertex is a fixed point or a $2-$periodic point under $A_{\Gamma}^{\mathrm{fBS}}$.

\item For $\Gamma=\pmb{\Gamma_{n,p}}$, the factor Bowen-Series map acts as complex conjugation on $\partial\mathcal{D}_{\pmb{\Gamma_{n,p}}}\cap\D$, which is the boundary of the ideal polygon $\overline{\D}\setminus\mathcal{D}_{\pmb{\Gamma_{n,p}}}$.

\noindent For a general $\Gamma$, the factor Bowen-Series map $A_{\Gamma}^{\mathrm{fBS}}$ acts by an involution on $\partial\mathcal{D}_{\Gamma}\cap\D$. 
\end{enumerate}
\end{proposition}
\begin{figure}[h!]
\captionsetup{width=0.98\linewidth}
\begin{tikzpicture}
\node[anchor=south west,inner sep=0] at (0.5,0) {\includegraphics[width=0.5\linewidth]{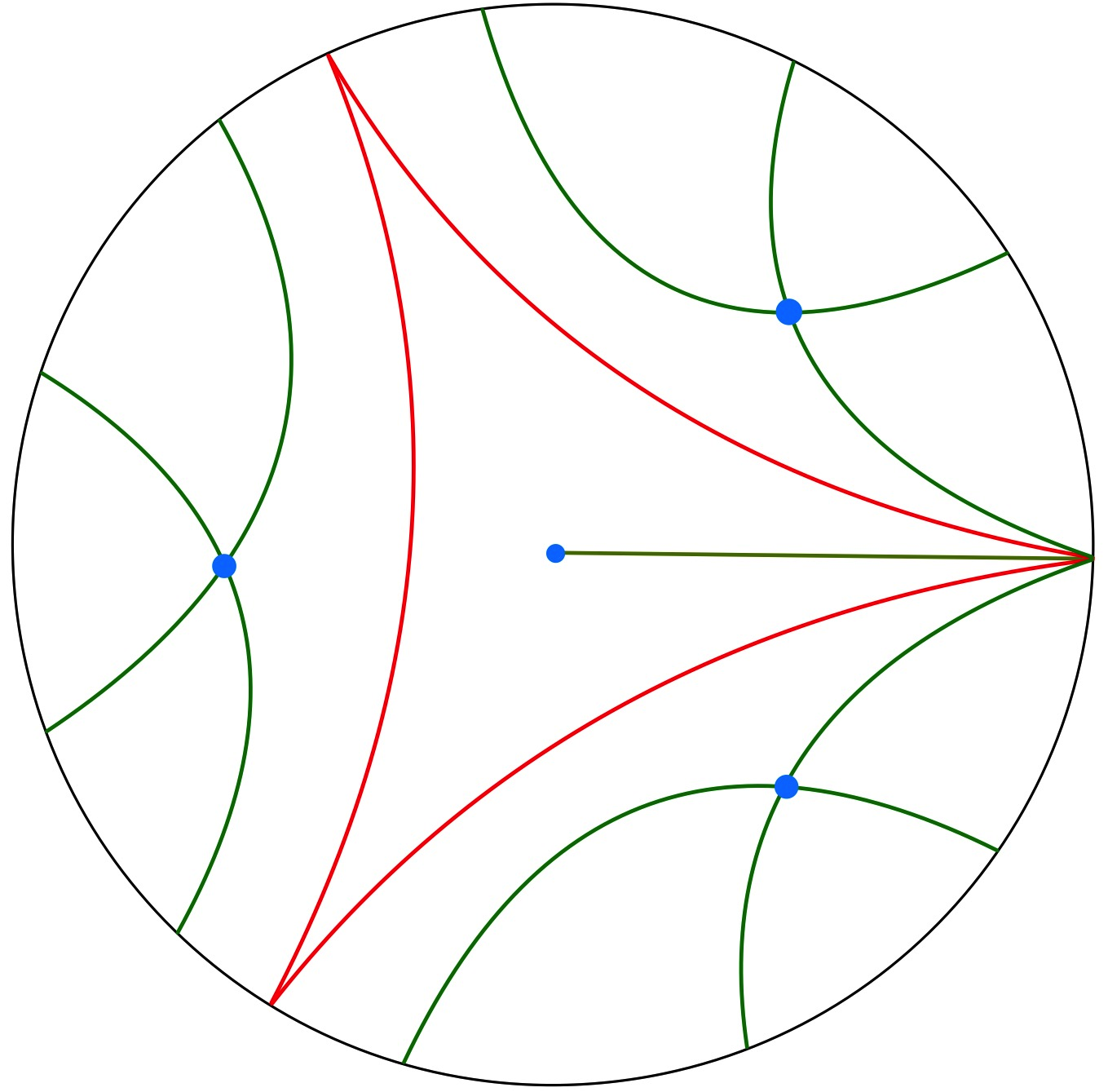}};
\node at (7.2,4) {$\mathfrak{J}_{1,1}^1$};
\node at (6.24,5.66) {$\mathfrak{J}_{1,1}^2$};
\node at (4.3,6.6) {$\mathfrak{J}_{1,1}^3$};
\node at (2.84,6.54) {$\mathfrak{J}_{1,1}^4$};
\node at (2,6.25) {$\mathfrak{J}_{1,2}^1$};
\node at (0.75,5.16) {$\mathfrak{J}_{1,2}^2$};
\node at (0.2,3.36) {$\mathfrak{J}_{1,2}^3$};
\node at (0.75,1.48) {$\mathfrak{J}_{1,2}^4$};
\node at (1.66,0.45) {$\mathfrak{J}_{1,2}^5$};
\node at (2.4,0.04) {$\mathfrak{J}_{1,3}^1$};
\node at (4,-0.28) {$\mathfrak{J}_{1,3}^2$};
\node at (5.88,0.32) {$\mathfrak{J}_{1,3}^3$};
\node at (7.08,1.88) {$\mathfrak{J}_{1,3}^4$};
\end{tikzpicture}
\caption{Displayed is the dynamical plane of the factor Bowen-Series map $A_{\pmb{\Gamma_{4,3}}}^{\mathrm{fBS}}$ and the partition of $\mathbb{S}^1$ given by the arcs $\mathfrak{J}_{1,s}^i:=\xi(\mathfrak{q}(J_{1,s}^i))$ (cf. Figure~\ref{factor_bs_fig}). Pulling this partition back by $A_{\pmb{\Gamma_{4,3}}}^{\mathrm{fBS}}$ yields a Markov partition for the map.}
\label{david_fig}
\end{figure}
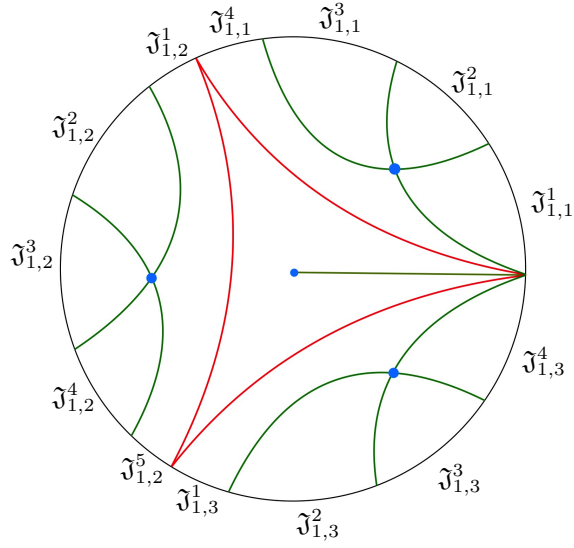

\begin{proof}
It is enough to verify the assertions for the base map $A_{\pmb{\Gamma_{n,p}}}^{\mathrm{fBS}}$.
 We endow the bordered Riemann surface $\mathcal{\mathcal{Q}}$ with a preferred choice of complex coordinates via its identification with $\{z\in\overline{\D}:\ 0\leq\arg{z}\leq\frac{2\pi}{n}\}\cup\{0\}$ (with the boundary radial lines glued together).

1) The facts that $A_{\pmb{\Gamma_{n,p}}}^{\mathrm{BS}}$ is continuous on $\mathbb{S}^1\setminus\sqrt[n]{1}$ and the left and right-hand limits of $A_{\pmb{\Gamma_{n,p}}}^{\mathrm{BS}}$ at the points of $\sqrt[n]{1}$ lie in the set $\sqrt[n]{1}$ together imply that 
$$
\overline{A}_{\pmb{\Gamma_{n,p}}}^{\mathrm{BS}}:=\mathfrak{q}\circ A_{\pmb{\Gamma_{n,p}}}^{\mathrm{BS}}\circ\mathfrak{q}^{-1}
$$ 
is continuous.

Let us partition each arc $J_{1,s}\subset\mathbb{S}^1$, $s\in\{1,\cdots,p\}$, into sub-arcs $J_{1,s}^1,\cdots, J_{1,s}^{m(s)}$, where each $J_{1,s}^i$ is a connected component of some $g_{1,s}^{-1}(J_r)$, $i\in\{1,\cdots,m(s)\}$, $r\in\{1,\cdots,n\}$.  Then, with the above choice of coordinates on $\mathcal{Q}$,
\begin{itemize}
\item $\overline{A}_{\pmb{\Gamma_{n,p}}}^{\mathrm{BS}}$ acts as a M{\"o}bius map $h_{i,s}$ (called a \emph{piece} of $\overline{A}_{\pmb{\Gamma_{n,p}}}^{\mathrm{BS}}$) on $\mathfrak{q}(J_{1,s}^i)$; specifically, $h_{i,s}$ is a composition of $g_{1,s}$ with a power of~$M_\omega$, 

\item $h_{i,s}(\mathfrak{q}(J_{1,s}^i))$ is the union of finitely many sub-arcs from the collection $\{\mathfrak{q}(J_{1,s}):s\in\{1,\cdots,p\}\}$.
\end{itemize}
\noindent (See Figure~\ref{david_fig} for the $\xi-$images of the arcs $\mathfrak{q}(J_{1,s}^i)$ in $\mathbb{S}^1$, for $n=4, p=3$.) It follows that  with the above choice of coordinates on $\mathcal{Q}$, the map $\overline{A}_{\pmb{\Gamma_{n,p}}}^{\mathrm{BS}}$ is a piecewise M{\"o}bius, orientation-preserving, covering map of $\partial\mathcal{Q}$. The statement that $\overline{A}_{\pmb{\Gamma_{n,p}}}^{\mathrm{BS}}:\partial\mathcal{Q}\to\partial\mathcal{Q}$ has degree $np-1$ follows from the fact that all but finitely many points in $\mathbb{S}^1$ have $np-1$ preimages under the map $A_{\pmb{\Gamma_{n,p}}}^{\mathrm{BS}}$ (since  $A_{\pmb{\Gamma_{n,p}}}^{\mathrm{BS}}$ maps each arc $J_{r,s}$ to $\mathbb{S}^1\setminus\Int{J_{r,p+1-s}}$). Finally, expansivity of $\overline{A}_{\pmb{\Gamma_{n,p}}}^{\mathrm{BS}}\vert_{\partial\mathcal{Q}}$ is a consequence of the fact that each $g_{1,s}$ has derivative larger than one on $\Int{J_{1,s}}$, for $s\in\{1,\cdots,p\}$ (cf. \cite[Lemma~3.8]{LMMN}).

Since $\xi:\mathcal{Q}\to\overline{\D}$ is a biholomorphism, the properties of $\overline{A}_{\pmb{\Gamma_{n,p}}}^{\mathrm{BS}}$ listed in the previous paragraph imply that the map $A_{\pmb{\Gamma_{n,p}}}^{\mathrm{fBS}}\equiv \xi\circ \overline{A}_{\pmb{\Gamma_{n,p}}}^{\mathrm{BS}}\vert_{\partial\mathcal{Q}}\circ\xi^{-1}:\mathbb{S}^1\to\mathbb{S}^1$ is a piecewise analytic (\emph{not} piecewise M{\"o}bius when $n>1$), orientation-preserving, expansive, covering map of degree $np-1$.

The existence of the circle homeomorphism $\mathfrak{h}_\Gamma$ that conjugates $z^{np-1}$ to $A_\Gamma^{\mathrm{fBS}}$ now follows from the fact that two expansive circle coverings of the same degree are topologically conjugate (cf. \cite[Property~(2'), p. 99]{CR80}).

2) and 3) The Bowen-Series map $A_{\pmb{\Gamma_{n,p}}}^{\mathrm{BS}}$ sends the hyperbolic half-plane bounded by the geodesic $C_{1,s}$ and the arc $J_{1,s}\subset\mathbb{S}^1$ conformally to the complement of the hyperbolic half-plane bounded by the geodesic $C_{1,p+1-s}$ and the arc $J_{1,p+1-s}\subset\mathbb{S}^1$, $s\in\{1,\cdots,p\}$. Hence, the region $\overline{\D}\setminus\mathcal{D}_{\Gamma}$ is covered $np$ times by $A_{\pmb{\Gamma_{n,p}}}^{\mathrm{fBS}}$, while $\mathcal{D}_{\Gamma}$ is covered $np-1$ times. Clearly, $A_{\Gamma}^{\mathrm{fBS}}: \left(A_{\Gamma}^{\mathrm{fBS}}\right)^{-1}(\mathcal{D}_{\Gamma})\to \mathcal{D}_{\Gamma}$ is piecewise conformal, and hence a covering map.

To locate the critical points of $A_{\pmb{\Gamma_{n,p}}}^{\mathrm{fBS}}$, let us denote the union of the radial lines in $\overline{\D}$ at angles $\frac{2j\pi}{n}$, $j\in\{0,\cdots,n-1\}$, by $\mathcal{P}$. Note that $A_{\pmb{\Gamma_{n,p}}}^{\mathrm{fBS}}$ maps each $\xi(\mathfrak{q}(g_{1,s}^{-1}(\mathcal{P})))$ to the line segment $[0,1]$ and sends $\xi(\mathfrak{q}(g_{1,s}^{-1}(0)))$ to $0$, for $s\in\{1,\cdots,p\}$ (see Figures~\ref{factor_bs_fig} and~\ref{david_fig}). It follows that for each $s\in\{1,\cdots,p\}$, the point $\xi(\mathfrak{q}(g_{1,s}^{-1}(0)))$ is a critical point of multiplicity $n-1$ with associated critical value $0$.

4) and 5) These follow from Definitions~\ref{factor_bs_base_def} and~\ref{factor_bs_gen_def}, and the discussion in Section~\ref{factor_bs_base_subsec}.
\end{proof}

We denote the circle homeomorphism that conjugates $z^{np-1}$ to $A_{\pmb{\Gamma_{n,p}}}^{\mathrm{fBS}}$ by $\mathfrak{h}_0$, normalized such that $\mathfrak{h}_0$ sends the fixed point $1$ of $z^{np-1}$ to the fixed point $1$ of $A_{\pmb{\Gamma_{n,p}}}^{\mathrm{fBS}}$. Further, the quasiconformal conjugacy $\psi_\rho$ between $\pmb{\Gamma_{n,p}}$ and $\Gamma$ induces a quasiconformal conjugacy $\widehat{\psi}_\rho$ between $A_{\pmb{\Gamma_{n,p}}}^{\mathrm{fBS}}$ and $A_{\Gamma}^{\mathrm{fBS}}$. 
In the remainder of the paper, we will work with the normalized conjugacy $\mathfrak{h}_\Gamma=\widehat{\psi}_\rho\circ\mathfrak{h}_0:\mathbb{S}^1\to\mathbb{S}^1$ between $z^{np-1}\vert_{\mathbb{S}^1}$ and $A_{\Gamma}^{\mathrm{fBS}}\vert_{\mathbb{S}^1}$.

Let us denote the set of $p$ ideal boundary points of $\overline{\D}\setminus \mathcal{D}_{\Gamma}$ on $\mathbb{S}^1$ by $S_{\Gamma}$.
Under the circle homeomorphism $\mathfrak{h}_\Gamma$, the set $S_{\Gamma}$ is pulled back to the set 
$$
\displaystyle\mathcal{A}_p:=\{\frac{i}{p}:i\in\{0,\cdots,p-1\}\}
$$
(here we identify $\mathbb{S}^1$ with $\R/\Z$). See Proposition~\ref{factor_bs_prop} (parts 4, 5). More precisely, if $p$ is even (respectively, odd), the two points (respectively, one point) of $S_{\Gamma}$ that are (respectively, is) fixed by $A_{\Gamma}^{\mathrm{fBS}}$ correspond to $0,\frac{p}{2p}=\frac{1}{2}$ (respectively, corresponds to $0$), and the $2-$cycles of $A_{\Gamma}^{\mathrm{fBS}}$ in $S_{\Gamma}$ correspond to the $2-$cycles $\pm \frac{i}{p}$, $i\in\{1,\cdots,\lfloor \frac{p-1}{2}\rfloor\}$, of $m_{np-1}$.

\subsection{Special case I: continuous Bowen-Series maps}\label{punc_sphere_bs_subsec}

In \cite{MM1,sullivan_survey}, Bowen-Series maps of Fuchsian punctured sphere groups (possibly with an order two orbifold point) equipped with special fundamental domains were studied. These maps, which are covering maps of $\mathbb{S}^1$, are contained in the class of maps constructed in Section~\ref{factor_bs_gen_subsec}.

\subsubsection{Bowen-Series maps of Fuchsian punctured sphere groups}\label{bs_1_subsubsec}

Let $n=1$ and $p\geq 4$ be an even integer. Then $\faktor{\D}{\pmb{\Gamma_{n,p}}}$ is a sphere with $\frac{p}{2}+1$ punctures, and for $(\rho: \pmb{\Gamma_{n,p}}\longrightarrow\Gamma)\in \mathrm{Teich}^\omega(\pmb{\Gamma_{n,p}}) \equiv \mathrm{Teich}(\pmb{\Gamma_{n,p}})$, the map $A_{\Gamma}^{\mathrm{fBS}}$ agrees with the standard Bowen-Series map $A_{\Gamma}^{\mathrm{BS}}$ of $\Gamma$ equipped with the fundamental domain $\psi_\rho(\pmb{\Pi})$ (see \cite[\S 3]{MM1} and Figure~\ref{punc_sphere_orbifold_bs_fig}(left)). This map restricts to a piece-wise analytic $C^1$, expansive, degree $p-1$ covering of $\mathbb{S}^1$. Moreover, it has no critical points in its domain of definition $\mathcal{D}_{\Gamma}$.

\begin{figure}[h!]
\captionsetup{width=0.96\linewidth}
	\begin{tikzpicture}
		\node[anchor=south west,inner sep=0] at (0.5,0) {\includegraphics[width=0.44\linewidth]{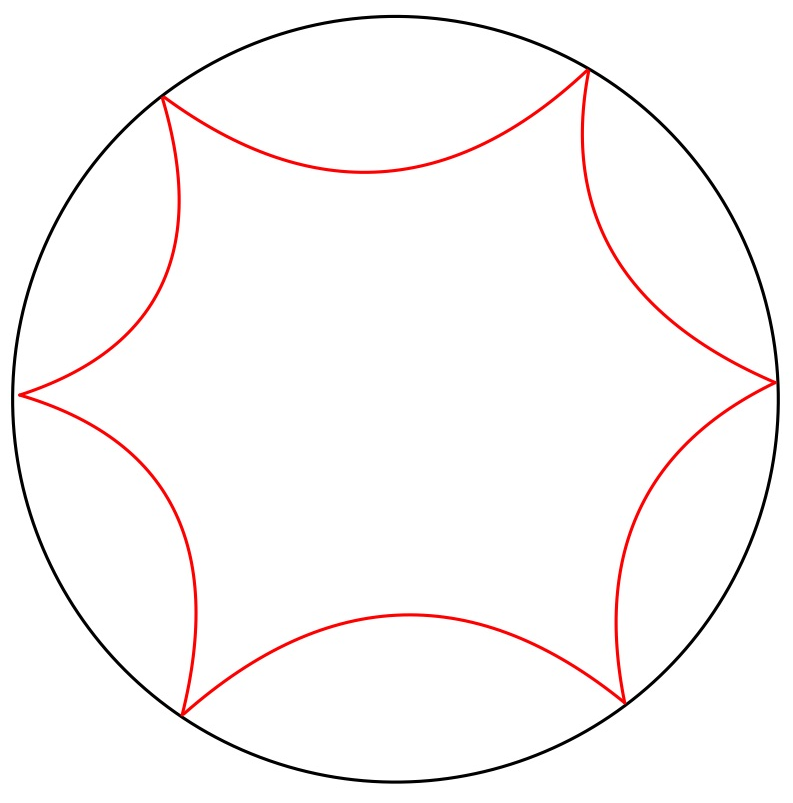}};
		\node[anchor=south west,inner sep=0] at (7,0) {\includegraphics[width=0.44\linewidth]{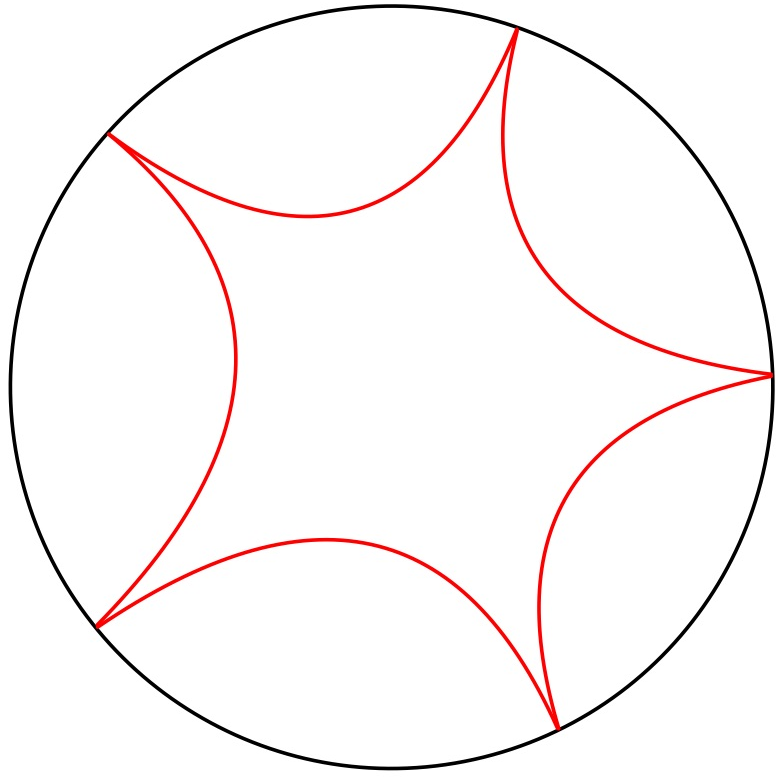}};
		\node at (3.36,2.8) {$\pmb{\Pi}(\pmb{\Gamma_{1,6}})$};
		\node at (4.6,3.66) {$C_{1,1}$};
		\node at (4.6,2) {$C_{1,6}$};
		\node at (3.5,4.2) {$C_{1,2}$};
		\node at (3.5,1.6) {$C_{1,5}$};
		\node at (2.1,3.6) {$C_{1,3}$};
		\node at (2.2,2.1) {$C_{1,4}$};
		\node at (6,4.6) {\begin{small}$g_{1,1}$\end{small}};
		\node at (6.1,1.56) {\begin{small}$g_{1,1}^{-1}$\end{small}};
		\node at (3.3,5.9) {\begin{small}$g_{1,2}$\end{small}};
		\node at (3.5,-0.2) {\begin{small}$g_{1,2}^{-1}$\end{small}};
		\node at (0.6,4.2) {\begin{small}$g_{1,3}$\end{small}};
		\node at (0.6,1.32) {\begin{small}$g_{1,3}^{-1}$\end{small}};	
			
		\node at (9.8,2.8) {$\pmb{\Pi}(\pmb{\Gamma_{1,5}})$};
		\node at (10.7,3.5) {$C_{1,1}$};
		\node at (10.75,2.25) {$C_{1,5}$};
		\node at (9.3,3.8) {$C_{1,2}$};
		\node at (9.3,2.05) {$C_{1,4}$};
		\node at (8.2,2.88) {$C_{1,3}$};
		\node at (12.5,4.5) {\begin{small}$g_{1,1}$\end{small}};
		\node at (12.5,1.28) {\begin{small}$g_{1,1}^{-1}$\end{small}};
		\node at (8.6,5.6) {\begin{small}$g_{1,2}$\end{small}};
		\node at (8.5,-0.08) {\begin{small}$g_{1,2}^{-1}$\end{small}};
		\node at (6.7,2.78) {\begin{small}$g_{1,3}$\end{small}};
		
	\end{tikzpicture}
	\caption{The preferred fundamental hexagon (respectively, pentagon) of $\pmb{\Gamma_{1,6}}$ (respectively, of $\pmb{\Gamma_{1,5}}$), which uniformizes a four times punctured sphere (respectively, a sphere with three punctures and an order two orbifold point), is shown. The action of the corresponding Bowen-Series maps on these arcs are also marked.}
	\label{punc_sphere_orbifold_bs_fig}
\end{figure}
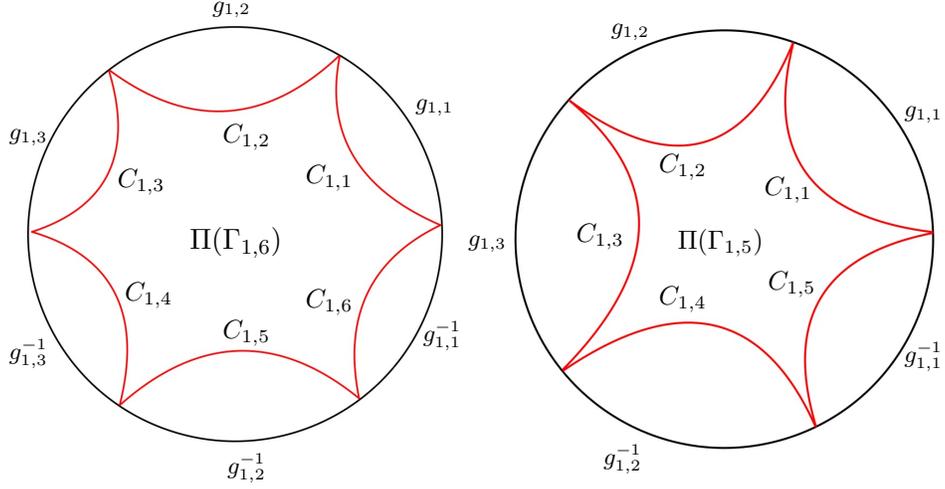

\subsubsection{Bowen-Series maps of Fuchsian groups uniformizing punctured spheres with an order two orbifold point}\label{bs_2_subsubsec}

Let $n=1$ and $p\geq 3$ be an odd integer. Then $\faktor{\D}{\pmb{\Gamma_{n,p}}}$ is a sphere with $\frac{p+1}{2}$ punctures and an order two orbifold point, and for $(\rho: \pmb{\Gamma_{n,p}}\longrightarrow\Gamma)\in \mathrm{Teich}^\omega(\pmb{\Gamma_{n,p}}) \equiv \mathrm{Teich}(\pmb{\Gamma_{n,p}})$, the map $A_{\Gamma}^{\mathrm{fBS}}$ agrees with the standard Bowen-Series map $A_{\Gamma}^{\mathrm{BS}}$ of $\Gamma$ equipped with the fundamental domain $\psi_\rho(\pmb{\Pi})$ (see \cite[\S 3]{MM1} and Figure~\ref{punc_sphere_orbifold_bs_fig}(right)). This map restricts to a $C^1$, expansive, degree $p-1$ covering of $\mathbb{S}^1$. Moreover, it has no critical points in its domain of definition $\mathcal{D}_{\Gamma}$.
\begin{figure}[h!]
\captionsetup{width=0.98\linewidth}
\begin{tikzpicture}
\node[anchor=south west,inner sep=0] at (0.5,0) {\includegraphics[width=0.44\linewidth]{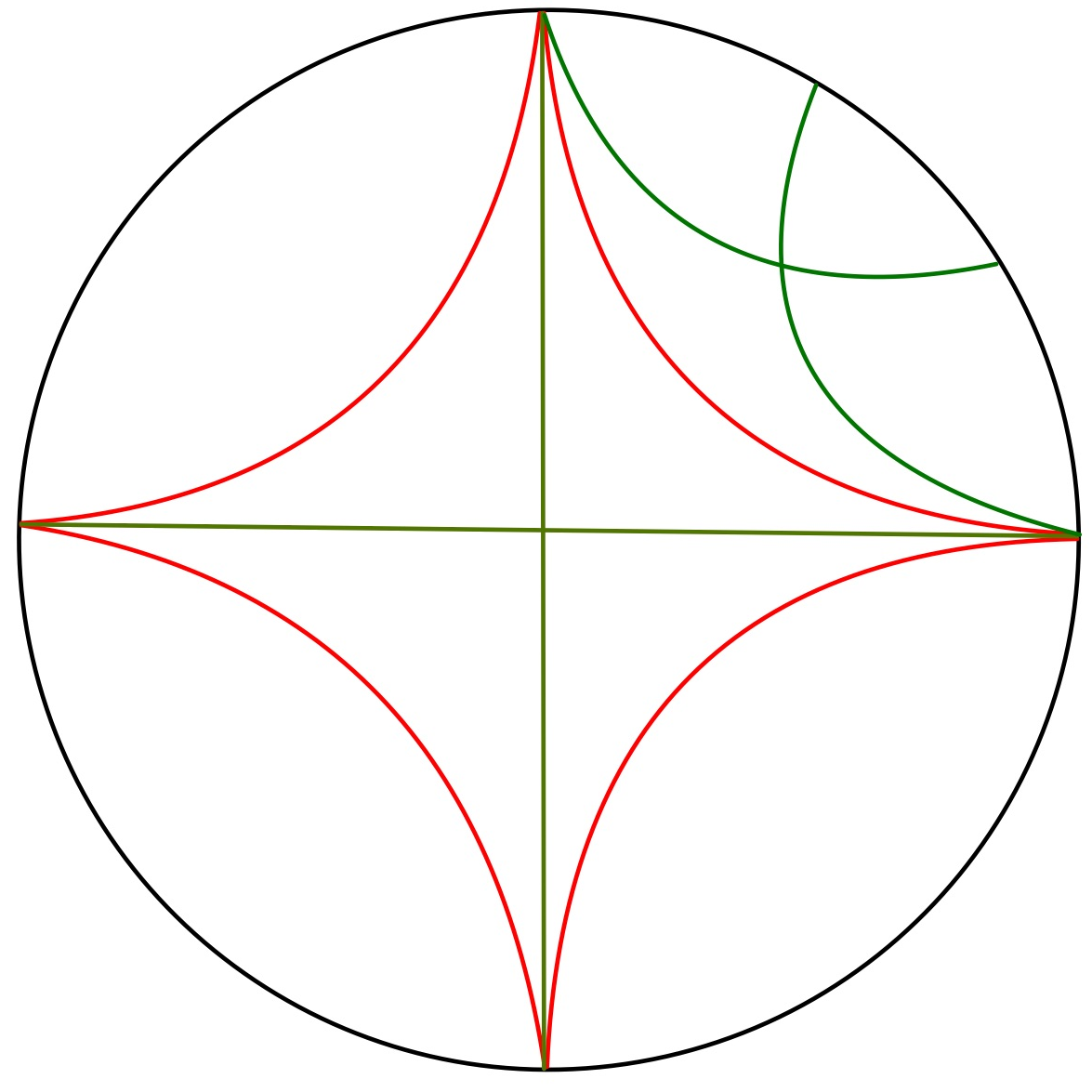}};
\node[anchor=south west,inner sep=0] at (7,0) {\includegraphics[width=0.42\linewidth]{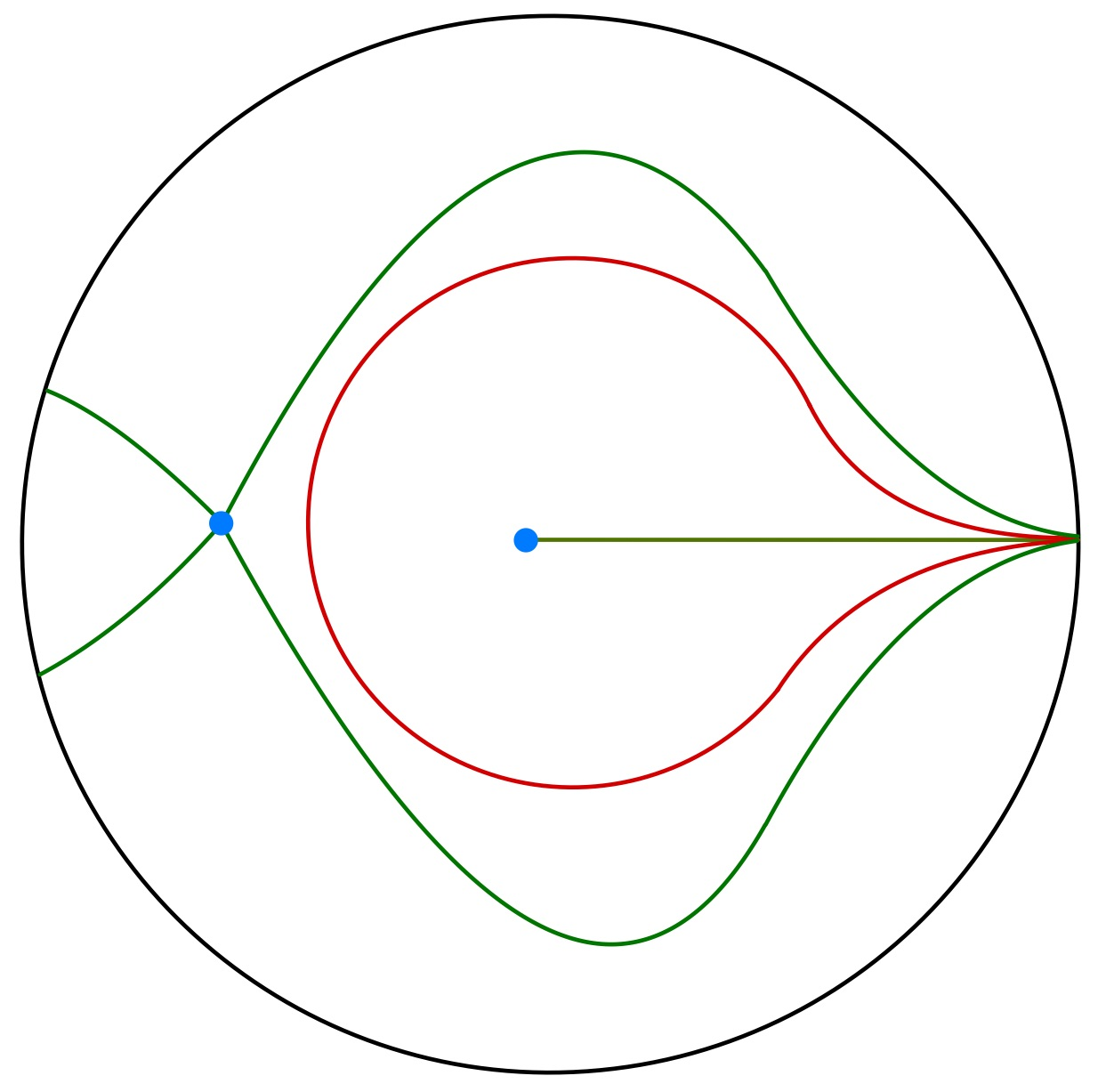}};
\node at (4.05,3.36) {\begin{small}$C_{1,1}$\end{small}};
\node at (4.48,1.88) {\begin{small}$C_{4,1}$\end{small}};
\node at (2.24,3.96) {\begin{small}$C_{2,1}$\end{small}};
\node at (2.24,1.88) {\begin{small}$C_{3,1}$\end{small}};
\node at (5.8,4.9) {\begin{small}$g_{1,1}$\end{small}};
\end{tikzpicture}
\caption{Left: The fundamental domain $\Pi$ of $\pmb{\Gamma_{4,1}}$ is the polygon having the geodesics $C_{r,1}$, $r\in\{1,2,3,4\},$ as its edges. The Bowen-Series map $A_{\pmb{\Gamma_{4,1}}}^{\mathrm{BS}}$, which commutes with $M_i$, acts as $g_{1,1}$ on the arc $J_{1,1}$. 
The pre-images of the vertical and horizontal radial lines under $g_{1,1}$ are displayed in green. Right: The factor Bowen-Series map $A_{\pmb{\Gamma_{4,1}}}^{\mathrm{fBS}}$ is defined outside of an ideal monogon with its vertex at $1$, and is a degree three covering of $\mathbb{S}^1$. The map $A_{\pmb{\Gamma_{4,1}}}^{\mathrm{fBS}}$ extends continuously to the boundary of the ideal monogon and acts on it by complex conjugation. It has a unique critical point of multiplicity three at the valence four vertex of the green graph.}
\label{fully_ramified_factor_bs_fig}
\end{figure}
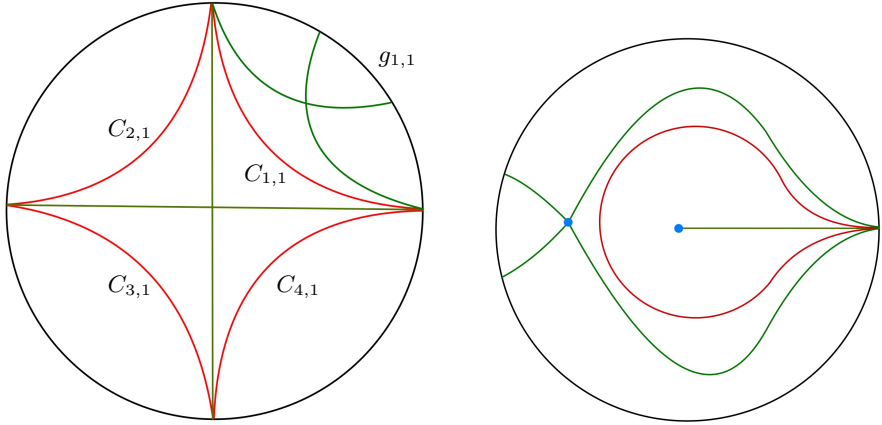

\subsection{Special case II: fully ramified factor Bowen-Series maps}\label{hecke_factor_bs_subsec}

We now look at the case when $p=1$ and $n\geq 3$ is any integer. By construction, for $r\in\{1,\cdots,n\}$, the map $g_{r,1}$ is an involution with an elliptic fixed point on $C_{r,1}$. Moreover, $\faktor{\D}{\pmb{\Gamma_{n,p}}}$ is a sphere with one puncture and $n$ order two orbifold points. The corresponding index $n$ extension $\widehat{\pmb{\Gamma}}_{n,p}$ is a classical Hecke group, which uniformizes a genus zero orbifold with exactly one puncture, exactly one order two orbifold point and exactly one order $n$ orbifold point. In particular, $\mathrm{Teich}(\widehat{\pmb{\Gamma}}_{n,p})$ is a singleton, and hence the factor Bowen-Series map associated with $\pmb{\Gamma_{n,p}}$ (equipped with the fundamental $n-$gon $\pmb{\Pi}$) is rigid.

The map $A_{\pmb{\Gamma_{n,p}}}^{\mathrm{fBS}}$ restricts to a $C^1$, expansive, degree $n-1$ covering of $\mathbb{S}^1$. Further, this map has a unique critical point of multiplicity $n-1$ (see Figure~\ref{fully_ramified_factor_bs_fig}).

\section{Conformal matings of factor Bowen-Series maps with polynomials}\label{conf_mating_sec}

The goal of this section is to prove Theorem~\ref{conf_mating_intro_thm}. In fact, we will prove a more general statement that allows the polynomials to lie in arbitrary hyperbolic components in the connectedness locus.

\subsection{The notion of conformal mating}\label{conf_mating_def_subsec}

Let $n, p$ be two positive integers with $np\geq 3$, and $P$ be a monic, centered complex polynomial of degree $d:=np-1$ with a connected and locally connected Julia set $\mathcal{J}(P)$. Let $\mathcal{K}(P)$ denote the filled Julia set.
Recall that there exists a unique conformal map 
$$
\psi_P:\widehat{\C}\setminus\overline{\D}\to\mathcal{B}_\infty(P):=\widehat{\C}\setminus\mathcal{K}(P)
$$ 
(called the \emph{B{\"o}ttcher coordinate} of $P$) that conjugates $z^{d}$ to $P$, and is tangent to the identity map near infinity (cf. \cite[Theorem~9.1]{Mil06}). As $\mathcal{J}(P)$ is locally connected, $\psi_P$ extends continuously to $\mathbb{S}^1$ to yield a semi-conjugacy between $z^d\vert_{\mathbb{S}^1}$ and $P\vert_{\mathcal{J}(P)}$. 

We now define the notion of topological/conformal mating of $P$ and $A_{\Gamma}^{\mathrm{fBS}}$, where $\Gamma\in\mathrm{Teich}^\omega(\pmb{\Gamma_{n,p}})$.
Recall from Section~\ref{factor_bs_gen_subsec} that there exists a normalized homeomorphism $\mathfrak{h}_\Gamma:\mathbb{S}^1\to\mathbb{S}^1$ that conjugates $z^d$ to $A_{\Gamma}^{\mathrm{fBS}}$.
Let us now consider the disjoint union $\mathcal{K}(P)\sqcup \overline{\D}$ and the map 
\begin{center}
	$P\sqcup A_{\Gamma}^{\mathrm{fBS}}: \mathcal{K}(P)\sqcup \mathcal{D}_{\Gamma}\to \mathcal{K}(P)\sqcup \overline{\D},$\\
	$\left(P\sqcup A_{\Gamma}^{\mathrm{fBS}}\right)\vert_{\mathcal{K}(P)}=P,\quad  \left(P\sqcup A_{\Gamma}^{\mathrm{fBS}}\right)\vert_{\mathcal{D}_{\Gamma}}=A_{\Gamma}^{\mathrm{fBS}}.$
\end{center}
Let $\sim_{\mathrm{m}}$ (here `m' stands for mating) be the equivalence relation on $\mathcal{K}(P)\bigsqcup \overline{\D}$ generated by 
\begin{equation}
\psi_{P}(z)\sim_{\mathrm{m}} \mathfrak{h}_\Gamma(\overline{z})),\ \textrm{for all}\ z\in\mathbb{S}^1.
\label{conf_mat_equiv_rel}
\end{equation} 
The map $P\sqcup A_{\Gamma}^{\mathrm{fBS}}$ descends to a continuous map $P\mate A_{\Gamma}^{\mathrm{fBS}}$ to the quotient 
$$
\mathcal{K}(P)\ \mate\ \overline{\D}\ :=\ \left(\mathcal{K}(P)\sqcup \overline{\D}\right)/\sim_{\mathrm{m}}\ \cong\ \mathbb{S}^2.
$$ 
The map $P\mate A_{\Gamma}^{\mathrm{fBS}}$ is called the \emph{topological mating} of $P$ and $A_{\Gamma}^{\mathrm{fBS}}$.
We say that $P$ and $A_{\Gamma}^{\mathrm{fBS}}$ are \emph{conformally mateable} if the topological $2-$sphere $\mathcal{K}(P)\mate \overline{\D}$ admits a complex structure that turns the topological mating $P\mate A_{\Gamma}^{\mathrm{fBS}}$ into a holomorphic~map.

Here is an equivalent formulation (cf. \cite[Definition~4.14]{petersen-meyer-mating}). 
\begin{definition}\label{conf_mat_def}
The maps $P$ and $A_{\Gamma}^{\mathrm{fBS}}$ are conformally mateable if there exist a continuous map $F\colon \mathrm{Dom}(F)\subsetneq\widehat{\C}\to\widehat{\C}$ (called a \emph{conformal mating} of $A_{\Gamma}^{\mathrm{fBS}}$ and $P$) that is complex-analytic in the interior of $\mathrm{Dom}(F)$ and continuous maps 
	$$
	\mathfrak{X}_P:\mathcal{K}(P)\to\widehat{\C}\ \textrm{and}\ \mathfrak{X}_\Gamma: \overline{\D}\to\widehat{\C},
	$$
	conformal on $\Int{\mathcal{K}(P)}$ and $\D$ (respectively), satisfying
	\begin{enumerate}
		\item\label{topo_cond} $\mathfrak{X}_P\left(\mathcal{K}(P)\right)\cup \mathfrak{X}_\Gamma\left(\overline{\D}\right) = \widehat{\C}$,
		
		\item\label{dom_cond} $\mathrm{Dom}(F)= \mathfrak{X}_P(\mathcal{K}(P))\cup\mathfrak{X}_\Gamma(\mathcal{D}_{\Gamma})$,
		
		\item $\mathfrak{X}_P\circ P(z) = F\circ \mathfrak{X}_P(z),\quad \mathrm{for}\ z\in\mathcal{K}(P)$,
		
		\item $\mathfrak{X}_\Gamma\circ A_{\Gamma}^{\mathrm{fBS}}(w) = F\circ \mathfrak{X}_\Gamma(w),\quad \mathrm{for}\ w\in
		\mathcal{D}_{\Gamma}$,\quad and

		\item\label{identifications} $\mathfrak{X}_P(z)=\mathfrak{X}_\Gamma(w)$ if and only if $z\sim_{\mathrm{m}} w$ where $\sim_{\mathrm{m}}$ is the equivalence relation on $\mathcal{K}(P)\sqcup \overline{\D}$ defined by Relation~\eqref{conf_mat_equiv_rel}.
	\end{enumerate}
The semi-conjugacies $\mathfrak{X}_P, \mathfrak{X}_\Gamma$ are called the \emph{mating semi-conjugacies} associated with the conformal mating $F$ of $P$ and $A_{\Gamma}^{\mathrm{fBS}}$. When mating semi-conjugacies are injective, they are simply referred to as \emph{mating conjugacies}.
\end{definition}

\subsection{Existence of conformal matings}\label{conf_mat_exists_subsec}

Let $(\rho: \pmb{\Gamma_{n,p}}\longrightarrow\Gamma)\in\mathrm{Teich}^\omega(\pmb{\Gamma_{n,p}})$ and $P$ a monic, centered, hyperbolic complex polynomial of degree $d=np-1$ with a connected Julia set. We now state and prove a generalization of \cite[Theorem~3.6]{MM1}.

\begin{theorem}\label{conf_mat_thm}
There exists a conformal mating $F$ of $P:\mathcal{K}(P)\to\mathcal{K}(P)$ and $A_{\Gamma}^{\mathrm{fBS}}:\mathcal{D}_{\Gamma}\to \overline{\D}$. Moreover, $F$ is unique up to M{\"o}bius conjugacy.
\end{theorem}

We will start with a technical lemma.

\begin{definition}\label{david_def}
	An orientation-preserving homeomorphism $H: U\to V$ between domains in the Riemann sphere $\widehat{\C}$ is called a \textit{David homeomorphism} if it lies in the Sobolev class $W^{1,1}_{\textrm{loc}}(U)$ and there exist constants $C,\alpha,\varepsilon_0>0$ with
	\begin{align}\label{david_cond}
		\sigma(\{z\in U: |\mu_H(z)|\geq 1-\varepsilon\}) \leq Ce^{-\alpha/\varepsilon}, \quad \varepsilon\leq \varepsilon_0.
	\end{align}
\end{definition}
Here $\sigma$ is the spherical measure, and
$\mu_H= \frac{\partial H/ \partial\overline{z}}{\partial H/\partial z}$
is the Beltrami coefficient of $H$ (see \cite[Chapter 20]{AIM09} for more background on David homeomorphisms).

\begin{lemma}\label{david_ext_lem}
The circle homeomorphism $\mathfrak{h}_\Gamma$ of Proposition~\ref{factor_bs_prop} continuously extends to a David homeomorphism of $\D$. 
\end{lemma}
\begin{proof}
Recall from Section~\ref{factor_bs_gen_subsec} that $h_\Gamma=\widehat{\psi}_\rho\circ\mathfrak{h}_0$, where $\widehat{\psi}_\rho$ is the restriction of a global quasiconformal map and $\mathfrak{h}_0$ conjugates $z^{np-1}$ to $A_{\pmb{\Gamma_{n,p}}}^{\mathrm{fBS}}$. Hence, it suffices to check that $\mathfrak{h}_0$ continuously extends to a David homeomorphism of $\D$.

We will use the notation introduced in Proposition~\ref{factor_bs_prop}.  In particular, we endow $\mathcal{\mathcal{Q}}$ with a preferred choice of complex coordinates via its identification with the set $\{z\in\overline{\D}:\ 0\leq\arg{z}\leq\frac{2\pi}{n}\}\cup\{0\}$ (with the boundary radial lines glued together).

The partition of $\partial\mathcal{Q}$ into the arcs 
$$
\{\mathfrak{q}(J_{1,s}^i): s\in\{1,\cdots,p\}, i\in\{1,\cdots,m(s)\}\}
$$ 
does \emph{not} necessarily give a Markov partition for $\overline{A}_{\pmb{\Gamma_{n,p}}}^{\mathrm{BS}}$ since the map may send both endpoints of such a partition piece to $1$. However, we can refine the above partition by pulling it back under $\overline{A}_{\pmb{\Gamma_{n,p}}}^{\mathrm{BS}}$, and this produces a Markov partition~$\{I_k\}$.

Each piece $h_{i,s}\vert_{I_k}$ of $\overline{A}_{\pmb{\Gamma_{n,p}}}^{\mathrm{BS}}$ extends conformally as $h_{i,s}$ to a neighborhood of $I_k$ in $\widetilde{\mathcal{Q}}$, where $\widetilde{\mathcal{Q}}$ is the double of $\mathcal{Q}$ and $I_k\subset \mathfrak{q}(J_{1,s}^i)$. 
Finally, since the pieces of $\overline{A}_{\pmb{\Gamma_{n,p}}}^{\mathrm{BS}}$ are M{\"o}bius  (with respect to the preferred coordinates on $\mathcal{Q}$), which send round disks to round disks, we can choose round disk neighborhoods $U_k\subset \widetilde{\mathcal{Q}}$ of the interiors of the Markov partition pieces $\Int{I_k}$ (intersecting $\partial\mathcal{Q}$ orthogonally) such that if $\overline{A}_{\pmb{\Gamma_{n,p}}}^{\mathrm{BS}}(I_k)\supset I_{k'}$, then $\overline{A}_{\pmb{\Gamma_{n,p}}}^{\mathrm{BS}}(U_k)\supset U_{k'}$.

The properties of $\overline{A}_{\pmb{\Gamma_{n,p}}}^{\mathrm{BS}}$ listed in the previous paragraph imply that the map $A_{\pmb{\Gamma_{n,p}}}^{\mathrm{fBS}}\equiv \xi\circ \overline{A}_{\pmb{\Gamma_{n,p}}}^{\mathrm{BS}}\vert_{\partial\mathcal{Q}}\circ\xi^{-1}:\mathbb{S}^1\to\mathbb{S}^1$ is a piecewise analytic orientation-preserving expansive covering map of degree $d\geq2$ admitting a Markov partition $\{\xi(I_k)\}$ satisfying conditions (4.1) and (4.2) of \cite[Theorem~4.12]{LMMN}. Moreover, each periodic breakpoint of its piecewise analytic definition is symmetrically parabolic (cf. \cite[Definition~4.6, Remark~4.7]{LMMN}). By \cite[Theorem~4.12]{LMMN}, the orientation-preserving homeomorphism $\mathfrak{h}_0:\mathbb{S}^1\to\mathbb{S}^1$ that conjugates the map $z^d\vert_{\mathbb{S}^1}$ to $A_{\pmb{\Gamma_{n,p}}}^{\mathrm{fBS}}\vert_{\mathbb{S}^1}$ (and sends $1$ to $1$) extends continuously as a David homeomorphism of~$\D$. 
\end{proof}

\begin{proof}[Proof of Theorem~\ref{conf_mat_thm}]
As $\mathcal{J}(P)$ is locally connected, $\psi_P$ extends to a continuous surjection $\psi_P:\mathbb{S}^1\to\mathcal{J}(P)$ semi-conjugating $z^d$ to $P$. Also note that since $P$ is hyperbolic with connected Julia set, $\mathcal{B}_\infty(P)$ is a John domain and $\mathcal{J}(P)$ is removable for $W^{1,1}$ functions \cite[Theorem~4]{JS00}.

Let $\mathfrak{h}_\Gamma:\overline{\D}\to\overline{\D}$ be a continuous extension of $\mathfrak{h}_\Gamma:\mathbb{S}^1\to\mathbb{S}^1$ that is a David homeomorphism on $\D$. The existence of such an extension is guaranteed by Lemma~\ref{david_ext_lem}.
Also recall that $\mathfrak{h}_\Gamma$ conjugates $z^d\vert_{\mathbb{S}^1}$ to $A_{\Gamma}^{\mathrm{fBS}}\vert_{\mathbb{S}^1}$.
Consider the topological dynamical system
	\begin{equation*}
		\widetilde{F}(w):=\left\{\begin{array}{ll}
			P & \mbox{on}\ \mathcal{K}(P), \\
			\psi_P\circ\eta\circ\mathfrak{h}_\Gamma^{-1}\circ A_{\Gamma}^{\mathrm{fBS}}\circ\mathfrak{h}_\Gamma\circ\eta\circ\psi_P^{-1} & \mbox{on}\  \psi_P\left(\eta\left(\mathfrak{h}_\Gamma^{-1}\left(\mathcal{D}_{\Gamma}\right)\right)\right)\subset\mathcal{B}_\infty(P),
		\end{array}\right.
	\end{equation*}
where $\eta(z)=1/z$. By equivariance properties of $\mathfrak{h}_\Gamma:\mathbb{S}^1\to\mathbb{S}^1$ and $\psi_P:\mathbb{S}^1\to\mathcal{J}(P)$, the two definitions agree on $\mathcal{J}(P)$. We denote the domain of $\widetilde{F}$ by $\mathrm{Dom}(\widetilde{F})$
	
	We define a Beltrami coefficient $\mu$ on the sphere as follows. On $\mathcal{K}(P)$ we set $\mu$ to be
the standard complex structure. On $\mathcal{B}_\infty(P)$, we set $\mu$ to be the pullback of the standard complex structure (on $\D$) under the map $\mathfrak{h}_\Gamma\circ\eta\circ\psi_P^{-1}$. Since $\mathfrak{h}_\Gamma\circ\eta\circ\psi_P^{-1}$ is a David homeomorphism (by \cite[Proposition~2.5 (part iv)]{LMMN}), it follows that $\mu$ is a David coefficient on $\widehat{\C}$. It is easy to check that $\mu$ is $\widetilde{F}-$invariant.

The David Integrability Theorem (see \cite{David}, \cite[Theorem~20.6.2, p.~578]{AIM09}) provides us with a David homeomorphism $\mathfrak{H}:\widehat{\C}\to\widehat{\C}$ such that the pullback of the standard complex structure under $\mathfrak{H}$ is equal to $\mu$. Conjugating $\widetilde{F}$ by $\mathfrak{H}$, we obtain the map 
$$
F:=\mathfrak{H}\circ\widetilde{F}\circ\mathfrak{H}^{-1}:\mathfrak{H}(\mathrm{Dom}(\widetilde{F}))\to\widehat{\C}.
$$
We set $\mathrm{Dom}(F):= \mathfrak{H}(\mathrm{Dom}(\widetilde{F}))$. 

We proceed to show that $F$ is holomorphic on $\Int{\mathrm{Dom}(F)}$.
As $\mathcal{J}(P)$ is removable for functions in class $W^{1,1}$, it follows from \cite[Theorem~2.7]{LMMN} that $\mathfrak{H}(\mathcal{J}(P))$ is locally conformally removable. Hence, it suffices to show that $F$ is holomorphic on the interior of $\mathrm{Dom}(F)\setminus\mathfrak{H}(\mathcal{J}(P))$. Indeed, this would imply that the continuous map $F$ is holomorphic on $\Int{\mathrm{Dom}(F)}$ away from the finitely many critical points of $F$. One can then conclude that $F$ is holomorphic on $\Int{\mathrm{Dom}(F)}$ using the Riemann removability theorem.

To this end, first observe that both the maps $\mathfrak{h}_\Gamma\circ\eta\circ\psi_P^{-1}$ and $\mathfrak{H}$ are David homeomorphisms on $\mathcal{B}_\infty(P)$ straightening $\mu\vert_{\mathcal{B}_\infty(P)}$. By \cite[Theorem~20.4.19, p.~565]{AIM09}, the map $\mathfrak{h}_\Gamma\circ\eta\circ\psi_P^{-1}\circ\mathfrak{H}^{-1}$ is conformal on $\mathfrak{H}(\mathcal{B}_\infty(P))$. It now follows from the definitions of $\widetilde{F}$ and $F$ that $F$ is holomorphic on $\mathfrak{H}(\mathcal{B}_\infty(P))\cap\Int{\mathrm{Dom}(F)}$. Similarly, both the identity map and the map $\mathfrak{H}$ are David homeomorphisms on each component of $\Int{\mathcal{K}(P)}$ straightening $\mu$. Once again by \cite[Theorem~20.4.19, p.~565]{AIM09}, $\mathfrak{H}$ is conformal on each component of $\Int{\mathcal{K}(P)}$. By definition of $\widetilde{F}$ and $F$, it now follows that $F$ is holomorphic on each interior component of $\mathfrak{H}(\mathcal{K}(P))$.
This completes the proof of the fact that $F$ is holomorphic on the interior of $\mathrm{Dom}(F)$.
\begin{figure}[h!]
\captionsetup{width=0.98\linewidth}
\begin{tikzpicture}
\node[anchor=south west,inner sep=0] at (0,6.4) {\includegraphics[width=0.96\textwidth]{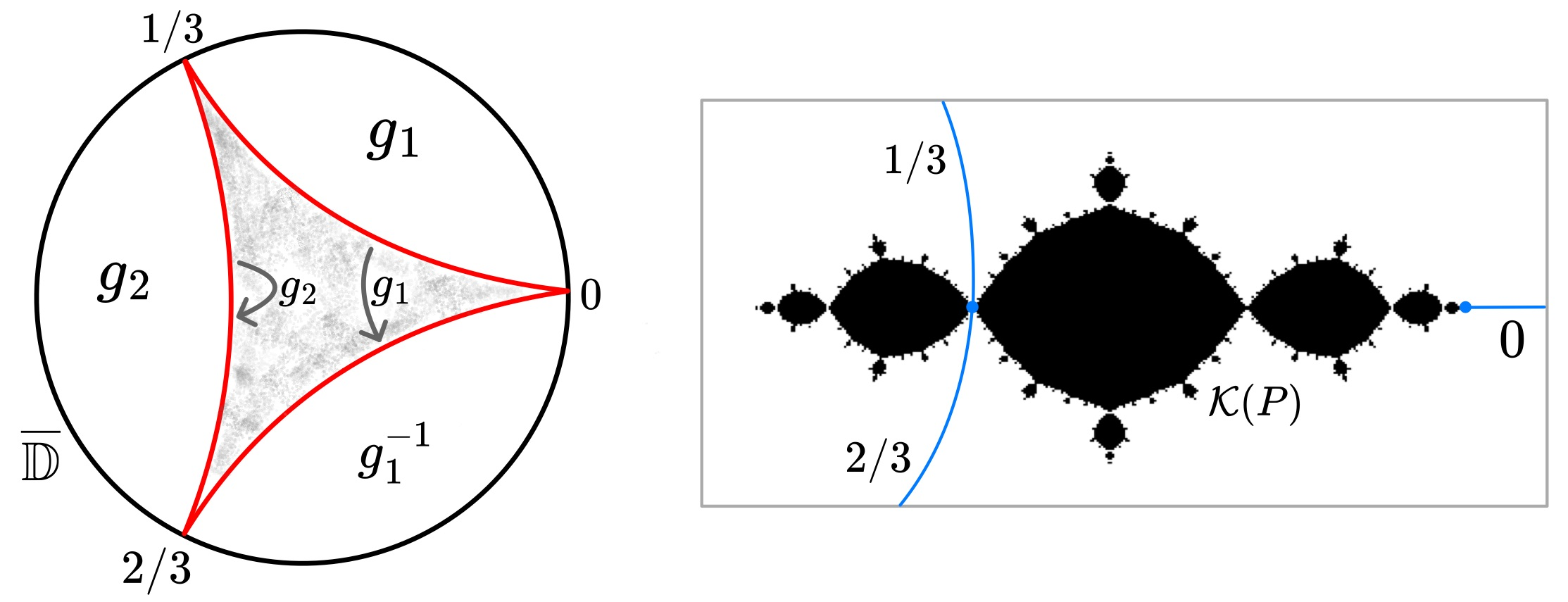}}; 
\node[anchor=south west,inner sep=0] at (1.6,0) {\includegraphics[width=0.72\textwidth]{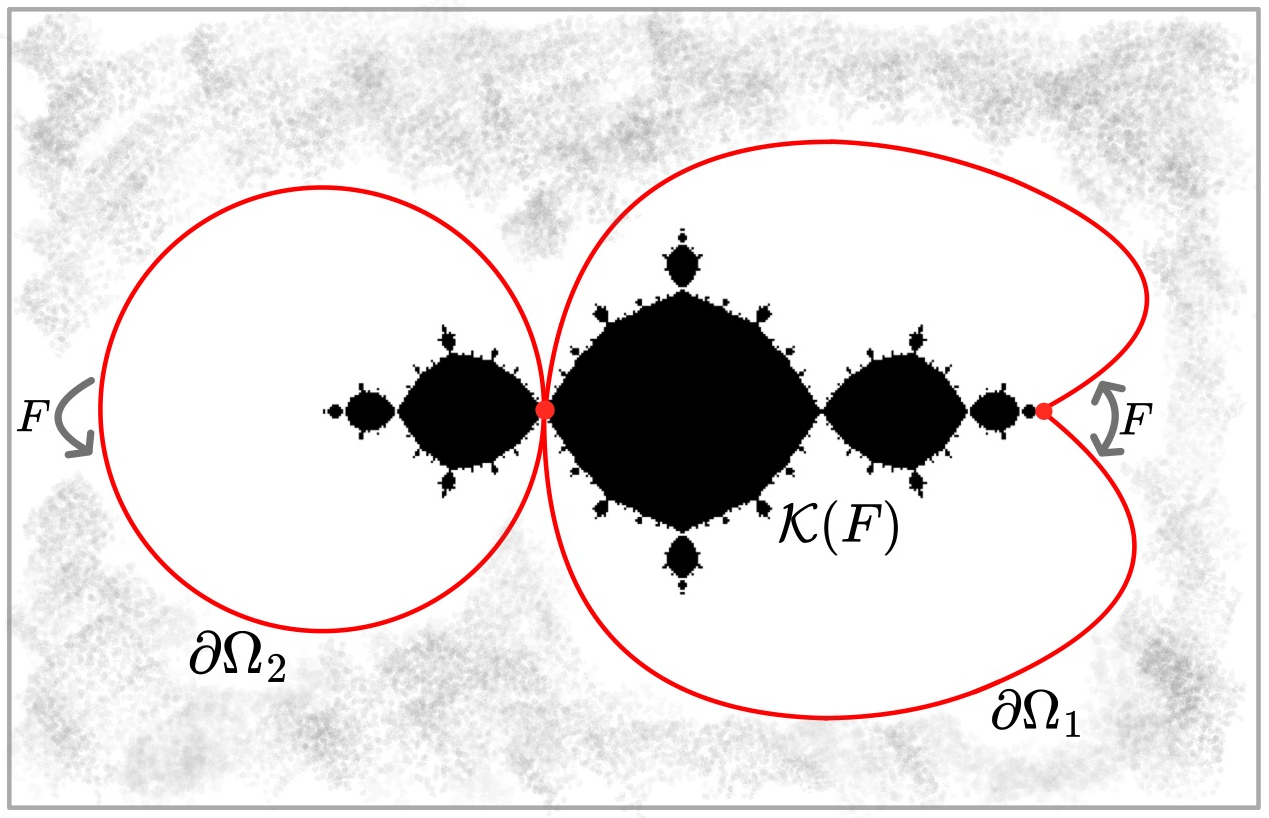}}; 
\end{tikzpicture}
\caption{Displayed are the Bowen-Series map of $\pmb{\Gamma_{1,3}}$, the filled Julia set of $P(z)=z^2-1$, and a schematic picture of the conformal mating $F$ of $A_{\pmb{\Gamma_{1,3}}}$ and~$P$.} 
\label{basilica_mating_fig}
\end{figure}

Finally, we set $\mathfrak{X}_P:=\mathfrak{H}:\mathcal{K}(P)\to\widehat{\C}$ and $\mathfrak{X}_\Gamma:=\mathfrak{H}\circ\psi_P\circ\eta\circ\mathfrak{h}_\Gamma^{-1}:\overline{\D}\to\widehat{\C}$. It is readily checked that these maps satisfy the requirements of Definition~\ref{conf_mat_def}. Thus, $F$ is a conformal mating of $P$ and $A_{\Gamma}^{\mathrm{fBS}}$.

Now suppose that there is another conformal mating $F_1$ of $P$ and $A_{\Gamma}^{\mathrm{fBS}}$. Then the respective mating semi-conjugacies paste together to yield a homeomorphism of $\widehat{\C}$ which is conformal away from $\mathfrak{H}(\mathcal{J}(P))$ and conjugates $F$ to $F_1$. Conformal removability of $\mathfrak{H}(\mathcal{J}(P))$ now implies that this homeomorphism is a M{\"o}bius map; i.e., $F$ and $F_1$ are M{\"o}bius conjugate. 
\end{proof}

\begin{remark}\label{mating_conj_rem}
  We note that since $\mathfrak{X}_P$ is the restriction of the homeomorphism $\mathfrak{H}$ to $\cK(P)$, the map $\mathfrak{X}_P$ is a homeomorphism.
\end{remark}

\subsection{Two examples}\label{examples_subsec}

We will now work out two concrete examples of conformal matings. The reader may keep these examples in mind while going through Sections~\ref{conf_mating_char_sec} and~\ref{gen_corr_subsec}.
\begin{figure}[h!]
\captionsetup{width=0.96\linewidth}
\begin{tikzpicture}
\node[anchor=south west,inner sep=0] at (0,8.8) {\includegraphics[width=0.96\textwidth]{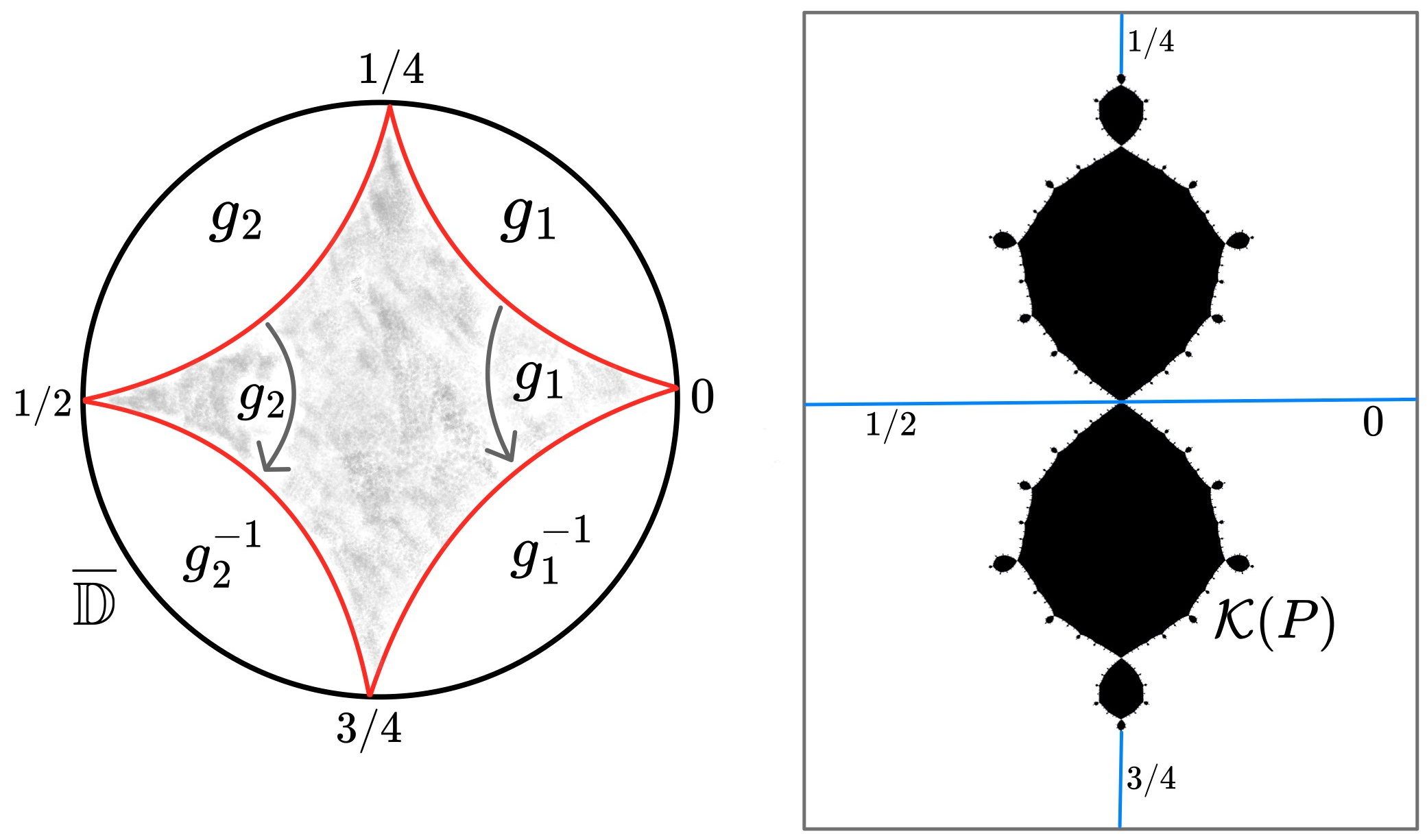}}; 
\node[anchor=south west,inner sep=0] at (3.2,0) {\includegraphics[width=0.45\textwidth]{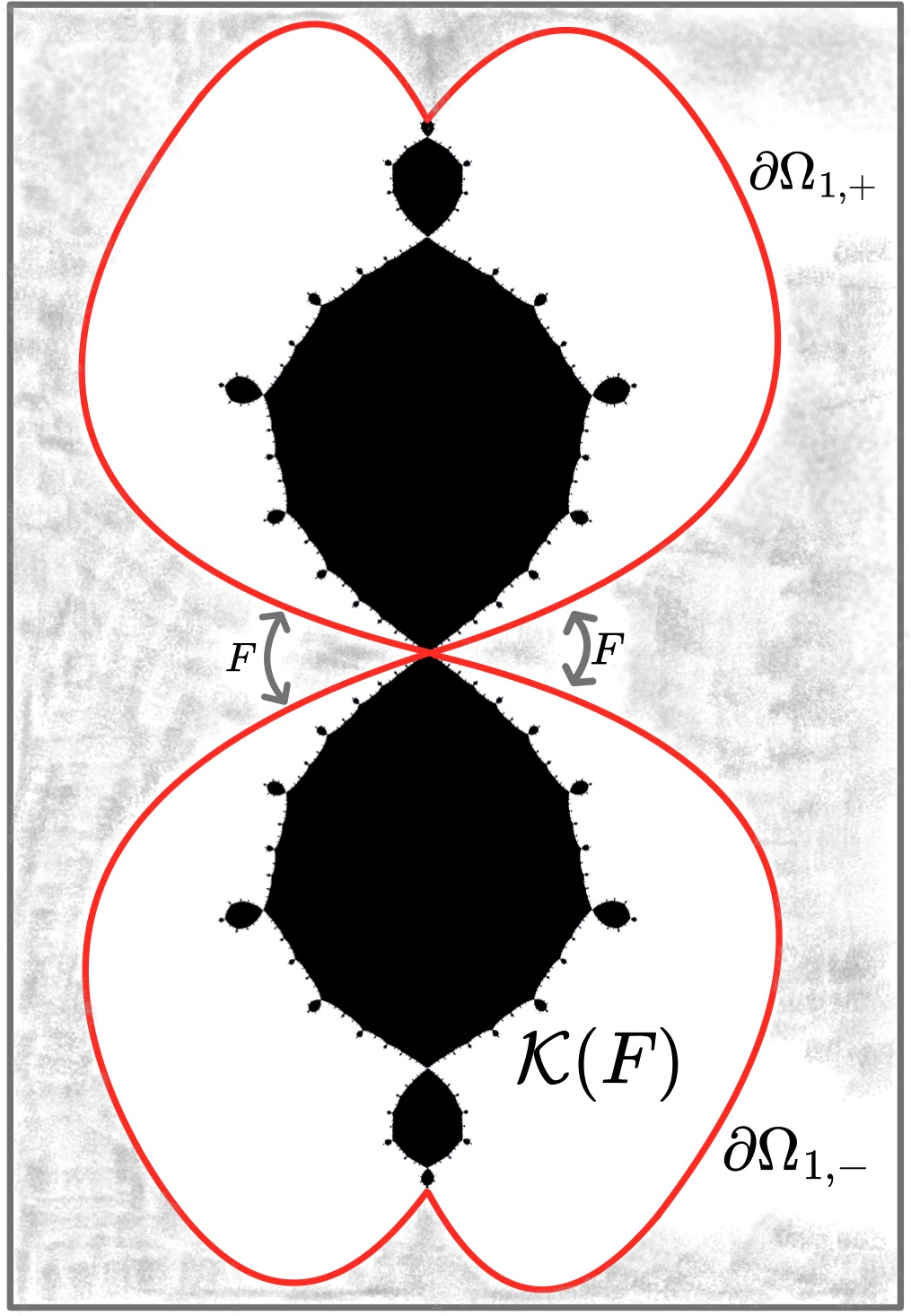}}; 
\end{tikzpicture}
\caption{Displayed are the Bowen-Series map of the Fuchsian thrice punctured sphere group, the filled Julia set of the critically fixed cubic polynomial $P(z)=z^3+\frac{3z}{2}$, and a cartoon of their conformal mating $F$.}
\label{cubic_crit_fixed_mating_fig}
\end{figure}

\subsubsection{A quadratic example}\label{quad_example_subsubsec}

Consider the `Basilica' polynomial $P(z)=z^2-1$ and the group $\pmb{\Gamma_{1,3}}$. The unique finite critical point of the polynomial $P$ forms a $2-$cycle $0\longleftrightarrow -1$. The filled Julia set $\cK(P)$, the fixed dynamical ray of $P$ at angle $0$, and the $2-$periodic dynamical rays at angles $1/3, 2/3$ are displayed in Figure~\ref{basilica_mating_fig} (top right). The group $\pmb{\Gamma_{1,3}}$ uniformizes a sphere with two punctures and an order two orbifold point. The Bowen-Series map $A_{\pmb{\Gamma_{1,3}}}:\mathcal{D}_{\pmb{\Gamma_{1,3}}}\to\overline{\D}$ is shown in Figure~\ref{basilica_mating_fig} (top left).

The construction of the conformal mating $F$ of $P$ and $A_{\pmb{\Gamma_{1,3}}}$ involves identification of the $2-$periodic points of $A_{\pmb{\Gamma_{1,3}}}$ in $S_{\pmb{\Gamma_{1,3}}}$ (see the last paragraph of Section~\ref{factor_bs_gen_subsec}) with the common landing point of the $1/3$ and $2/3-$dynamical rays of $P$. Consequently, $\mathrm{Dom}(F)$ is the union of two closed Jordan disks $\overline{\Omega_1}, \overline{\Omega_2}$ touching at a point, and $F$ restricts as an involution on the boundary of each of the two components $\Omega_1$, $\Omega_2$; see Figure~\ref{basilica_mating_fig} (bottom). 

\subsubsection{A cubic example}\label{cubic_example_subsubsec}

Let $P$ be the critically fixed cubic polynomial $P(z)=z^3+\frac{3z}{2}$. It fixes both  its finite critical points $\pm \frac{i}{\sqrt{2}}$. The filled Julia set $\cK(P)$, the fixed dynamical ray of $P$ at angles $0, 1/2$, and the $2-$periodic dynamical rays at angles $1/4, 3/4$ are shown in Figure~\ref{cubic_crit_fixed_mating_fig} (top right).

Further, let $\pmb{\Gamma_{1,4}}$ be the Fuchsian group uniformizing the thrice punctured sphere. The Bowen-Series map $A_{\pmb{\Gamma_{1,4}}}:\mathcal{D}_{\pmb{\Gamma_{1,4}}}\to\overline{\D}$ is depicted in Figure~\ref{cubic_crit_fixed_mating_fig} (top left).

In the dynamical plane of the conformal mating $F$ of $P$ and $A_{\pmb{\Gamma_{1,4}}}$, the fixed points of $A_{\pmb{\Gamma_{1,4}}}$ in $S_{\pmb{\Gamma_{1,4}}}$ are identified with the common landing point of the $0$ and $1/2-$dynamical rays of $P$. Thus, $\mathrm{Dom}(F)$ is the union of two closed Jordan disks $\overline{\Omega_{1,\pm}}$ touching at a point. However, unlike in the previous example, the map $F$ sends $\partial\Omega_{1,\pm}$ onto $\partial\Omega_{1,\mp}$ with $F^{\circ 2}\vert_{\partial\Omega_{1,+}\cup\partial\Omega_{1,-}}=\mathrm{id}$; see Figure~\ref{cubic_crit_fixed_mating_fig} (bottom). 

\section{Conformal matings and rational uniformization}\label{conf_mating_char_sec}

With the conformal matings of factor Bowen-Series maps and polynomials at our disposal (Theorem~\ref{conf_mat_thm}), we now take up the task of recognizing the class of holomorphic maps 
$F$ that arise in this process and thus answering the first part of Question~\ref{qn_main}.
	 We carry this out in Subsections~\ref{real_sym_mating_subsec} and~\ref{qc_real_sym_mating_subsec}, where a generalization of Proposition~\ref{rat_unif_conf_mating_intro_prop} is established. The resulting algebraic description of  matings in terms of uniformizing rational maps serves as a connecting link between conformal matings (Section~\ref{conf_mating_sec})
	 and correspondences (to be dealt with in Section~\ref{corr_sec}).
Finally in Section~\ref{unif_rat_crit_pnt_subsec}, we investigate the structure of the critical points of the uniformizing rational maps. This structure will play a crucial role in studying the dynamics of the associated correspondences in Section~\ref{corr_sec}.

Throughout this section, we will work with a representation $(\rho:\pmb{\Gamma_{n,p}}\longrightarrow\Gamma)\in\mathrm{Teich}^\omega(\pmb{\Gamma_{n,p}})$ and a monic, centered, hyperbolic complex polynomial $P$ of degree $d=np-1$ with a connected Julia set. As in Theorem~\ref{conf_mat_thm}, the unique conformal mating of $P$ and $A_{\Gamma}^{\mathrm{fBS}}$ will be denoted by $F$. The associated mating semi-conjugacies are denoted by $\mathfrak{X}_P$ and $\mathfrak{X}_\Gamma$ (see Definition~\ref{conf_mat_def}). Moreover, $\psi_P, \psi_\rho$ denote the B{\"o}ttcher coordinate for $P$, and the quasiconformal homeomorphism that defines the representation $\rho$, respectively.

\subsection{Lamination model of domain of conformal matings}\label{lami_model_subsec} For the purposes of this section, a \emph{lamination} will be a closed set consisting of a collection of non-intersecting hyperbolic geodesics in $\overline{\D}$. If the number of geodesics in the collection is finite, the lamination is said to be finite.

\begin{proposition}\label{mat_dom_prop}
$\mathrm{Dom}(F)$ is homeomorphic to the quotient of $\overline{\D}$ under an equivalence relation given by a finite lamination. In particular,
\noindent\begin{enumerate}
\item $\mathrm{Dom}(F)$ is connected, and
\item $\Int{\mathrm{Dom}(F)}$ has finitely many connected components and each of them is a Jordan domain.
\end{enumerate}
\end{proposition}
\begin{proof}
Recall that $\overline{\D}\setminus \mathcal{D}_{\Gamma}$ is the interior of a topological ideal $p-$gon. Since $\mathfrak{X}_\Gamma$ is conformal on $\D$ and extends continuously to $\mathbb{S}^1$, it follows that the complement of $\mathrm{Dom}(F)$ is a topological disk and hence $\mathrm{Dom}(F)$ is a full continuum. Moreover, $\mathfrak{X}_\Gamma$ can introduce at most finitely many identifications on the boundary of $\overline{\D}\setminus \mathcal{D}_{\Gamma}$ (namely, at the $p$ ideal boundary points of $\overline{\D}\setminus \mathcal{D}_{\Gamma}$). It follows that $\mathrm{Dom}(F)$ is homeomorphic to the quotient of $\overline{\D}$ under a finite lamination.
\end{proof}

We will now give an explicit description of the finite lamination appearing in the statement of Proposition~\ref{mat_dom_prop}. This will be useful in determining the topology of $\mathrm{Dom}(F)$. Recall the set $\mathcal{A}_p$ from Section~\ref{factor_bs_gen_subsec}.

\begin{definition}\label{lamination_def}
We define the equivalence relation $\mathcal{L}_P$ on $\mathcal{A}_p$ as: $\theta_1\sim_{\mathrm{P}}\theta_2$ if and only if the external dynamical rays of $P$ at angles $\theta_1, \theta_2$ land at the same point of~$\mathcal{J}(P)$. 
\end{definition}
\begin{figure}[h!]
\captionsetup{width=0.96\linewidth}
\begin{tikzpicture}
\node[anchor=south west,inner sep=0] at (0,0) {\includegraphics[width=0.8\textwidth]{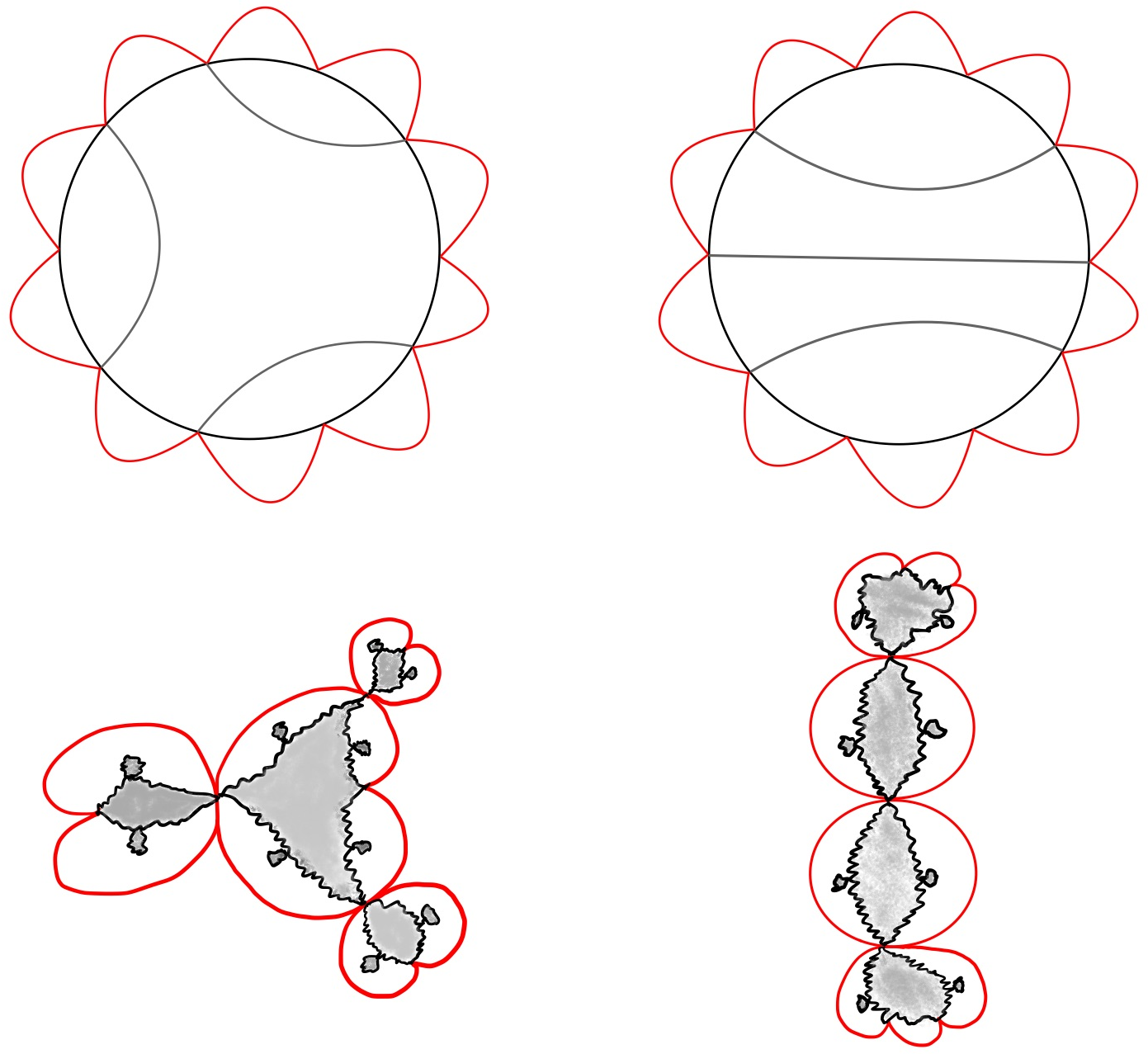}}; 
\end{tikzpicture}
\caption{Top: Depicted are various laminations $\mathcal{L}_P$ with the domain of definition of a factor Bowen-Series map superimposed outside the disk, for $p=10$. The domain of definition of the conformal mating is obtained by pinching the leaves of the lamination. Bottom: The corresponding topological models of $\mathrm{Dom}(F)$ and cartoons of the non-escaping sets are shown. The bottom left figure has three corners and three pinched points with rotational symmetry. The bottom right figure has two corners each at the top and bottom, and three pinched~points.}
\label{pinched_dom_fig}
\end{figure}
We remark that the above equivalence relation is \emph{unlinked}. One can view $\mathcal{L}_P$ as a finite lamination on $\D$ by joining the two points of an equivalence class by a bi-infinite hyperbolic geodesic. Note that \emph{leaves}, \emph{gaps} of this lamination can be defined in the usual way. In what follows, the word `gap' will refer to gaps of infinite hyperbolic area, while gaps of finite hyperbolic area will be called \emph{polygons}. 
By the construction of $F$, the landing point of the external dynamical ray of $P$ at angle $\theta$ is identified with $\mathfrak{h}_\Gamma(-\theta)$. Hence, the landing points of the external dynamical rays of $P$ at angles in $\mathcal{A}_p$ are identified with points in $S_{\Gamma}$. 

The following result easily follows from the above discussion.

\begin{lemma}\label{comp_count_lem}
The connected components of $\Int{\mathrm{Dom}(F)}$ correspond bijectively to the gaps of the lamination $\mathcal{L}_P$. Moreover, two components of $\Int{\mathrm{Dom}(F)}$ touch at a point if and only if the corresponding gaps of $\mathcal{L}_P$ either share a common boundary leaf or have boundary leaves corresponding to the sides of a common polygon.
\end{lemma}

\begin{remark}
The topology of $\mathrm{Dom}(F)$ depends on the equivalence relation $\mathcal{L}_P$ and the integer $p$, but not on the integer $n$.
\end{remark}

\begin{definition}\label{gap_cusp_def}
For a gap $\mathcal{G}$ of the lamination $\mathcal{L}_P$, the set of points of $\mathcal{A}_p$ that lie on $\overline{\mathcal{G}}$ but are not endpoints of any leaf of $\mathcal{L}$ is denoted by $\mathrm{cusps}(\mathcal{G})$.
\end{definition}

Recall from \cite[\S 18]{Mil06} that the periods of the external dynamical rays landing at a periodic point of $\mathcal{J}(P)$ are equal. Thus, if $0$ or $1/2$ belongs to a non-trivial equivalence class of $\mathcal{L}_P$, then this class must be $\{0,1/2\}$. On the other hand, if $\frac{i}{p}$ and $\frac{j}{p}$ (where $i,j\in\{1,\cdots,\lfloor \frac{p-1}{2}\rfloor\}$) lie in the same equivalence class, then $m_d(\frac{i}{p})=-\frac{i}{p}$ and $m_d(\frac{j}{p})=-\frac{j}{p}$ must do so as well. It follows that the gaps of the lamination $\mathcal{L}_P$ either intersect the real line, or come in complex conjugate pairs (see Figure~\ref{pinched_dom_fig}). We enumerate the gaps of $\mathcal{L}_P$ as $\mathcal{G}_1,\cdots,\mathcal{G}_l,\mathcal{G}_{1,\pm},\cdots,\mathcal{G}_{m,\pm}$, and label the corresponding components of $\Int{\mathrm{Dom}(F)}$ as $\Omega_1,\cdots,\Omega_l,\Omega_{1,\pm},\cdots,\Omega_{m,\pm}$, where $\mathcal{G}_i$ are the real-symmetric gaps and $\mathcal{G}_{j,\pm}$ are the complex-conjugate gaps of $\mathcal{L}_P$.
Thus, 
$$
\displaystyle\mathrm{Dom}(F)=\left(\bigcup_{i=1}^l\overline{\Omega_i}\right)\bigcup\left(\bigcup_{j=1}^m \overline{\Omega_{j,+}}\cup \overline{\Omega_{j,-}}\right).
$$

\begin{remark}\label{examples_lamination_rem}
For the conformal mating considered in Section~\ref{quad_example_subsubsec}, the lamination $\mathcal{L}_P$ consists of a unique leaf connecting $1/3$ and $2/3$. Similarly, for the conformal mating considered in Section~\ref{cubic_example_subsubsec}, the lamination $\mathcal{L}_P$ consists of a unique leaf connecting $0$ and $1/2$. In both cases, $\mathcal{L}_P$ has two gaps. This is in consonance with the fact that the domains of definition of the associated conformal matings have two interior components.
\end{remark}

\subsection{Explicit description of real-symmetric matings via rational uniformizations}\label{real_sym_mating_subsec}

In this subsection, we will give a concrete description of the mating between $A_{\pmb{\Gamma_{n,p}}}^{\mathrm{fBS}}$ and $\pmb{P}$, where $\pmb{P}$ is a real-symmetric hyperbolic polynomial of degree $d=np-1$ with a connected Julia set (assuming that the mating exists).
The characterization of such matings will be based on the following lemma. We recall the notation $\iota(z)=\overline{z}$, $\eta^{-}(z)=1/\overline{z}$, and $\D^*= \widehat{\C}\setminus\overline{\D}$.

\begin{lemma}\label{qd_lem}
Let $\Omega\subset\widehat{\C}$ be a simply connected domain with locally connected boundary. Suppose further that $f:\overline{\Omega}\to\widehat{\C}$ is a continuous function such that
\begin{enumerate}
\item $f$ is meromorphic on $\Omega$, and
\item $f\equiv \iota$ on $\partial\Omega$.
\end{enumerate}
Then there exists a rational map $R:\widehat{\C}\to\widehat{\C}$ that carries $\D$ univalently onto $\Omega$ and $\left(\iota\circ f\right)\vert_{\Omega}\equiv R\circ\eta^-\circ\left(R\vert_{\D}\right)^{-1}$.
\end{lemma}
\begin{proof}
As $\Omega$ is simply connected with locally connected boundary, there exists a conformal isomorphism $\phi:\D\to\Omega$ that extends to a continuous surjection $\phi:\mathbb{S}^1\to\partial\Omega$. We will show that $\phi$ extends to a meromorphic map of the Riemann sphere, and hence is a rational map. To this end, we define
$$
R:\widehat{\C}\to\widehat{\C},\quad R\equiv
\begin{cases}
\phi,\quad \textrm{on}\ \overline{\D},\\
\left(\iota \circ f\right) \circ \phi \circ \eta^-,\quad \textrm{on}\ \D^*.
\end{cases}
$$
By our assumption, $\iota\circ f\equiv \mathrm{id}$ on $\partial\Omega$. This fact, combined with continuity of $\phi$ and $f$, implies that $R$ is continuous on $\widehat{\C}$. Moreover, $R$ is meromorphic away from $\mathbb{S}^1$. It follows from the conformal removability of analytic arcs that $R$ is a global meromorphic function. Therefore, $R$ is a rational map that takes $\D$ injectively onto $\Omega$. Finally, by construction of $R$, we have that $(\iota\circ f)\circ R\circ \eta^-\equiv R$ on $\D^*$, and hence, $\left(\iota\circ f\right)\equiv R\circ\eta^-\circ\left(R\vert_{\D}\right)^{-1}$ on $\Omega=R(\D)$.
\end{proof}

\begin{remark}
Domains $\Omega$ satisfying the conditions of Lemma~\ref{qd_lem} are examples of so-called \emph{quadrature domains}, and the associated maps $f$ are called \emph{Schwarz functions}. The characterization of such domains given in Lemma~\ref{qd_lem} is a special case of \cite[Theorem~1]{AS76}, where a similar result is proven without the local connectedness assumption. However, we will not need this more general statement in this paper.
\end{remark}

With the above preparatory lemma at our disposal, we now proceed to prove the main result of this subsection.
Let us denote the conformal mating of $\pmb{P}:\mathcal{K}(\pmb{P})\to\mathcal{K}(\pmb{P})$ and $A_{\pmb{\Gamma_{n,p}}}^{\mathrm{fBS}}:\mathcal{D}_{\pmb{\Gamma_{n,p}}}\to~\overline{\D}$ by $\pmb{F}$. Following the convention of Section~\ref{lami_model_subsec}, we will label the components of $\Int{\mathrm{Dom}(\pmb{F})}$ as 
$$
\{\pmb{\Omega}_\alpha:\alpha\in\mathcal{I}\},
$$
where
$$
\mathcal{I}:=\mathcal{I}_1\sqcup\mathcal{I}_2,\quad  \textrm{with}\quad \mathcal{I}_1:=\{1,\cdots, l\},\quad \mathcal{I}_2:=\{(1,+),(1,-),\cdots,(m,+),(m,-)\}. 
$$
We  also define an involution
$$
\kappa:\mathcal{I}\to\mathcal{I},\quad 
\begin{cases}
i\mapsto i,\ \mathrm{for}\ i\in\mathcal{I}_1,\\
(j,\pm)\mapsto (j,\mp),\ \mathrm{for}\ (j,\pm)\in\mathcal{I}_2.
\end{cases} 
$$

\begin{lemma}\label{real_sym_schwarz_lem}
There exist rational maps $\pmb{R}_\alpha$, $\alpha\in\mathcal{I}$, of $\widehat{\C}$ such that for each $\alpha$
\begin{enumerate}
\item $\iota(\pmb{\Omega}_\alpha)=\pmb{\Omega}_{\kappa(\alpha)}$,

\item $\pmb{R}_\alpha$ maps $\overline{\D}$ injectively onto $\overline{\pmb{\Omega}_\alpha}$,

\item  $\iota\circ\pmb{R}_\alpha=\pmb{R}_{\kappa(\alpha)}\circ\iota$, and

\item $\pmb{F}\vert_{\overline{\pmb{\Omega}_\alpha}}\equiv \pmb{R}_{\kappa(\alpha)}\circ\eta\circ(\pmb{R}_\alpha\vert_{\overline{\D}})^{-1}$.
\end{enumerate}
\end{lemma}
\begin{proof}
We first claim that $\mathrm{Dom}(F)$ is real-symmetric and $F\circ \iota=\iota\circ F$.
\begin{proof}[Proof of claim]
The real-symmetry property of $\pmb{P}$ implies that $\iota\circ\psi_{\pmb{P}}\circ\iota$ conjugates $z^{d}$ to $\pmb{P}$ and is tangent to the identity map near infinity. By uniqueness of B{\"o}ttcher coordinates, we have that $\iota\circ\psi_{\pmb{P}}\circ\iota=\psi_{\pmb{P}}$; i.e., $\psi_{\pmb{P}}$ commutes with $\iota$. 

The orientation-preserving topological conjugacy $\mathfrak{h}_0:\mathbb{S}^1\to\mathbb{S}^1$ between $z^{d}$ and $A_{\pmb{\Gamma_{n,p}}}^{\mathrm{fBS}}$ sends the fixed point $1$ of $z^{d}$ to the fixed point $1$ of $A_{\pmb{\Gamma_{n,p}}}^{\mathrm{fBS}}$. Since both $z^d$ and $A_{\pmb{\Gamma_{n,p}}}^{\mathrm{fBS}}$ are real-symmetric maps (i.e., they commute with $\iota$), it follows that $\iota\circ\mathfrak{h}_0\circ\iota:\mathbb{S}^1\to\mathbb{S}^1$ is also an orientation-preserving topological conjugacy between $z^d$ and $A_{\pmb{\Gamma_{n,p}}}^{\mathrm{fBS}}$ sending $1$ to $1$. Thus, $\mathfrak{g}:=(\iota\circ\mathfrak{h}_0\circ\iota)^{-1}\circ \mathfrak{h}_0:\mathbb{S}^1\to\mathbb{S}^1$ commutes with $z^{d}$, and carries the fixed point $1$ of $z^{d}$ to itself. Due to this commutation property, the orientation-preserving circle homeomorphism $\mathfrak{g}$ acts as the identity on the $m-$th pre-images of $1$ under $z^{d}$, for $m\geq 1$. Since the iterated pre-images of $1$ under $z^{d}$ are dense in $\mathbb{S}^1$, it follows that $\mathfrak{g}$ is the identity map on $\mathbb{S}^1$. Therefore, $\mathfrak{h}_0$ is real-symmetric. By \cite[Remark~2.4]{LMMN}, the David extension of $\mathfrak{h}_0$ to $\D$ is real-symmetric. Moreover, as $\psi_{\pmb{P}}$ is real-symmetric, it follows that the David coefficient $\mu$ appearing in the construction of Theorem~\ref{conf_mat_thm} is also real-symmetric. The uniqueness part of David Integrability Theorem (see \cite{David}, \cite[Theorem~20.6.2, p.~578]{AIM09}) then implies that the David homeomorphism $\mathfrak{H}$ solving the Beltrami equation with coefficient $\mu$ is real-symmetric, from which real-symmetry of $\pmb{F}$ follows.
\end{proof}

Due to real-symmetry of $\pmb{F}$, the escaping and non-escaping sets of $\pmb{F}$ and the associated mating semi-conjugacies are real-symmetric. It follows that if $\pmb{\Omega}_\alpha$ is a component of $\Int{\mathrm{Dom}(\pmb{F})}$, then $\iota(\pmb{\Omega}_\alpha)=\pmb{\Omega}_{\kappa(\alpha)}$.

Note that by the description of the M{\"o}bius maps $g_{1,s}$ given in Subsection~\ref{factor_bs_base_subsec}, the factor Bowen-Series map $A_{\pmb{\Gamma_{n,p}}}^{\mathrm{fBS}}$ acts as the complex conjugation map $\iota$ on the boundary of $\overline{\D}\setminus\mathcal{D}_{\pmb{\Gamma_{n,p}}}$. By the real-symmetry property of the mating semi-conjugacies, it follows that $F$ also acts as $\iota$ on the boundary of $\mathrm{Dom}(\pmb{F})$. \emph{This is one of the key facts that 
leads to the uniformization maps $\pmb{R}_\alpha$ in Equation~\ref{schwarz_real_eqn} below.}

Also observe that since $\partial\mathrm{Dom}(F)$ is the image of the boundary of $\overline{\D}\setminus \mathcal{D}_{\Gamma}$ under the mating semi-conjugacy $\mathfrak{X}_\Gamma$ (see Definition~\ref{conf_mat_def}), it is locally connected. Thus, the restriction of $\pmb{F}$ to the closure of each component of $\Int{\mathrm{Dom}(\pmb{F})}$ satisfies the hypothesis of Lemma~\ref{qd_lem}. Therefore, for each $\alpha\in\mathcal{I}$, there exists a rational map $\pmb{R}_\alpha$ of $\widehat{\C}$ such that $\pmb{R}_\alpha:\overline{\D}\to\overline{\pmb{\Omega}_\alpha}$ is a conformal isomorphism. 
Moreover, by Lemma~\ref{qd_lem}, we have 
\begin{equation}
\left(\iota\circ \pmb{F}\right)\vert_{\pmb{\Omega}_\alpha}\equiv \pmb{R}_\alpha\circ\eta^-\circ\left(\pmb{R}_\alpha\vert_{\D}\right)^{-1}.
\label{schwarz_real_eqn}
\end{equation}

Since $\iota(\pmb{\Omega}_\alpha)=\pmb{\Omega}_{\kappa(\alpha)}$, the uniqueness of Riemann maps implies that the uniformizing rational maps $\pmb{R}_\alpha$ can be chosen so that $\iota\circ\pmb{R}_\alpha=\pmb{R}_{\kappa(\alpha)}\circ\iota$.
\[
\begin{tikzcd}
\D \arrow{d}[swap]{\eta} \arrow{r}{\pmb{R}_\alpha} & \pmb{\Omega}_\alpha \arrow{d}{\pmb{F}} \\
\widehat{\C} \arrow{r}{\pmb{R}_{\kappa(\alpha)}}  & \widehat{\mathbb{C}}
\end{tikzcd}
\]
These relations, combined with those in \eqref{schwarz_real_eqn}, imply that $
\pmb{F}\vert_{\overline{\pmb{\Omega}_\alpha}}\equiv \pmb{R}_{\kappa(\alpha)}\circ\eta\circ(\pmb{R}_{\alpha}\vert_{\overline{\D}})^{-1}.$
\end{proof}

As $\pmb{P}:\mathcal{K}(\pmb{P})\to\mathcal{K}(\pmb{P})$ and $A_{\pmb{\Gamma_{n,p}}}^{\mathrm{fBS}}: \left(A_{\pmb{\Gamma_{n,p}}}^{\mathrm{fBS}}\right)^{-1}(\mathcal{D}_{\pmb{\Gamma_{n,p}}})\to \mathcal{D}_{\pmb{\Gamma_{n,p}}}$ are degree $d$ maps (see Proposition~\ref{factor_bs_prop}), it follows that
$\pmb{F}:\pmb{F}^{-1}(\Int{\mathrm{Dom}(\pmb{F})})\longrightarrow \Int{\mathrm{Dom}(\pmb{F})}$ is a branched covering of degree $d$. On the other hand, as $A_{\pmb{\Gamma_{n,p}}}^{\mathrm{fBS}}: \left(A_{\pmb{\Gamma_{n,p}}}^{\mathrm{fBS}}\right)^{-1}(\overline{\D}\setminus\mathcal{D}_{\pmb{\Gamma_{n,p}}})\to \overline{\D}\setminus\mathcal{D}_{\pmb{\Gamma_{n,p}}}$ has degree $d+1$, it follows that $\pmb{F}:\pmb{F}^{-1}(\widehat{\C}\setminus\mathrm{Dom}(\pmb{F}))\longrightarrow \widehat{\C}\setminus\mathrm{Dom}(\pmb{F})$ is a branched covering of degree $d+1$.

Suppose that the global degree of the rational map $\pmb{R}_\alpha$, $\alpha\in\mathcal{I}$, is $d_\alpha$. The next result relates these degrees to the degree $d$ of the polynomial $\pmb{P}$.
\begin{lemma}\label{deg_formula_lem}
We have the following relation:
\begin{equation}
\sum_{\alpha\in\mathcal{I}} d_\alpha=\sum_{i=1}^l d_i + 2\cdot \sum_{j=1}^m d_{j,+} = d+1.
\label{degree_formula}
\end{equation}
\end{lemma}
\begin{proof}
By Lemma~\ref{real_sym_schwarz_lem}, we have $d_{\alpha}=d_{\kappa(\alpha)}$. The first equality now follows from the definition of $d_\alpha$.

The relation $\pmb{F}\vert_{\overline{\pmb{\Omega}_\alpha}}\equiv \pmb{R}_{\kappa(\alpha)}\circ\eta\circ(\pmb{R}_\alpha\vert_{\overline{\D}})^{-1}$ implies that $\pmb{F}:\pmb{F}^{-1}(\pmb{\Omega}_{\kappa(\alpha)})\cap\pmb{\Omega}_\alpha\longrightarrow \pmb{\Omega}_{\kappa(\alpha)}$ is a branched covering of degree $d_\alpha-1$, and
$\pmb{F}:\pmb{F}^{-1}(\Int{\pmb{\Omega}_{\kappa(\alpha)}^\complement})\cap\pmb{\Omega}_\alpha\longrightarrow \Int{\pmb{\Omega}_{\kappa(\alpha)}^\complement}$ is a branched covering of degree $d_\alpha$. Consequently, points in $\widehat{\C}\setminus\mathrm{Dom}(\pmb{F})=\Int{\left(\bigcap_{\alpha\in\mathcal{I}} \pmb{\Omega}_\alpha^\complement\right)}$ have $\sum_{\alpha\in\mathcal{I}} d_\alpha$ many preimages under $\pmb{F}$ (counted with multiplicity); i.e., $\pmb{F}:\pmb{F}^{-1}(\widehat{\C}\setminus\mathrm{Dom}(\pmb{F}))\longrightarrow \widehat{\C}\setminus\mathrm{Dom}(\pmb{F})$ has degree $\sum_{\alpha\in\mathcal{I}} d_\alpha$. In light of the discussion preceding this proposition, we conclude that $\sum_{\alpha\in\mathcal{I}} d_\alpha=d+1.$
\end{proof}

\subsection{Quasiconformal conjugations of real-symmetric matings}\label{qc_real_sym_mating_subsec}

We now look at real-symmetric hyperbolic components in connectedness loci of polynomials. Any polynomial $P$ in such a hyperbolic component is quasiconformally conjugate to a real-symmetric hyperbolic polynomial $\pmb{P}$ with a connected Julia set (cf. \cite{MS98}).

\begin{proposition}\label{mating_class_prop}
Let $(\rho:\pmb{\Gamma_{n,p}}\longrightarrow\Gamma)\in\mathrm{Teich}^\omega(\pmb{\Gamma_{n,p}})$ and $P$ be a polynomial lying in a real-symmetric hyperbolic component in the connectedness locus of degree $d$ polynomials. Further let $F$ be the conformal mating of $P$ and $A_{\Gamma}^{\mathrm{fBS}}$, and $\Omega_\alpha$, $\alpha\in\mathcal{I}$, be the components of $\Int{\mathrm{Dom}(F)}$ (where the labeling follows the convention of Section~\ref{lami_model_subsec}).

Then, for all $\alpha\in\mathcal{I}$, there exist Jordan domains $\mathfrak{D}_\alpha$ and rational maps $R_\alpha$ of degree $d_\alpha$ (with $d_\alpha=d_{\kappa(\alpha)}$) of $\widehat{\C}$ such that 
\begin{enumerate}
\item $\eta(\partial\mathfrak{D}_\alpha)=\partial\mathfrak{D}_{\kappa(\alpha)}$, 

\item $R_\alpha$ maps $\overline{\mathfrak{D}_\alpha}$ injectively onto $\overline{\Omega_\alpha}$,

\item $F\vert_{\overline{\Omega_\alpha}}\equiv R_{\kappa(\alpha)}\circ\eta\circ(R_\alpha\vert_{\overline{\mathfrak{D}_\alpha}})^{-1}$, and

\item $\sum_{\alpha\in\mathcal{I}} d_\alpha = d+1$. 
\end{enumerate}
\end{proposition}
\begin{proof}
Note that $P$ is quasiconformally conjugate to some real-symmetric hyperbolic polynomial $\pmb{P}$ of degree $d$, and $A_{\Gamma}^{\mathrm{fBS}}$ is quasiconformally conjugate to $A_{\pmb{\Gamma_{n,p}}}^{\mathrm{fBS}}$. It follows that the conformal mating $F$ of $P$ and $A_{\Gamma}^{\mathrm{fBS}}$, which is unique up to M{\"o}bius conjugacy, is quasiconformally conjugate to the conformal mating $\pmb{F}$ of $\pmb{P}$ and $A_{\pmb{\Gamma_{n,p}}}^{\mathrm{fBS}}$. Let $\pmb{\Omega}_\alpha, \pmb{R}_\alpha$ be as in Lemma~\ref{real_sym_schwarz_lem}.

Suppose that $\Psi$ is a quasiconformal homeomorphism of the Riemann sphere to itself such that $F=\Psi\circ \pmb{F}\circ \Psi^{-1}$. We set $\mu:=\Psi^*(\mu_0)=\partial_{\overline{z}}\Psi/\partial_z\Psi$ (here, $\mu_0$ is the trivial Beltrami coefficient on the Riemann sphere). As the holomorphic map $F$ preserves $\mu_0$, we have that $\mu$ is an $\pmb{F}-$invariant Beltrami coefficient. We pull $\mu$ back by $\pmb{R}_\alpha$ to obtain Beltrami coefficients $\mu_\alpha:=\pmb{R}_\alpha^*(\mu)$. 

Let us first work with $\alpha=i\in\mathcal{I}_1$. Since $\pmb{F}(\pmb{R}_i(z))=\pmb{R}_i(\eta(z))$ for $z\in\D$, the invariance of $\mu$ under $\pmb{F}$ translates to the $\eta-$invariance of $\mu_i$. Let $\Psi_i$ be a quasiconformal map solving the Beltrami equation with coefficient $\mu_i$.
By construction, the quasiregular map $R_i:=\Psi\circ \pmb{R}_i\circ \Psi_i^{-1}:\widehat{\C}\to\widehat{\C}$ preserves the standard complex structure, and is thus a rational map. Also, $\Psi_i\circ\eta\circ\Psi_i^{-1}$ is a M{\"o}bius involution.
Since any M{\"o}bius involution is conjugate to $\eta$, the map $\Psi_i\circ\eta\circ\Psi_i^{-1}$ can be chosen to be $\eta$ after possibly post-composing $\Psi_i$ with a M{\"o}bius map. 
Set $\mathfrak{D}_i:=\Psi_i(\D)$, and $\Omega_i:=\Psi(\pmb{\Omega}_i)= R_i(\mathfrak{D}_i)$. Since $\D$ is mapped inside out by $\eta$, it follows that $\mathfrak{D}_i$ is also mapped inside out by $\eta$. In particular, $\mathfrak{D}_i$ is a Jordan domain such that $\partial\mathfrak{D}_i$ is $\eta-$invariant (see Figure~\ref{deform_fig} and the commutative diagram that follows).
\begin{figure}[h!]
\captionsetup{width=0.96\linewidth}
\begin{tikzpicture}
\node[anchor=south west,inner sep=0] at (0,0) {\includegraphics[width=0.775\textwidth]{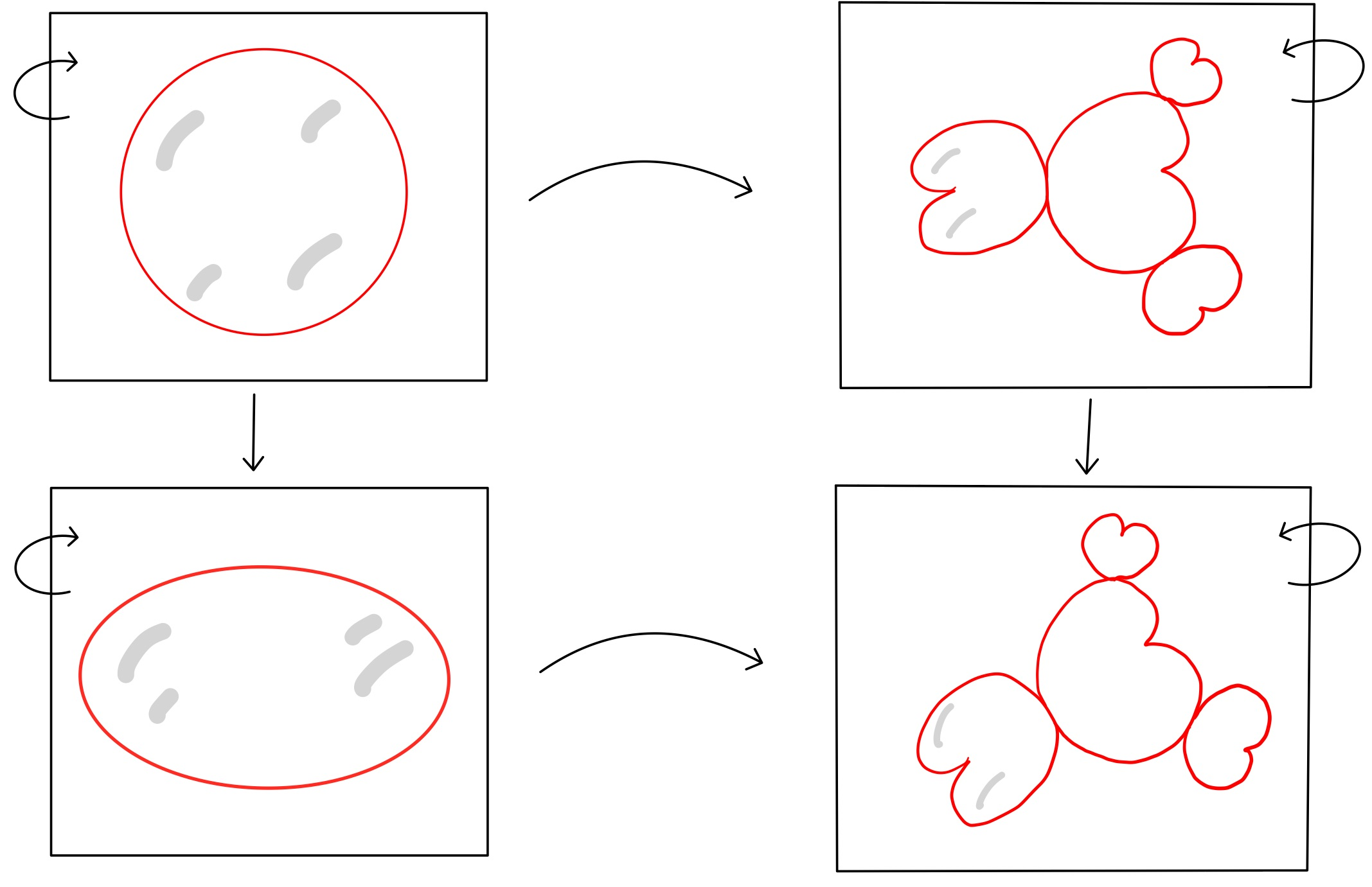}}; 
\node at (2,5.06) {\begin{large}$\D$\end{large}};
\node at (2,1.6) {\begin{large}$\mathfrak{D}_i$\end{large}};
\node at (-0.1,2.32) {$\eta$};
\node at (-0.1,5.84) {$\eta$};
\node at (1.6,3.36) {$\Psi_i$};
\node at (8.32,3.36) {$\Psi$};
\node at (4.75,5.64) {\begin{large}$\pmb{R}_i$\end{large}};
\node at (4.75,2.16) {\begin{large}$R_i$\end{large}};
\node at (10.32,6) {$\pmb{F}$};
\node at (10.32,2.5) {$F$};
\node at (7.4,1.12) {$\Omega_i$};
\node at (7.36,5.16) {$\pmb{\Omega}_i$};
\end{tikzpicture}
\caption{Illustrated is the proof of Proposition~\ref{mating_class_prop}. See Figure~\ref{pinched_dom_fig} (left) for a cartoon of the non-escaping set of $\pmb{F}$.}
\label{deform_fig}
\end{figure}
\[
  \begin{tikzcd}
    (\mathfrak{D}_i,\mu_0) \arrow{d}{\eta} \arrow{r}{\Psi_i^{-1}} & (\D,\mu_i) \arrow{d}{\eta} \arrow{r}{\pmb{R}_i} & (\pmb{\Omega}_i,\mu) \arrow{d}{\pmb{F}} \arrow{r}{\Psi} & (\Omega_i,\mu_0) \arrow{d}{F} \\
     (\widehat{\C},\mu_0) \arrow{r}{\Psi_i^{-1}}  & (\widehat{\mathbb{C}},\mu_i) \arrow{r}{\pmb{R}_i} & (\widehat{\mathbb{C}},\mu) \arrow{r}{\Psi} &  (\widehat{\mathbb{C}},\mu_0)
       \end{tikzcd}
\] 
As $\pmb{R}_i$ is injective on $\overline{\D}$, we conclude that $R_i$ is injective on the closed Jordan disk $\overline{\mathfrak{D}}_i$ (which is mapped inside out by $\eta$), and $F$ can be written as $R_i\circ\eta\circ (R_i\vert_{\overline{\mathfrak{D}_i}})^{-1}$ on $\overline{\Omega_i}$.

Now we turn our attention to $\alpha=(j,\pm)\in\mathcal{I}_2$. Let $\Psi_{j,\pm}$ be quasiconformal maps solving the Beltrami equations with coefficient $\mu_{j,\pm}$.
As before, it follows from the construction that the quasiregular maps $R_{j,\pm}:=\Psi\circ \pmb{R}_{j,\pm}\circ \Psi_{j,\pm}^{-1}:\widehat{\C}\to\widehat{\C}$ preserve the standard complex structure, and hence they are rational maps of $\widehat{\C}$.

The relation $\pmb{F}\circ\pmb{R}_{j,\pm}=\pmb{R}_{j,\mp}\circ\circ\eta$ (on $\overline{\D}$) and the definition of $\Psi_{j,\pm}$ imply that $\Psi_{j,\mp}\circ\eta\circ\Psi_{j,\pm}^{-1}$ preserve the standard complex structure, and hence they are M{\"o}bius maps. After possibly post-composing $\Psi_{j,\pm}$ with M{\"o}bius maps, we can assume that $\Psi_{j,\pm}$ fix $0, 1$, and $\infty$. Then, the M{\"o}bius maps $\Psi_{j,\mp}\circ\eta\circ\Psi_{j,\pm}^{-1}$ send $0$ to $\infty$, $\infty$ to $0$, and $1$ to $1$. It follows that $\Psi_{j,\mp}\circ\eta\circ\Psi_{j,\pm}^{-1}\equiv \eta$.
We set $\mathfrak{D}_{j,\pm}:=\Psi_{j,\pm}(\D)$, and $\Omega_{j,\pm}:=\Psi(\pmb{\Omega}_{j,\pm})= R_{j,\pm}(\mathfrak{D}_{j,\pm})$. Since $\D$ is mapped inside out by $\eta$, it follows that $\mathfrak{D}_{j,\pm}$ is mapped onto $\widehat{\C}\setminus\overline{\mathfrak{D}_{j,\mp}}$ by $\eta$. In particular, $\mathfrak{D}_{j,\pm}$ are Jordan domains such that $\eta(\partial\mathfrak{D}_{j,\pm})=\partial\mathfrak{D}_{j,\mp}$.

As $\pmb{R}_{j,\pm}$ are injective on $\overline{\D}$, we conclude that $R_{j,\pm}$ are injective on the closed Jordan disk $\overline{\mathfrak{D}_{j,\pm}}$. Moreover, we have
\begin{align*}
F &= \Psi\circ\pmb{F}\circ\Psi^{-1}\\
&=\Psi\circ(\pmb{R}_{j,\mp}\circ\eta\circ(\pmb{R}_{j,\pm}\vert_{\overline{\D}})^{-1})\circ\Psi^{-1}\\
&= (\Psi\circ\pmb{R}_{j,\mp}\circ\Psi_{j,\mp}^{-1})\circ(\Psi_{j,\mp}\circ\eta\circ\Psi_{j,\pm}^{-1})\circ(\Psi_{j,\pm}\circ(\pmb{R}_{j,\pm}\vert_{\overline{\D}})^{-1}\circ\Psi^{-1})\\
&=R_{j,\mp}\circ\eta\circ (R_{j,\pm}\vert_{\overline{\mathfrak{D}_{j,\pm}}})^{-1}
\end{align*} 
on $\overline{\Omega_{j,\pm}}$.

Finally, the last item follows from Lemma~\ref{deg_formula_lem}.
\end{proof}

Note that any polynomial in the principal hyperbolic component $\mathcal{H}_d$ is quasiconformally conjugate to a real-symmetric hyperbolic polynomial.

\begin{corollary}\label{main_hyp_comp_mating_class_cor}
Let $P\in\mathcal{H}_{d}$ and $\Gamma\in\mathrm{Teich}^\omega(\pmb{\Gamma_{n,p}})$. Then, there exist
\begin{enumerate}
\item a Jordan domain $\mathfrak{D}$ with $\eta(\partial\mathfrak{D})=\partial\mathfrak{D}$, and

\item a degree $d+1$ rational map $R$ of $\widehat{\C}$ that is injective on $\overline{\mathfrak{D}}$,
\end{enumerate}
such that the conformal mating $F$ of $P$ and $A_{\Gamma}^{\mathrm{fBS}}$ is given by 
$$
R\circ\eta\circ (R\vert_{\overline{\mathfrak{D}}})^{-1}:R(\overline{\mathfrak{D}})\to~\widehat{\C}.
$$
\end{corollary}

\subsection{Critical points of uniformizing rational maps}\label{unif_rat_crit_pnt_subsec}

We continue to use the notation of Proposition~\ref{mating_class_prop}. Our aim in this subsection is to give a complete description of the critical points of the uniformizing rational maps $R_\alpha$, $\alpha\in\mathcal{I}$, given by Proposition~\ref{mating_class_prop}. 

We recall the notation $\mathfrak{X}_P$ and $\mathfrak{X}_\Gamma$ from Definition~\ref{conf_mat_def}, and note that the dynamical plane of $F$ splits into the following invariant subsets:
$$
\mathcal{K}\equiv\mathcal{K}(F):=\mathfrak{X}_P(\cK(P))\quad \mathrm{and}\quad \cT\equiv\cT(F):=\mathfrak{X}_\Gamma(\D),
$$ 
which we term the \emph{non-escaping set} and the \emph{escaping/tiling set} of $F$ (respectively).
By definition, the action of $F$ on $\mathcal{K}$ (respectively, on $\cT$) is conformally conjugate to $P\vert_{\mathcal{K}(P)}$ (respectively, to $A_{\Gamma}^{\mathrm{fBS}}:\mathcal{D}_{\Gamma}\to\overline{\D}$). We also denote the common boundary of $\mathcal{K}$ and $\cT$ by $\Lambda\equiv\Lambda(F)$, and call it the \emph{limit set} of $F$. By construction, we have that $\Lambda=\mathfrak{X}_P(\mathcal{J}(P))$. Since $\mathfrak{X}_P$ is a homeomorphism, it follows that the limit set of $F$ is homeomorphic to the Julia set of $P$ (see Remark~\ref{mating_conj_rem} and Figures~\ref{basilica_mating_fig},~\ref{cubic_crit_fixed_mating_fig}).

As we shall see in Proposition~\ref{crit_pnt_r_sharp_prop}, the critical points of $R_\alpha$ can be organized into three categories:
\begin{itemize}
\item the critical points of $R_\alpha$ on $\partial\mathfrak{D}_\alpha$, which come from cusps of the group $\widehat{\Gamma}=\Gamma\rtimes \langle M_\omega\rangle$,
\item the critical points of $R_\alpha$ in $R_\alpha^{-1}(\cT)$, which are associated with the order $n$ elliptic element of $\Gamma$, when $n\geq 3$, and
\item the critical points of $R_\alpha$ in $R_\alpha^{-1}(\cK)$, which correspond to the critical points of $P$ in $\cK(P)$.
\end{itemize}

To make book-keeping easier, we denote the domain of $R_\alpha$ (this is a copy of the Riemann sphere) by $\widehat{\C}_\alpha$, and denote points in $\widehat{\C}_\alpha$ by $(z,\alpha)$. Note that $\mathfrak{D}_\alpha \subset\widehat{\C}_\alpha$.

Let us now consider the disjoint union 
$$
\mathfrak{U}\ :=\ \bigsqcup_{\alpha\in\mathcal{I}} \widehat{\C}_{\alpha}\ \cong\ \widehat{\C}\times\mathcal{I},
$$
and define the maps
$$
R:\ \mathfrak{U}\longrightarrow \widehat{\C},\quad (z,\alpha)\mapsto R_{\alpha}(z),
$$
and
$$
\eta_\ast\ : \mathfrak{U}\longrightarrow \mathfrak{U},\quad (z,\alpha)\mapsto (\eta(z),\kappa(\alpha)).
$$
Note that by Proposition~\ref{mating_class_prop}(part (4), the map $R$ is a branched covering of degree $np$, and $\eta_\ast$ is a homeomorphism. We also set 
$$
\mathfrak{D}:= \bigsqcup_{\alpha\in\mathcal{I}} \mathfrak{D}_\alpha.
$$

Note that the boundary of 
$$
\mathrm{Dom}(F)=\bigcup_{\alpha\in\mathcal{I}} R_\alpha(\overline{\mathfrak{D}_\alpha}) = R(\overline{\mathfrak{D}})
$$ 
meets $\Lambda$ at finitely many points, each of which is either fixed or $2-$periodic under $F$. We denote the set of these points by $S_F$, and note that 
$$
S_F=\mathfrak{X}_\Gamma(S_{\Gamma})
$$ 
(see Section~\ref{lami_model_subsec} for the definition of $S_{\Gamma}$). We denote the set of points in $S_F$ that do not disconnect $\Lambda(F)$ (or equivalently, are not cut-points of $\partial\mathrm{Dom}(F)$) by $S_F^{\mathrm{cusp}}$.
Finally, we set 
$$
\widetilde{S}_\alpha\ := (R_\alpha\vert_{\partial\mathfrak{D}_\alpha})^{-1}\left(S_F \cap \partial\Omega_\alpha\right).
$$ 
It is easy to see that $\partial\mathrm{Dom}(F)\setminus S_F$ is a union of finitely many non-singular analytic arcs. Indeed, $\partial\mathrm{Dom}(F)\setminus S_F$ is the image of finitely many hyperbolic geodesics of $\D$ under $\mathfrak{X}_\Gamma\circ\xi$ (where $\xi(w)=w^n$). Moreover, since $A_{\Gamma}^{\mathrm{fBS}}$ admits an analytic continuation to a neighborhood of $\mathcal{D}_{\Gamma}\setminus S_{\Gamma}$, it follows that $F$ admits an analytic continuation to a neighborhood of $\mathrm{Dom}(F)\setminus S_F$.

Recall that the global degree of the rational map $\pmb{R}_\alpha$, $\alpha\in\mathcal{I}$, is denoted by $d_\alpha$.

\begin{lemma}
\noindent\begin{enumerate}
\item $F(S_F)=S_F$.

\item Each point of $S_F^{\mathrm{cusp}}$ is a critical value of some $R_\alpha$ with an associated critical point on $\partial\mathfrak{D}_\alpha$.

\item $\eta(\widetilde{S}_\alpha)=\widetilde{S}_{\kappa(\alpha)}$.

\item $d_\alpha=n q_\alpha$, where $q_\alpha$ is the number of components of $\mathbb{S}^1\setminus\mathcal{A}_p$ on $\partial\mathcal{G}_\alpha\cap\mathbb{S}^1$ (boundary taken in $\overline{\D}$).
\end{enumerate}
\label{S_F_lem}
\end{lemma}
\begin{proof}
1) Note that $A_{\Gamma}^{\mathrm{fBS}}$ carries the set $S_{\Gamma}$ onto itself. Thanks to the semi-conjugacy between $A_{\Gamma}^{\mathrm{fBS}}$ and $F$ (via $\mathfrak{X}_\Gamma$), we conclude that $F(S_F)=S_F$.

2) Since $A_{\Gamma}^{\mathrm{fBS}}$ does not admit an analytic continuation in a neighborhood of any point of $S_{\Gamma}$, the map $F$ does not admit an analytic continuation in a neighborhood of any point of $S_F$. Suppose that $x\in S_F^{\mathrm{cusp}}$ lies on the boundary of $\Omega_\alpha$ only.
The fact that $F\vert_{\overline{\Omega_\alpha}}$ is given by $R_{\kappa(\alpha)}\circ\eta\circ (R_\alpha\vert_{\overline{\mathfrak{D}_\alpha}})^{-1}$ implies that $(R_\alpha\vert_{\overline{\mathfrak{D}_\alpha}})^{-1}$ does not extend complex-analytically to a neighborhood of $x$. This forces $x$ to be a critical value of $R_\alpha$ with a corresponding critical point on $\partial\mathfrak{D}_\alpha$. 

3) This follows from item (1) and the relation $F\vert_{\partial\Omega_{\alpha}}\equiv R_{\kappa(\alpha)}\circ\eta\circ(R_{\alpha}\vert_{\partial\mathfrak{D}_{\alpha}})^{-1}$.

4) The relation $F\vert_{\Omega_{\alpha}}\equiv R_{\kappa(\alpha)}\circ\eta\circ(R_{\alpha}\vert_{\mathfrak{D}_{\alpha}})^{-1}$ implies that 
$$
F:F^{-1}(\Omega_{\kappa(\alpha)})\cap\Omega_\alpha\longrightarrow \Omega_{\kappa(\alpha)}
$$ 
is a branched covering of degree $d_\alpha-1$.

Let $q_\alpha$ be the number of components of $\mathbb{S}^1\setminus\mathcal{A}_p$ on $\partial\mathcal{G}_\alpha\cap\mathbb{S}^1$. Note that under the map $m_{d}$ (which models the dynamics of $P$ on $\mathcal{J}(P)$) each arc of $\mathbb{S}^1\setminus\mathcal{A}_p$ is wrapped onto the whole circle $(n-1)$ times and onto the complement of the closure of its complex conjugate arc once. Hence, $m_{d}(\partial\mathcal{G}_\alpha\cap\mathbb{S}^1)$ covers $\partial\mathcal{G}_{\kappa(\alpha)}\cap\mathbb{S}^1$ exactly 
$$
(n-1)q_\alpha+(q_\alpha-1)=n q_\alpha-1
$$ 
times. It follows that $\Lambda(F)\cap\Omega_\alpha$ covers $\Lambda(F)\cap\Omega_{\kappa(\alpha)}$ exactly $(n q_\alpha-1)$ times under the map $F$. Since the limit set of $F$ is completely invariant, we conclude that $F:F^{-1}(\Omega_{\kappa(\alpha)})\cap\Omega_\alpha\longrightarrow \Omega_{\kappa(\alpha)}$ is a degree $(nq_\alpha-1)$ branched covering. Therefore, $d_\alpha=n q_\alpha$.
\end{proof}

\begin{lemma}\label{crit_pnt_r_lem}
\noindent\begin{enumerate}
\item We have $\mathrm{crit}(R\vert_{\mathfrak{D}})=\emptyset$, and $(R\vert_{\partial\mathfrak{D}})^{-1}(S_F^{\mathrm{cusp}})\subset \mathrm{crit}(R)\cap \partial\mathfrak{D}$. Further, the points of $(R_\alpha\vert_{\partial\mathfrak{D}_\alpha})^{-1}(S_F^{\mathrm{cusp}})$ correspond bijectively to the points of $\mathrm{cusps}(\mathcal{G}_\alpha)$. 

\item $F(S_F^{\mathrm{cusp}}\cap\partial\Omega_\alpha)=S_F^{\mathrm{cusp}}\cap\partial\Omega_{\kappa(\alpha)}$. The involution $\eta$ carries $(R_\alpha\vert_{\partial\mathfrak{D}_\alpha})^{-1}(S_F^{\mathrm{cusp}})$ onto $(R_{\kappa(\alpha)}\vert_{\partial\mathfrak{D}_{\kappa(\alpha)}})^{-1}(S_F^{\mathrm{cusp}})$.

\item The critical points of $R_\alpha$ in $\widehat{\C}\setminus\overline{\mathfrak{D}_\alpha}$ correspond bijectively to the critical points of $F$ in $\Omega_{\kappa(\alpha)}$ (counted with multiplicities). In particular, $R$ has $p$ distinct critical points, each of multiplicity $n-1$, in $R^{-1}(\cT)\setminus\overline{\mathfrak{D}}$, and all these critical points are mapped by $R$ to the same point in $\cT$. On the other hand, $R$ has $d-1=np-2$ critical points in $R^{-1}(\cK)\setminus\overline{\mathfrak{D}}$.
\end{enumerate}
\end{lemma}
\begin{proof}
1) The first statement follows from injectivity of $R_\alpha\vert_{\mathfrak{D}_\alpha}$.
The proof of part (2) of Lemma~\ref{S_F_lem} shows that the pre-images of the points of $S_F^{\mathrm{cusp}}\cap\partial\Omega_\alpha$ under $R_\alpha\vert_{\partial\mathfrak{D}_\alpha}$ are critical points of $R_\alpha$. The statement that the points of $(R_\alpha\vert_{\partial\mathfrak{D}_\alpha})^{-1}(S_F^{\mathrm{cusp}})$ correspond bijectively to the points of $\mathrm{cusps}(\mathcal{G}_\alpha)$ is a trivial consequence of Definition~\ref{gap_cusp_def}.

2) The real-symmetry property of the lamination $\mathcal{L}_P$ implies that if $\theta\in\mathrm{cusps}(\mathcal{G}_\alpha)$, then $m_d(\theta)=-\theta\in\mathrm{cusps}(\mathcal{G}_{\kappa(\alpha)})$ (see Section~\ref{lami_model_subsec}). Under the mating semi-conjugacy $\mathfrak{X}_P$, this translates to the fact that if $x\in S_F^{\mathrm{cusp}}\cap\partial\Omega_\alpha$, then $F(x)\in S_F^{\mathrm{cusp}}\cap\partial\Omega_{\kappa(\alpha)}$. In light of the relation
$F\vert_{\partial\Omega_\alpha}\equiv R_{\kappa(\alpha)}\circ\eta\circ (R_\alpha\vert_{\partial\mathfrak{D}_\alpha})^{-1}$, we conclude that $\eta$ sends $(R_\alpha\vert_{\partial\mathfrak{D}_\alpha})^{-1}(S_F^{\mathrm{cusp}}\cap\partial\Omega_\alpha)$ onto $(R_{\kappa(\alpha)}\vert_{\partial\mathfrak{D}_{\kappa(\alpha)}})^{-1}(S_F^{\mathrm{cusp}}\cap\partial\Omega_{\kappa(\alpha)})$.

3) The first statement follows from the relation $F\vert_{\Omega_{\kappa(\alpha)}}\equiv R_{\alpha}\circ\eta\circ(R_{\kappa(\alpha)}\vert_{\mathfrak{D}_{\kappa(\alpha)}})^{-1}$ (recall that $\eta$ carries $\mathfrak{D}_{\kappa(\alpha)}$ onto $\widehat{\C}\setminus\overline{\mathfrak{D}_\alpha}$). The remaining claims are consequences of the facts that $F$ has $d-1=np-2$ critical points in $\mathcal{K}$ (coming from the $d-1$ critical points of $P$ in $\cK(P)$) and $p$ critical points, each of multiplicity $n-1$, in $\cT$ (coming from the $p$ critical points of $A_{\Gamma}^{\mathrm{fBS}}$ in $\mathcal{D}_{\Gamma}$). Moreover, $F$ maps all the $p(n-1)$ critical points in $\cT$ to the same critical value since $A_{\Gamma}^{\mathrm{fBS}}$ sends all of its $p(n-1)$ critical points in $\mathcal{D}_{\Gamma}$ to the origin.
\end{proof}

We will conclude this section with a refined version of part (1) of Lemma~\ref{crit_pnt_r_lem}. The proof will go through an intermediate lemma about the structure of the lamination $\mathcal{L}_P$.

\begin{lemma}\label{lami_no_poly_lem}
The lamination $\mathcal{L}_P$ contains no polygon; i.e., each equivalence class of $\mathcal{L}_P$ contains at most two elements.
\end{lemma}
\begin{proof}
By Lemma~\ref{crit_pnt_r_lem}, we have that 
$$
2d_\alpha-2=\#\ \mathrm{crit}(R_\alpha)\geq \#\ \mathrm{cusps}(\mathcal{G}_\alpha) + \#\ \mathrm{crit}(F\vert_{\Omega_{\kappa(\alpha)}}),
$$
for each $\alpha\in\mathcal{I}$. Summing this inequality over $\alpha\in\mathcal{I}$, we get that
$$
\sum_{\alpha\in\mathcal{I}} (2d_\alpha-2)=\sum_{\alpha\in\mathcal{I}} \#\  \mathrm{crit}(R_\alpha)\geq \sum_{\alpha\in\mathcal{I}} \left( \#\ \mathrm{cusps}(\mathcal{G}_\alpha) + \#\ \mathrm{crit}(F\vert_{\Omega_{\kappa(\alpha)}})\right).
$$
Taking into account the relation $d_\alpha=n q_\alpha$ (where $q_\alpha$ is the number of components of $\mathbb{S}^1\setminus\mathcal{A}_p$ on $\partial\mathcal{G}_\alpha\cap\mathbb{S}^1$) and the facts that $\mathbb{S}^1\setminus \mathcal{A}_p$ has $p$ components, $\mathcal{L}_P$ has $\#\ \mathcal{I}$ many gaps, and $F$ has a total of 
$$
(np-2)+p(n-1)=2np-p-2
$$ 
critical points in $\Int{\mathrm{Dom}(F)}$ (since $P$ has $np-2$ critical points in $\mathcal{K}(P)$ and $A_{\Gamma}^{\mathrm{fBS}}$ has $p(n-1)$ critical points in $\mathcal{D}_{\Gamma}$), we can rewrite the above inequality as
$$
2np - 2\cdot \#\ \mathcal{I} \geq \sum_{\alpha\in\mathcal{I}} \#\ \mathrm{cusps}(\mathcal{G}_\alpha) + (2np-p-2).
$$
Thus,
\begin{equation}
p+2 \ \geq \ 2\cdot \#\ \mathcal{I} + \sum_{\alpha\in\mathcal{I}} \#\ \mathrm{cusps}(\mathcal{G}_\alpha).
\label{lami_ineq}
\end{equation}
We claim that Inequality~\eqref{lami_ineq} is satisfied only if the lamination $\mathcal{L}_P$ contains no polygon. To see this, first note that if $\#\ \mathcal{I}=1$; i.e., when the lamination is empty, then $\mathrm{cusps}(\mathcal{G}_1)$ consists of $p$ points, and hence the two sides of the inequality coincide. Now, the introduction of a $k-$gon in the lamination (or a leaf, when $k=2$) adds $k-1$ gaps and kills $k$ cusps. For $k>2$, this procedure increases the right side of Inequality~\eqref{lami_ineq} by $2(k-1)-k=k-2>0$. Clearly, this violates the inequality, which proves that  each equivalence class of $\mathcal{L}_P$ contains at most two elements.
\end{proof}

In the next proposition, we record the locations of all the critical points of $R_\alpha$.

\begin{proposition}\label{crit_pnt_r_sharp_prop}
\noindent\begin{enumerate}
\item $R$ has no critical points in $\mathfrak{D}$.

\item $\mathrm{crit}(R)\cap \partial\mathfrak{D}=(R\vert_{\partial\mathfrak{D}})^{-1}(S_F^{\mathrm{cusp}})$.
Consequently, $\eta$ maps $\mathrm{crit}(R_\alpha)\cap \partial\mathfrak{D}_\alpha$ bijectively to $\mathrm{crit}(R_{\kappa(\alpha)})\cap \partial\mathfrak{D}_{\kappa(\alpha)}$.

\item $R$ has $p$ distinct critical points, each of multiplicity $n-1$, in $R^{-1}(\cT)\setminus\overline{\mathfrak{D}}$, and all these critical points are mapped by $R$ to the same point in $\cT$. On the other hand, $R$ has $d-1=np-2$ critical points in $R^{-1}(\cK)\setminus\overline{\mathfrak{D}}$.
\end{enumerate}
\end{proposition}
\begin{proof}
The first and third items follow from Lemma~\ref{crit_pnt_r_lem}. It remains to prove the second part.

Since each equivalence class of $\mathcal{L}_P$ is a point or a leaf, it is easy to see that 
$$
2\cdot \#\ \mathcal{I} + \sum_{\alpha\in\mathcal{I}} \#\ \mathrm{cusps}(\mathcal{G}_\alpha)= p+2.
$$
This implies, by the proof of Lemma~\ref{lami_no_poly_lem}, that
\begin{equation}
\sum_{\alpha\in\mathcal{I}} \left(\#\ \left(\mathrm{crit}\left(R_\alpha\right)\cap \partial\mathfrak{D}_\alpha\right) - \#\ \mathrm{cusps}\left(\mathcal{G}_\alpha\right)\right)=0.
\label{crit_pnt_r_eqn}
\end{equation}
By part (1) of Lemma~\ref{crit_pnt_r_lem}, we have that $\#\ (\mathrm{crit}(R_\alpha)\cap \partial\mathfrak{D}_\alpha) - \#\ \mathrm{cusps}(\mathcal{G}_\alpha)\geq 0$ for each $\alpha\in\mathcal{I}$, and hence by Relation~\eqref{crit_pnt_r_eqn}, $\#\ (\mathrm{crit}(R_\alpha)\cap \partial\mathfrak{D}_\alpha) = \#\ \mathrm{cusps}(\mathcal{G}_\alpha)$ for each $\alpha\in\mathcal{I}$. Since the points of $(R_\alpha\vert_{\partial\mathfrak{D}_\alpha})^{-1}(S_F^{\mathrm{cusp}})$ correspond bijectively to the points of $\mathrm{cusps}(\mathcal{G}_\alpha)$, we conclude that $(R_\alpha\vert_{\partial\mathfrak{D}_\alpha})^{-1}(S_F^{\mathrm{cusp}}\cap\partial\Omega_\alpha)= \mathrm{crit}(R_\alpha)\cap \partial\mathfrak{D}_\alpha$.
The second statement now follows from part (2) of Lemma~\ref{crit_pnt_r_lem}.
\end{proof}

\begin{remark}\label{examples_deg_crit_pnt_rem}
For the conformal mating $F$ considered in Section~\ref{quad_example_subsubsec}, the structure of the lamination $\mathcal{L}(P)$ was discussed in Remark~\ref{examples_lamination_rem}. In particular, the lamination has two gaps. By Lemma~\ref{S_F_lem}, the rational uniformizing map $R_1$ for one of the components of $\Int{\mathrm{Dom}(F)}$ is quadratic, while the other uniformizing rational map $R_2$ is a M{\"o}bius map. The critical points of $R_1$ correspond to the unique point of $S_F^{\mathrm{cusp}}$ (see Figure~\ref{basilica_mating_fig}) and the unique critical point of $F$ in $\cK$.

The domain of definition of the conformal mating $F$ of Section~\ref{cubic_example_subsubsec} also has two interior components (see Remark~\ref{examples_lamination_rem} for a description of the corresponding lamination $\mathcal{L}(P)$). By Lemma~\ref{S_F_lem}, the rational uniformizing maps $R_{1,\pm}$ for the two components of $\Int{\mathrm{Dom}(F)}$ are quadratic. The critical points of $R_{1,\pm}$ correspond to the two points of $S_F^{\mathrm{cusp}}$ (see Figure~\ref{cubic_crit_fixed_mating_fig}) and the two critical points of $F$ in $\cK$.
\end{remark}

The next corollary follows from Proposition~\ref{crit_pnt_r_sharp_prop} and the observation that if $\mathcal{L}_P=\emptyset$, then $\Int{\mathrm{Dom}(F)}$ is a Jordan domain and hence no point of $S_F$ disconnects $\partial\mathrm{Dom}(F)$.

\begin{corollary}\label{crit_pnt_r_cor}
If $\mathcal{L}_P=\emptyset$, then $\Omega:=\Int{\mathrm{Dom}(F)}$ is connected, the degree of $R:=R_1$ is $d+1=np$, and the set of critical points of $R$ on $\partial\mathfrak{D}$ is given by $\widetilde{S}:=\widetilde{S}_1=(R\vert_{\partial\mathfrak{D}})^{-1}(S_F)$. In particular, $R$ has $p$ critical points on $\partial\mathfrak{D}$, $d-1=np-2$ critical points in $R^{-1}(\mathcal{K})\setminus\overline{\mathfrak{D}}$, and $p$ distinct critical points, each of multiplicity $n-1$, in $R^{-1}(\cT)\setminus\overline{\mathfrak{D}}$. All the $p(n-1)$ critical points of $R$ in $R^{-1}(\cT)\setminus\overline{\mathfrak{D}}$ are mapped to the same critical value. 

Moreover, depending on whether $p$ is even/odd, exactly two/one points of $\widetilde{S}$ are fixed by $\eta$ and the others form $2-$cycles under $\eta$. 
\end{corollary}
\begin{proof}
We only need to justify the last statement. To this end, observe that depending on whether $p$ is even/odd, exactly two/one points of $S_F$ are fixed by $F$, and the others form $2-$cycles under $F$. Since $F\equiv R\circ\eta\circ (R\vert_{\overline{\mathfrak{D}}})^{-1}$, it follows that depending on whether $p$ is even/odd, exactly two/one points of $\widetilde{S}$ are fixed by $\eta$ and the others form $2-$cycles under $\eta$. 
\end{proof}

\begin{remark}
Lemma~\ref{lami_no_poly_lem} and Proposition~\ref{crit_pnt_r_sharp_prop} can also be proved by looking at the real-symmetric map $\pmb{F}$ that $F$ is quasiconformally conjugate to. For Lemma~\ref{lami_no_poly_lem}, note that the only singularities on the boundaries of the Jordan quadrature domains $\pmb{\Omega}_\alpha$ are (inward) conformal cusps (cf. \cite{Sak91}), and hence more than two such quadrature domains cannot touch at a point. 

For Proposition~\ref{crit_pnt_r_sharp_prop}, first observe that if $y\in\partial\mathfrak{D}_\alpha\setminus (R_\alpha\vert_{\partial\mathfrak{D}_\alpha})^{-1}(S_F\cap\partial\Omega_\alpha)$ were a critical point of $R_\alpha$, then $(R_\alpha\vert_{\overline{\mathfrak{D}_\alpha}})^{-1}$ would not extend analytically to a neighborhood of $R_\alpha(y)\in\overline{\Omega_\alpha}\setminus S_F$, which contradicts the fact that $R_\alpha$ admits an analytic continuation to a neighborhood of $\overline{\Omega_\alpha}\setminus S_F$. Finally, if $R_\alpha$ had a critical point $y\in(R_\alpha\vert_{\partial\mathfrak{D}_\alpha})^{-1}(S_F \setminus S_F^{\mathrm{cusp}})$, then $\pmb{R}_\alpha(\psi_\alpha^{-1}(y))$ would be a conformal cusp of $\partial\pmb{\Omega}_\alpha$, while $\pmb{R}_\alpha(\psi_\alpha^{-1}(y))$ would also be a touching point of $\partial\pmb{\Omega}_\alpha$ and $\partial\pmb{\Omega}_\beta$ for some $\beta\neq\alpha$; which is impossible.
\end{remark}

\section{Correspondences associated with conformal matings}\label{corr_sec}

We will now use the algebraic representation of conformal matings $F$ in terms of uniformizing rational maps $R$ to define algebraic correspondences $\mathfrak{C}$, and study their dynamical properties to complete the proof of Theorem~\ref{corr_mating_intro_thm}, thus answering the second part of Question~\ref{qn_main}. It turns out that the dynamics of $\mathfrak{C}$ is intimately related to that of $F$: the mating structure in the $F-$plane can be lifted via $R$ to obtain the desired mating structure in the $\mathfrak{C}-$plane.

\subsection{The case of principal hyperbolic components}\label{special_corr_subsec}

Let $n,p$ be positive integers with $np\geq 3$. Set $d:=np-1$. Suppose that $(\rho:\pmb{\Gamma_{n,p}}\to\Gamma)\in\mathrm{Teich}^\omega(\pmb{\Gamma_{n,p}})$, $P\in\mathcal{H}_{np-1}$, and $F:\overline{\Omega}\to\widehat{\C}$ be the conformal mating of $A_{\Gamma}^{\mathrm{fBS}}$ and $P$. Further, let $R, \mathfrak{D}$ be as in Corollary~\ref{main_hyp_comp_mating_class_cor}. Finally, we set $\widetilde{S}:=\widetilde{S}_1=(R\vert_{\partial\mathfrak{D}})^{-1}(S_F)$.

We will define a holomorphic correspondence $\mathfrak{C}\subset\widehat{\C}\times\widehat{\C}$ of bi-degree $d$:$d$ as (cf. \cite{DS06}):
\begin{equation}
(z,w)\in\mathfrak{C}\iff \frac{R(w)-R(\eta(z))}{w-\eta(z)}=0.
\label{corr_eqn}
\end{equation}
The correspondence $\mathfrak{C}$ can be regarded as a  multi-valued map with holomorphic local branches. The forward branches of this multi-valued map are given by $z\mapsto w=R^{-1}(R(\eta(z)))$, $w\neq \eta(z)$, and the backward branches are given by $w\mapsto z=\eta(R^{-1}(R(w)))$, $z\neq \eta(w)$. Note also that the correspondence $\mathfrak{C}$ is \emph{reversible}; i.e., its forward branches are conjugate to its backward branches via $\eta$.

The following observations show that $\mathfrak{C}$ is obtained by lifting $F$ and its appropriate backward branches via the rational map $R$. 
\begin{itemize}
\item Fix $z\in\overline{\mathfrak{D}}$. Then, $F(R(z))=R(\eta(z))$, and hence,
\begin{equation}
(z,w)\in\mathfrak{C} \iff R(w)=R(\eta(z))=F(R(z)).
\label{f_lift_eqn}
\end{equation}

\item Now fix $z\in\widehat{\C}\setminus\overline{\mathfrak{D}}$. Then, $F(R(\eta(z)))=R(z)$; i.e.,
\begin{equation}
(z,w)\in\mathfrak{C} \iff R(w)=R(\eta(z))=F^{-1}(R(z)),
\label{f_inv_lift_eqn}
\end{equation}
where $F^{-1}$ is a suitable backward branch of $F$. 
\end{itemize}

\subsubsection{Dynamical partition for $\mathfrak{C}$}\label{inv_partition_corr_1_subsubsec}

The invariant partition of the dynamical plane of $F$, given by $\mathcal{K}$ and $\cT$, can be pulled back by $R$ to produce an invariant partition of the dynamical plane of the correspondence $\mathfrak{C}$. More precisely, we set
$$
\widetilde{\mathcal{K}}:=R^{-1}(\mathcal{K}),\quad \widetilde{\cT}:= R^{-1}(\cT).
$$
We call these sets the \emph{non-escaping set} and the \emph{tiling set} of the correspondence $\mathfrak{C}$.
Note that the common boundary of $\widetilde{\mathcal{K}}$ and $\widetilde{\mathcal{T}}$ is given by 
$$
\widetilde{\Lambda}:=R^{-1}(\Lambda).
$$
We call $\widetilde{\Lambda}$ the \emph{limit set} of $\mathfrak{C}$. (See Figure~\ref{corr_fig} for an illustration.)

\begin{proposition}\label{corr_partition_1_prop}
\noindent\begin{enumerate}
\item $\eta(\widetilde{\cT})=\widetilde{\cT}$, and $\eta(\widetilde{\cK})=\widetilde{\cK}$.

\item Let $(z,w)\in\mathfrak{C}$. Then $z\in\widetilde{\cT}$ (respectively, $z\in\widetilde{\mathcal{K}}$) if and only if $w\in\widetilde{\cT}$ (respectively, $w\in\widetilde{\mathcal{K}}$).
\end{enumerate}
\end{proposition}
\begin{proof}
1) It suffices to show that $\eta(\widetilde{\mathcal{K}})=\widetilde{\mathcal{K}}$. 

Let us fist assume that $z\in\overline{\mathfrak{D}}\cap\widetilde{\mathcal{K}}$. Then, $R(z)\in\mathcal{K}$ and $R(\eta(z))=F(R(z))$. As $\mathcal{K}$ is invariant under $F$, it follows that $R(\eta(z))\in\mathcal{K}$. We conclude that $\eta(z)\in\widetilde{\mathcal{K}}$.

Next let $z\in\widetilde{\mathcal{K}}\setminus\overline{\mathfrak{D}}$. Then, $R(z)\in\mathcal{K}$ and $F(R(\eta(z)))=R(z)$. As $\mathcal{K}$ is backward invariant under $F$, it follows that $R(\eta(z))\in\mathcal{K}$. We conclude that $\eta(z)\in\widetilde{\mathcal{K}}$.

2) It suffices to show that if $z\in\widetilde{\cT}$ (respectively, if $z\in\widetilde{\mathcal{K}}$), then $w\in\widetilde{\cT}$ (respectively, $w\in\widetilde{\mathcal{K}}$). 

To this end, first suppose that $z\in\overline{\mathfrak{D}}$, which implies that $R(w)=R(\eta(z))=F(R(z))$. Now let $z\in\widetilde{\cT}$ (respectively, $z\in\widetilde{\mathcal{K}}$). The $F-$invariance of $\cT$ (respectively, $\mathcal{K}$) implies that $R(w)\in\cT$ (respectively, $R(w)\in\mathcal{K}$). Hence, $w\in\widetilde{\cT}$ (respectively, $w\in\widetilde{\mathcal{K}}$). 

Next suppose that $z\in\widehat{\C}\setminus\overline{\mathfrak{D}}$, which implies that $F(R(w))=R(z)$. Now let $z\in\widetilde{\cT}$ (respectively, $z\in\widetilde{\mathcal{K}}$). The backward invariance of $\cT$ (respectively, $\mathcal{K}$) under $F$ implies that $R(w)\in\cT$ (respectively, $R(w)\in\mathcal{K}$). Hence, $w\in\widetilde{\cT}$ (respectively, $w\in\widetilde{\mathcal{K}}$). 
\end{proof}

By Corollary~\ref{crit_pnt_r_cor}, there is a unique critical value of $R$ in $\cT$ when $n\geq 2$ and no critical value when $n=1$. Moreover, the fiber of this critical value (under $R$) consists of $p$ distinct points, each of which is a critical point of multiplicity $n-1$. As $\cT$ is simply connected, a routine application of the Riemann-Hurwitz formula on the branched covering $R:\widetilde{\cT}\to\cT$ shows that $\widetilde{\cT}$ is the union of $p$ disjoint open topological disks $U_0,\cdots,U_{p-1}$. We can enumerate these components so that $\eta(U_i)=U_{p-1-i}$, $i\in\Z/p\Z$.
Moreover, each $U_i$ contains a unique critical point (of multiplicity $n-1$) of $R$ and maps onto $\cT$ with degree $n$. We now study the topology of $\Cl{\widetilde{\cT}}$ (the topological closure of $\widetilde{\cT}$ in $\widehat{\C}$).

\begin{figure}[h!]
\captionsetup{width=0.96\linewidth}
\begin{tikzpicture}
\node[anchor=south west,inner sep=0] at (0,0) {\includegraphics[width=1\textwidth]{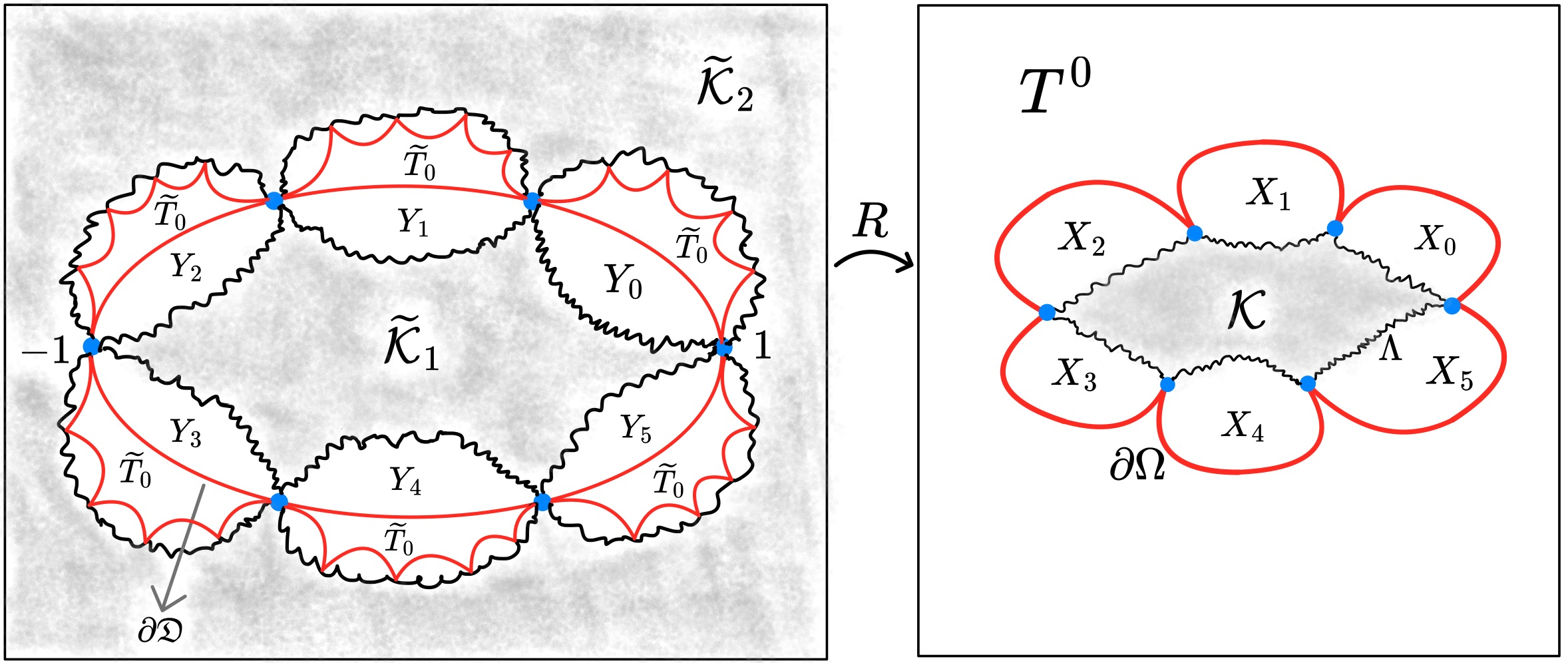}}; 
\end{tikzpicture}
\caption{The dynamical planes of $F$ and $\mathfrak{C}$ are displayed, where $P\in\mathcal{H}_5$ and $\Gamma\in\mathrm{Teich}(\pmb{\Gamma_{1,6}})$.}
\label{corr_fig}
\end{figure}
\noindent\textbf{Figure~\ref{corr_fig}.} Right: The dynamical plane of the conformal mating $F$ of some $P\in\mathcal{H}_{5}$ and some $\Gamma\in\mathrm{Teich}(\pmb{\Gamma_{1,6}})$ is depicted. The domain of definition $\overline{\Omega}$ of the mating $F$ is the bounded region enclosed by the red Jordan curve, and its exterior is $T^0$. The blue points on $\partial T^0$ comprise $S_F$. The non-escaping set $\cK$ of $F$ is shaded in gray. The components $X_1,\cdots,X_6$ of $\overline{\mathcal{T}}\setminus T^0=\overline{\Omega}\setminus\Int{\cK}$ are marked. Left: The dynamical plane of $\mathfrak{C}$ is shown. The domain $\mathfrak{D}$ is the bounded Jordan disk enclosed by the $\eta-$invariant, piecewise analytic (oval shaped) red curve marked as $\partial\mathfrak{D}$. The six blue points marked on $\partial\mathfrak{D}$ constitute $\widetilde{S}$. The part of the non-escaping set of $\mathfrak{C}$ inside $\overline{\mathfrak{D}}$ (respectively, outside $\mathfrak{D}$) is shaded in gray and marked as $\widetilde{\mathcal{K}}_1$ (respectively, as $\widetilde{\mathcal{K}}_2$); it is carried by $R$ univalently (respectively, as a $5:1$ branched cover) onto $\mathcal{K}$. The sets $\widetilde{\mathcal{K}}_1$ and $\widetilde{\mathcal{K}}_2$ intersect in $\widetilde{S}$. On the other hand, the tiling set of $\mathfrak{C}$ is the union of six Jordan domains, each of which is bounded by a black curve. These black curves constitute the limit set $\widetilde{\Lambda}=R^{-1}(\Lambda)$ of the correspondence $\mathfrak{C}$. The six components of $R^{-1}(T^0)$, one in each component of $\widetilde{\cT}$, are marked as $\widetilde{T^0}$. Finally, the closed topological disks $Y_i=(R\vert_{\overline{\mathfrak{D}}})^{-1}(X_i)$ are marked.

\begin{lemma}\label{lifted_tiling_top_lem}
\noindent\begin{enumerate}
\item $\displaystyle\Cl{\widetilde{\cT}}=\bigcup_{i=0}^{p-1} \overline{U_i}$ is connected, where each $U_i$ is a Jordan domain that is mapped with degree $n$ onto $\mathcal{T}$ by $R$,
\item  $\overline{U_i}\cap \overline{U}_{i+1}$ is a single point belonging to $\widetilde{S}$, and
\item $\overline{U_i}\cap \overline{U_j}=\emptyset$ if $\vert j-i\vert\neq1$.
\end{enumerate}
Here, $i,j\in\Z/p\Z$.
\end{lemma}
\begin{proof}
We set $T^0:=\widehat{\C}\setminus\overline{\Omega}$. Note that $\overline{\mathcal{T}}\setminus T^0$ is a connected set. In fact, it is the union of $p$ closed topological disks (i.e., closures of Jordan domains) $X_0,\cdots,X_{p-1}$ such that 
\begin{itemize}
\item $\partial X_i\cap\partial T^0 = R(\overline{U_i}\cap \partial\mathfrak{D})$,
\item $X_i\cap X_{i+1}$ is a single point belonging to $S_F$, and
\item $X_i\cap X_j=\emptyset$ if $\vert j-i\vert\neq1$, where $i,j\in\Z/p\Z$.
\end{itemize}
(See Figure~\ref{corr_fig} (right).)
The inverse branch $(R\vert_{\overline{\mathfrak{D}}})^{-1}$ carries each $X_i$ to a closed topological disk $Y_i$ in $\overline{\mathfrak{D}}$. The boundary of each of these pulled back closed disks $Y_i$ contains (the closure of) a component of $\partial\mathfrak{D}\setminus\widetilde{S}$. Hence the boundary of $(R\vert_{\overline{\mathfrak{D}}})^{-1}(\overline{\mathcal{T}}\setminus T^0)=\cup_{i=1}^p Y_i$ contains all of $\partial\mathfrak{D}$. In fact, $(R\vert_{\overline{\mathfrak{D}}})^{-1}(\overline{\mathcal{T}}\setminus T^0)$ contains a relative neighborhood in $\overline{\mathfrak{D}}$ of each point in $\partial\mathfrak{D}\setminus\widetilde{S}$ (see Figure~\ref{corr_fig} (left)). 

Since $\eta(\widetilde{\mathcal{T}})=\widetilde{\mathcal{T}}$ and $\Cl{\widetilde{\mathcal{T}}}\cap\overline{\mathfrak{D}}= (R\vert_{\overline{\mathfrak{D}}})^{-1}(\overline{\mathcal{T}}\setminus T^0)$, we conclude that
$$
\Cl{\widetilde{\mathcal{T}}}\ =\ (R\vert_{\overline{\mathfrak{D}}})^{-1}(\overline{\mathcal{T}}\setminus T^0)\ \bigcup\ \eta\left((R\vert_{\overline{\mathfrak{D}}})^{-1}(\overline{\mathcal{T}}\setminus T^0)\right).
$$
Finally, that fact that $\widetilde{S}$ is $\eta-$invariant implies that
\begin{itemize}
\item  $\overline{U_i}\cap \overline{U}_{i+1}$ is a single point belonging to $\widetilde{S}$,
\item $\overline{U_i}\cap \overline{U_j}=\emptyset$ if $\vert j-i\vert\neq1$, where $i,j\in\Z/p\Z$, and
\item $\Cl{\widetilde{\mathcal{T}}}=\bigcup_{i=1}^p \overline{U_i}$ is connected.
\end{itemize}
(See Figure~\ref{corr_fig} (left).) 
\end{proof}

The main result of this subsection is the following theorem, which ties up our framework of mating factor Bowen-Series maps of genus zero orbifolds with polynomials in principal hyperbolic components with the Bullett-Penrose mating phenomenon. We recall the notation $\widehat{\Gamma}=\langle\Gamma,M_\omega\rangle$ from Section~\ref{factor_bs_gen_subsec}.

\begin{theorem}\label{corr_mating_thm_1}
The correspondence $\mathfrak{C}$ defined by Equation~\eqref{corr_eqn} is a mating of $P$ and $\Sigma:=\faktor{\D}{\widehat{\Gamma}}$ in the following sense.
\begin{enumerate}
\item The dynamics of $\mathfrak{C}$ on $\widetilde{\cT}$ is equivalent to the action of a group 
$$
\langle\eta\rangle\ast\langle\tau\rangle\ \cong\ \Z/2\Z\ast\Z/(np)\Z
$$ 
of conformal automorphisms of $\widetilde{\cT}$. Here, $\tau$ is a conformal automorphism of $\widetilde{\cT}$ of order $np$ such that $\tau^p$ induces an order $n$ conformal automorphism on each component of $\widetilde{\cT}$.

Moreover, the above group action is properly discontinuous, and the quotient orbifold $\faktor{\widetilde{\cT}}{\mathfrak{C}}$ is biholomorphic to $\Sigma$.

\item The correspondence $\mathfrak{C}$ has a forward branch carrying $\widetilde{\mathcal{K}}\cap\overline{\mathfrak{D}}$ onto itself with degree $np-1$. This branch is topologically conjugate to $P:\mathcal{K}(P)\to \mathcal{K}(P)$, and the restriction of the branch to $\Int{(\widetilde{\mathcal{K}}\cap\overline{\mathfrak{D}})}$ is conformally conjugate to $P\vert_{\Int{\cK(P)}}$. On the other hand, $\mathfrak{C}$ has a backward branch carrying $\widetilde{\mathcal{K}}\setminus\mathfrak{D}$ onto itself with degree $np-1$, and this branch is also topologically conjugate to $P:\mathcal{K}(P)~\to~\mathcal{K}(P)$ with the conjugacy being conformal on $\Int{(\widetilde{\mathcal{K}}\setminus\mathfrak{D})}$.
\end{enumerate}
\end{theorem}

The proof of this theorem will be given in the next three subsections. In Subsections~\ref{corr_group_1_subsubsec} and~\ref{lifted_tiling_quotient_subsubsec}, we will study the group structure of $\mathfrak{C}$ on $\widetilde{\cT}$ which will allow us to identify the conformal structure of the quotient orbifold $\ \faktor{\widetilde{\cT}}{\mathfrak{C}}$. In Subsection~\ref{corr_poly_1_subsubsec}, we will analyze the dynamics of suitable branches of $\mathfrak{C}$ on $\widetilde{\mathcal{K}}$, which will reveal the polynomial structure of the correspondence.

\subsubsection{Group structure in $\mathfrak{C}$}\label{corr_group_1_subsubsec}

\begin{proposition}\label{deck_prop}
There exists a conformal automorphism $\tau$ of $\widetilde{\mathcal{T}}$ such that 
$$
\tau^{np}=\mathrm{id},\ \textrm{and}\ R^{-1}(R(z))=\{z,\tau(z),\cdots,\tau^{np-1}(z)\}\ \forall\ z\in \widetilde{\mathcal{T}}.
$$ 
\end{proposition}
\begin{proof}
Let $\pmb{\Phi}:\D\times\Z/p\Z\longrightarrow\widetilde{\cT}$ be a conformal isomorphism that sends $(0,j)$ to the unique critical point (of multiplicity $n-1$) of $R$ in $U_j$, for $j\in\Z/p\Z$. Recall that $\mathfrak{X}_\Gamma:\D\to\cT$ is a conformal isomorphism that conjugates $A_{\Gamma}^{\mathrm{fBS}}$ to $F$, and hence sends the unique critical value $0$ of $A_{\Gamma}^{\mathrm{fBS}}$ to the unique critical value of $R$ in $\cT$ (which is also the unique critical value of $F$ in $\cT$). Thus, 
$$
\widetilde{R}:=\mathfrak{X}_\Gamma^{-1}\circ R\circ\pmb{\Phi}:\D\times\Z/p\Z\longrightarrow\D
$$
is a holomorphic branched covering of degree $np$, that restricts to a degree $n$ branched covering $\D\times\{j\}\to\D$ and carries $(0,j)$ to $0$ with local degree $n$. Thus, after possibly pre-composing $\pmb{\Phi}$ with a rotation on each $\D\times\{j\}$, we can write $\widetilde{R}$ as
$$
(w,j)\mapsto w^n,\quad w\in\D,\ j\in\Z/p\Z.
$$
Let us now define a conformal automorphism 
$$
\widetilde{\tau}:\D\times\Z/p\Z\longrightarrow \D\times\Z/p\Z,\quad (w,j)\mapsto 
\begin{cases}
(w,j+1),\quad \mathrm{for}\quad j\in \{0,\cdots,p-2\}\\
(e^{\frac{2i\pi}{n}} w, 0),\quad \mathrm{for}\quad j=p-1.
\end{cases}
$$
It is readily checked that 
$$
\widetilde{\tau}^{np}=\mathrm{id},\ \textrm{and}\ \widetilde{R}^{-1}(\widetilde{R}(w,j))=\{(w,j),\widetilde{\tau}(w,j),\cdots,\widetilde{\tau}^{np-1}(w,j)\}\ \forall\ (w,j)\in \D\times\Z/p\Z.
$$ 
The desired automorphism $\tau$ of $\widetilde{\cT}$ is now given by $\pmb{\Phi}\circ\widetilde{\tau}\circ\pmb{\Phi}^{-1}$.
\end{proof}

\begin{remark}
Note that by construction, $\tau(U_j)=U_{j+1},\ j\in\Z/p\Z$. Moreover, $\tau^p$ restricts to an order $n$ conformal automorphism on each $U_j$.
\end{remark}

It follows that the forward branches of $\mathfrak{C}$ on $\widetilde{\mathcal{T}}$ are given by the conformal automorphisms $\tau\circ\eta,\cdots,\tau^{np-1}\circ\eta$.

\begin{proposition}\label{grand_orbit_group_prop}
The dynamics of $\mathfrak{C}$ on $\widetilde{\cT}$ is equivalent to the action of the group 
$$
\langle\eta\rangle\ast\langle\tau\rangle\ \cong\ \Z/2\Z\ast\Z/(np)\Z
$$ 
of conformal automorphisms of $\widetilde{\cT}$. 
\end{proposition}
\begin{proof}
The tiling set $\mathcal{T}$ is dynamically tessellated. We call $\overline{T^0}$ (closure taken in $\mathcal{T}$) the rank zero tile (where $T^0=\widehat{\C}\setminus\overline{\Omega}$), and connected components of $F^{-m}(\overline{T^0})$ tiles of rank $m$. A connected component of the pre-image of a rank $m$ tile of $\mathcal{T}$ under $R$ is called a rank $m$ tile of $\widetilde{\mathcal{T}}$.

We have already observed that the forward branches of $\mathfrak{C}$ on $\widetilde{\mathcal{T}}$ are given by $\tau\circ\eta,\cdots,\tau^{np-1}\circ\eta$. Furthermore, $\tau=(\tau^{2}\circ\eta)\circ(\tau\circ\eta)^{-1}$, and hence 
$$
\langle\tau\circ\eta,\cdots,\tau^{np-1}\circ\eta\rangle=\langle\eta,\tau\rangle.
$$ 
It now remains to justify that $\langle\eta,\tau\rangle$ is the free product of the cyclic groups $\langle\eta\rangle$ and $\langle\tau\rangle$. This will be done by applying a ping-pong type argument using the tiling structure of $\widetilde{\cT}$.

To this end, first note that any relation in $\langle\eta,\tau\rangle$ other than $\eta^{2}=\mathrm{id}$ and $\tau^{np}=\mathrm{id}$ can be reduced to one of the form 
\begin{equation}
(\tau^{k_1}\circ\eta)\circ\cdots\circ(\tau^{k_r}\circ\eta)=\mathrm{id},
\label{group_relation_1}
\end{equation} 
or 
\begin{equation}
(\tau^{k_1}\circ\eta)\circ\cdots\circ(\tau^{k_r}\circ\eta)=\eta,
\label{group_relation_2}
\end{equation}
where $r\geq 1$ and $k_1,\cdots,k_r\in\{1,\cdots,np-1\}$.
\smallskip

\noindent\textbf{Case 1:} Let us first assume that there exists a relation of the form (\ref{group_relation_1}) in $\langle\eta,\tau\rangle$. We claim that $(\tau^{k_j}\circ\eta)$ maps a tile $\mathfrak{T}$ of rank $s$ in $\widetilde{\mathcal{T}}\setminus \mathfrak{D}$ to a tile of rank $(s+1)$ in~$\widetilde{\mathcal{T}}\setminus \mathfrak{D}$. 
\begin{proof}[Proof of claim]
By Relation~\eqref{f_inv_lift_eqn}, we have that $F(R(\eta(z)))=R(z)$ for $z\in\widetilde{\cT}\setminus\mathfrak{D}$. Hence, $F$ maps $R(\eta(\mathfrak{T}))$ to a rank $s$ tile in $\cT$. It follows that $R(\eta(\mathfrak{T}))$ is a rank $(s+1)$ tile in $\cT$, and hence $\eta(\mathfrak{T})$ is a rank $(s+1)$ tile in $\widetilde{\cT}\cap\overline{\mathfrak{D}}$. As $R$ is injective on $\overline{\mathfrak{D}}$, the non-trivial deck transformation $\tau^{k_j}$ of $R$ carries $\eta(\mathfrak{T})$ to a tile of rank $(s+1)$ in $\widetilde{\mathcal{T}}\setminus \mathfrak{D}$.
\end{proof}

\noindent Hence, the group element on the left of Relation~(\ref{group_relation_1}) maps a tile of rank $0$ in $\widetilde{\cT}$ to a tile of rank $r\geq 1$. Clearly, such an element cannot be the identity map.
\smallskip

\noindent\textbf{Case 2:} Now we consider a relation of the form (\ref{group_relation_2}) in $\langle\tau,\eta\rangle$. Each $(\tau^{k_j}\circ\eta)$ maps $\widetilde{\mathcal{T}}\setminus\mathfrak{D}$ to itself. Hence, the group element on the left of Relation~(\ref{group_relation_2}) maps $\widetilde{\mathcal{T}}\setminus\mathfrak{D}$ to itself, while $\eta$ maps $\widetilde{\mathcal{T}}\setminus\mathfrak{D}$ to $\widetilde{\mathcal{T}}\cap\overline{\mathfrak{D}}$. This shows that there cannot exist a relation of the form (\ref{group_relation_2}) in $\langle\tau,\eta\rangle$.

We conclude that $\eta^{2}=\mathrm{id}$ and $\tau^{np}=\mathrm{id}$ are the only relations in $\langle\eta,\tau\rangle$, and hence $\langle\eta,\tau\rangle=\langle\eta\rangle\ast\langle\tau\rangle\cong\Z/2\Z\ast\Z/(np)\Z$. 
\end{proof}

\subsubsection{The quotient orbifold}\label{lifted_tiling_quotient_subsubsec}

\begin{proposition}\label{lifted_tiling_quotient_prop}
The group $\langle\eta\rangle\ast\langle\tau\rangle$ acts properly discontinuously on $\widetilde{\cT}$. Moreover, the quotient orbifold $\faktor{\widetilde{\cT}}{\langle\eta\rangle\ast\langle\tau\rangle}$ is biholomorphic to $\Sigma=\faktor{\D}{\widehat{\Gamma}}$.
\end{proposition}
\begin{proof}
We set $\widetilde{T^0}:=R^{-1}(T^0)$. 
Note that $\widetilde{T^0}$ consists of $p$ components, one in each $U_i$, $i\in\Z/p\Z$.
Further, let 
$$
G_0:=\{f\in\langle\eta\rangle\ast\langle\tau\rangle: f(U_0)=U_0\}
$$
be the stabilizer subgroup of $U_0$ in $\langle\eta\rangle\ast\langle\tau\rangle$. As the cyclic group $\langle\tau\rangle$ acts transitively on the components of $\widetilde{\cT}$, it suffices to show that $G_0$ acts properly discontinuously on $U_0$ and that $\faktor{U_0}{G_0}$ is biholomorphic to $\Sigma$. 

Note that the maps $\xi$ and $R$ in the vertical arrows of the following commutative diagram are degree $n$ branched coverings. 
\[
 \begin{tikzcd}
 \left(\D,0\right)    \arrow{d}[swap]{\mathrm{\xi:w\mapsto w^n}} \arrow{r}{\widetilde{\mathfrak{X}_\Gamma}} &  \left(U_0,x_0\right) \arrow{d}{R} \\
\left(\D,0\right)   \arrow{r}{\mathfrak{X}_{\Gamma}}  & \left(\cT,R(x_0)\right)
\end{tikzcd}
\]
Moreover, $\xi$ (respectively, $R$) has an $(n-1)-$fold critical point at $0$ (respectively, at $x_0$) with the associated critical value at $0$ (respectively, at $R(x_0)$). Recall also that $0\in\D$ (respectively, $R(x_0)$) is the unique critical value of $A_{\Gamma}^{\mathrm{fBS}}$ (respectively, of $F$ in $\cT$). Since the conformal map $\mathfrak{X}_{\Gamma}$ conjugates $A_{\Gamma}^{\mathrm{fBS}}$ to $F$, it follows that $\mathfrak{X}_{\Gamma}$ sends $0$ to $R(x_0)$. Hence, $\mathfrak{X}_{\Gamma}$ lifts to a conformal isomorphism $\widetilde{\mathfrak{X}_\Gamma}:\D\to U_0$ that maps $0$ to $x_0$. 

By construction, $\widetilde{\mathfrak{X}_\Gamma}$ maps $\Int{\psi_\rho(\pmb{\Pi})}$ conformally onto $\widetilde{T^0_{U_0}}:=\widetilde{T^0}\cap U_0$. After possibly pre-composing $\widetilde{\mathfrak{X}_\Gamma}$ with a power of $M_\omega$, we can assume that $\widetilde{\mathfrak{X}_\Gamma}$ takes the bi-infinite geodesic $\psi_\rho(C_{1,1})\subset \partial\psi_\rho(\pmb{\Pi})$ onto $\partial\mathfrak{D}\cap U_0 \subset\partial\widetilde{T^0_{U_0}}$. Since $\tau^p$ restricts to an order $n$ automorphism of the $np-$gon $\widetilde{T^0_{U_0}}$ that fixes the unique critical point $x_0$ of $R$ (of multiplicity $n-1$) in $U_0$ and has derivative $\omega$ at this fixed point (this follows from the construction of $\tau$ in Proposition~\ref{deck_prop}), the above construction implies that $\widetilde{\mathfrak{X}_\Gamma}$ conjugates $M_\omega$ to $\tau^p$.

Let us set
$$
\mathfrak{S}:=\widetilde{\mathfrak{X}_\Gamma}(\psi_\rho(\widehat{\pmb{\Pi}}))\subseteq \widetilde{T^0_{U_0}},
$$
where $\widehat{\pmb{\Pi}}$ is the fundamental domain of $\widehat{\pmb{\Gamma}}_{n,p}=\pmb{\Gamma_{n,p}}\rtimes\langle M_\omega\rangle$ introduced in Subsection~\ref{factor_bs_gen_subsec}.
The set $\mathfrak{S}$ is a closed sector (in the topology of $U_0$) based at $x_0$ whose sides are geodesics in the hyperbolic metric of $U_0$. Moreover, $R$ is injective on the interior of $\mathfrak{S}$ and maps the two geodesics emanating from $x_0$ to the line segment $\mathfrak{X}_\Gamma(0,1)$ in $\cT$.

It is not hard to see using the actions of the generators $\tau^{j}\circ\eta$ of the group $\langle\eta\rangle\ast\langle\tau\rangle$ on the tiles of $\widetilde{\cT}$ that $\mathfrak{S}$ is a closed fundamental domain for the $G_0-$action on $U_0$. In particular, $G_0$ acts properly discontinuously on $U_0$.

We now proceed to identify the quotient $\faktor{U_0}{G_0}$. Each component $U_i$ of $\widetilde{\cT}$ is stabilized by some $\tau^j\circ\eta$. All these maps, conjugated by suitable powers of $\tau$, give elements of $G_0$ that act as side-pairing transformations on the boundary of the $np-$gon $\widetilde{T^0_{U_0}}$. Combined with the map $\tau^p$, (a subset of) these maps pair the sides of $\mathfrak{S}$. Finally, $\widetilde{\mathfrak{X}_\Gamma}$ conjugates these side-pairing transformations for the sector $\mathfrak{S}$ to the side-pairing transformations $\rho(g_{1,1}),\cdots,\rho(g_{1,p}), M_\omega$ for the fundamental domain $\psi_\rho(\widehat{\pmb{\Pi}})$ of $\widehat{\Gamma}=\Gamma\rtimes\langle M_\omega\rangle$.
It now follows that $\ \faktor{U_0}{G_0}$ is biholomorphic to the quotient~$\ \faktor{\D}{\widehat{\Gamma}}$.
\end{proof}

\begin{remark}
Not just the group $\widehat{\Gamma}$, the representation $\widehat{\rho}:\widehat{\pmb{\Gamma}}_{n,p}\to\widehat{\Gamma}$ (see Section~\ref{factor_bs_gen_subsec}) is also recovered from $\mathfrak{C}$ via the side-pairing transformations of $\mathfrak{S}$ described in the proof of Proposition~\ref{lifted_tiling_quotient_prop}.
\end{remark}

\subsubsection{Polynomial structure in $\mathfrak{C}$}\label{corr_poly_1_subsubsec}

We now set $\widetilde{\mathcal{K}}_1:=\widetilde{\mathcal{K}}\cap\overline{\mathfrak{D}}$ and $\widetilde{\mathcal{K}}_2:=\widetilde{\mathcal{K}}\setminus\mathfrak{D}$. The description of $\widetilde{\mathcal{T}}$ given in Section~\ref{corr_group_1_subsubsec} can be used to study the structure of $\widetilde{\cK}_1$ and $\widetilde{\cK}_2$.

\begin{lemma}\label{lifted_ne_top_lem}
\noindent\begin{enumerate}
\item $\widetilde{\cK}_1\cap\partial\mathfrak{D}=\widetilde{\cK}_2\cap\partial\mathfrak{D}=\widetilde{\mathcal{K}}_1\cap\widetilde{\mathcal{K}}_2=\widetilde{S}$.

\item $\widetilde{\mathcal{K}}_2=\eta(\widetilde{\mathcal{K}}_1)$.

\item $R$ carries $\widetilde{\mathcal{K}}_1$ (respectively, $\widetilde{\mathcal{K}}_2$) homeomorphically (respectively, as a degree $np-1$ branched cover) onto $\mathcal{K}$.

\item $\widetilde{\cK}$ is connected.
\end{enumerate}
\end{lemma}

\begin{proof}
1) By definition, $\widetilde{\cK}_i\cap\partial\mathfrak{D}=\{z\in\partial\mathfrak{D}: R(z)\in\cK\}$, for $i\in\{1,2\}$. Recall from Section~\ref{unif_rat_crit_pnt_subsec} that $\partial\Omega=R(\partial\mathfrak{D})$ meets $\cK$ precisely at the finite set $S_F$. Hence, $\widetilde{\cK}_i\cap\partial\mathfrak{D}=(R\vert_{\partial\mathfrak{D}})^{-1}(S_F)=\widetilde{S}$, for $i\in\{1,2\}$. Since $\widetilde{\mathcal{K}}_1\cap\widetilde{\mathcal{K}}_2\subset\partial\mathfrak{D}$, it now follows that $\widetilde{\mathcal{K}}_1\cap\widetilde{\mathcal{K}}_2 =\widetilde{S}$.

2) The $\eta-$invariance of $\widetilde{\mathcal{K}}$ (see Proposition~\ref{corr_partition_1_prop}) implies that $\eta(\widetilde{\cK}\cap\mathfrak{D})=\widetilde{\cK}\setminus\overline{\mathfrak{D}}$. By Lemma~\ref{S_F_lem}, $\eta(\widetilde{S})=\widetilde{S}$. The result now follows from these facts and the description of $\widetilde{\cK}_i\cap\partial\mathfrak{D}$, $i\in\{1,2\}$, given in the previous part.

3) As $R$ is a homeomorphism from $\overline{\mathfrak{D}}$ onto $\overline{\Omega}$ and $\cK\subset\overline{\Omega}$, it follows that $\widetilde{\mathcal{K}}_1=R^{-1}(\mathcal{K})\cap\overline{\mathfrak{D}}=(R\vert_{\overline{\mathfrak{D}}})^{-1}(\cK)$. Hence, $R$ carries $\widetilde{\cK}_1$ homeomorphically onto $\cK$. Since $R$ is a global branched covering of degree $np$, it now follows that it maps $\widetilde{\cK}_2=R^{-1}(\cK)\setminus\mathfrak{D}$ as a degree $np-1$ branched cover onto $\mathcal{K}$.

4) Connectivity of $\cK$, combined with parts (2) and (3) of this lemma, implies that both $\widetilde{\cK}_1$ and $\widetilde{\cK}_2$ are connected. Since $\widetilde{\mathcal{K}}_1\cap\widetilde{\mathcal{K}}_2=\widetilde{S}\neq\emptyset$, we conclude that $\widetilde{\cK}=\widetilde{\mathcal{K}}_1\cup\widetilde{\mathcal{K}}_2$ is connected.
\end{proof}

\begin{proposition}\label{lifted_ne_dyn_prop}
\noindent\begin{enumerate}
\item $\widetilde{\mathcal{K}}_2$ is forward invariant, and hence, $\widetilde{\mathcal{K}}_1$ is backward invariant under $\mathfrak{C}$.

\item $\mathfrak{C}$ has a forward branch carrying $\widetilde{\mathcal{K}}_1$ onto itself with degree $np-1$, and this branch is conformally conjugate to $P:\mathcal{K}(P)\to \mathcal{K}(P)$. 

\item $\mathfrak{C}$ has a backward branch carrying $\widetilde{\mathcal{K}}_2$ onto itself with degree $np-1$, and this branch is conformally conjugate to $P:\mathcal{K}(P)\to \mathcal{K}(P)$.
\end{enumerate}
\end{proposition}
\begin{proof}
1) Let $z\in\widetilde{\mathcal{K}}_2$. By Lemma~\ref{lifted_ne_top_lem}, $\eta(z)\in \widetilde{\mathcal{K}}_1$. The same lemma also tells us that each point in $\cK$ has a unique preimage under $R$ in $\widetilde{\mathcal{K}}_1$ and $d=np-1$ preimages under $R$ (counted with multiplicity) in $\widetilde{\mathcal{K}}_2$. Hence, the non-identity local deck transformations of $R$ (i.e., local branches of $R^{-1}\circ R$) send $\eta(z)\in\widetilde{\mathcal{K}}_1$ to $\widetilde{\mathcal{K}}_2$. Thus, $\widetilde{\mathcal{K}}_2$ is preserved by the forward branches of the correspondence $\mathfrak{C}$.

A similar reasoning shows that $\widetilde{\mathcal{K}}_1$ is invariant under the backward branches of $\mathfrak{C}$.
\smallskip
 
2) Recall from Lemma~\ref{lifted_ne_top_lem} that $R$ carries $\widetilde{\cK}_1$ homeomorphically onto $\cK$. We denote the correspondence inverse branch by $\left(R\vert_{\widetilde{\mathcal{K}}_1}\right)^{-1}$. Further, $R:\widetilde{\cK}_2\to\cK$ is a degree $np-1$ map. Now define 
\begin{equation*}
g:\widetilde{\cK}_2\to\widetilde{\cK}_1,\quad g:= \left(R\vert_{\widetilde{\mathcal{K}}_1}\right)^{-1}\circ R\vert_{\widetilde{\cK}_2}.
\end{equation*}
By definition, $g$ is a degree $np-1$ map satisfying $R\circ g=R$. Thus,
$$
g\circ\eta:\widetilde{\mathcal{K}}_1\to\widetilde{\mathcal{K}}_1
$$ 
is a degree $np-1$ forward branch of the correspondence $\mathfrak{C}$.

Clearly, the forward branch $(g\circ\eta)\vert_{\widetilde{\mathcal{K}}_1}$ is topologically conjugate (conformally on the interior) to $F\vert_{\mathcal{K}}\equiv R\circ\eta\circ (R\vert_{\widetilde{\mathcal{K}}_1})^{-1}$ via the homeomorphism $R:\widetilde{\mathcal{K}}_1\to \mathcal{K}$. The result now follows from the above discussion and the fact that $F:\cK\to\cK$ is topologically conjugate (conformally on the interior) to $P:\cK(P)\to\cK(P)$ via $\mathfrak{X}_P$.
\smallskip

3) It is easy to see that the map 
$$
\eta\circ g=\eta\circ\left(R\vert_{\widetilde{\cK}_1}\right)^{-1}\circ R: \widetilde{\mathcal{K}}_2\to\widetilde{\mathcal{K}}_2
$$
is a backward branch of the correspondence $\mathfrak{C}$ carrying $\widetilde{\mathcal{K}}_2$ onto itself with degree $np-1$. 
Finally, $\eta$ restricts to a conformal conjugacy between the backward branch $(\eta\circ g)\vert_{\widetilde{\mathcal{K}}_2}$ and the forward branch $(g\circ\eta)\vert_{\widetilde{\mathcal{K}}_1}$. 
\end{proof}

\begin{proof}[Proof of Theorem~\ref{corr_mating_thm_1}]
Follows from Propositions~\ref{grand_orbit_group_prop},~\ref{lifted_tiling_quotient_prop}, and~\ref{lifted_ne_dyn_prop}.
\end{proof}

\subsection{The general case}\label{gen_corr_subsec}

In this subsection, we will associate an algebraic correspondence with the conformal mating 
$$
F:\bigcup_{\alpha\in\mathcal{I}} R_\alpha(\overline{\mathfrak{D}_\alpha})\longrightarrow\widehat{\C}
$$
of the factor Bowen-Series map $A_{\Gamma}^{\mathrm{fBS}}$, where $(\rho:\pmb{\Gamma_{n,p}}\to\Gamma)\in\mathrm{Teich}^\omega(\pmb{\Gamma_{n,p}})$, and a polynomial $P$ lying in any real-symmetric (not necessarily principal) hyperbolic component in the connectedness locus of degree $np-1$ polynomials (see Proposition~\ref{mating_class_prop}). The correspondence will live on a nodal Riemann surface whose non-singular components are Riemann spheres.

Let us recall some notation (from Section~\ref{unif_rat_crit_pnt_subsec}) that will be used in this section. 
\begin{itemize}
\item The domain of $R_\alpha$ is denoted by $\widehat{\C}_\alpha$, and points in $\widehat{\C}_\alpha$ are denoted by $(z,\alpha)$. In particular, $\mathfrak{D}_\alpha \subset\widehat{\C}_\alpha$.

\item $\mathfrak{U}\ =\ \bigsqcup_{\alpha\in\mathcal{I}} \widehat{\C}_{\alpha}.$

\item $R:\ \mathfrak{U}\longrightarrow \widehat{\C},\ (z,\alpha)\mapsto R_{\alpha}(z),$ is a branched covering of degree $np$.

\item $\eta_\ast\ : \mathfrak{U}\longrightarrow \mathfrak{U},\ (z,\alpha)\mapsto (\eta(z),\kappa(\alpha)),$ is a homeomorphism. 

\item $\mathfrak{D}= \bigsqcup_{\alpha\in\mathcal{I}} \mathfrak{D}_\alpha.$
\end{itemize}

The conformal mating $F$ gives rise to a holomorphic correspondence on $\mathfrak{U}$ as follows. For $\alpha,\beta\in\mathcal{I}$, define
$$
\mathfrak{C}_{\alpha,\beta}:=\ \{(z,w)\in\widehat{\C}_\alpha\times\widehat{\C}_\beta: \frac{R_\beta(w)-R_{\kappa(\alpha)}(\eta(z))}{w-\eta(z)}=0\},\qquad \mathrm{if}\ \kappa(\alpha)=\beta,
$$
and
$$
\mathfrak{C}_{\alpha,\beta}:=\ \{(z,w)\in\widehat{\C}_\alpha\times\widehat{\C}_\beta: R_\beta(w)-R_{\kappa(\alpha)}(\eta(z))=0\},\qquad \mathrm{if}\ \kappa(\alpha)\neq \beta.
$$
The union of the algebraic curves $\mathfrak{C}_{\alpha,\beta}$ can be written succinctly as
\begin{equation}
\{(\mathfrak{u}_1,\mathfrak{u}_2)\in\mathfrak{U}\times\mathfrak{U}: \frac{R(\mathfrak{u}_2)-R(\eta_\ast(\mathfrak{u}_1))}{\mathfrak{u}_2-\eta_\ast(\mathfrak{u}_1)}=0\}.
\label{corr_gen_eqn}
\end{equation}
(The division in Equation~\eqref{corr_gen_eqn} makes sense since the numerator and the denominator can be viewed as points of $\widehat{\C}$.)
The first and second coordinate projection maps $\pi_1^\alpha$ and $\pi_2^\beta$ from $\mathfrak{C}_{\alpha,\beta}$ onto $\widehat{\C}_\alpha$ and $\widehat{\C}_\beta$ define a holomorphic (in fact, algebraic) correspondence from $\widehat{\C}_\alpha$ onto $\widehat{\C}_\beta$ (cf. \cite{DS06}):
\[
\begin{tikzcd}
 & \mathfrak{C}_{\alpha,\beta} \arrow{dl}[swap]{\pi_1^\alpha} \arrow{dr}{\pi_2^\beta} \\
\widehat{\C}_\alpha && \widehat{\C}_\beta.
\end{tikzcd}
\]
Combining all these holomorphic correspondences for various $\alpha,\beta\in\mathcal{I}$, we obtain a holomorphic correspondence on $\mathfrak{U}$ defined by the reducible curve $\sum_{\alpha,\beta}\mathfrak{C}_{\alpha,\beta}$. We denote this correspondence by $\mathfrak{C}_\ast$.

In order to capture the mating structure of the correspondence, we need to pass to a quotient of $\mathfrak{U}$. To this end, we endow $\mathfrak{U}$ with the following finite equivalent relation:
\begin{center}
For $z\in\widetilde{S}_\alpha\subset\widehat{\C}_\alpha$ and $w\in\widetilde{S}_\beta\subset\widehat{\C}_\beta$,
\smallskip

$(z,\alpha)\sim_{\mathrm{w}} (w,\beta)\iff R_\alpha(z)=R_\beta(w)$.
\end{center}
(See Section~\ref{unif_rat_crit_pnt_subsec} for the definition of $\widetilde{S}_\alpha$.)
The fact that $\mathrm{Dom}(F)$ is the quotient of $\overline{\D}$ by a finite lamination (see Proposition~\ref{mat_dom_prop}) and that $R_\alpha\vert_{\partial\mathfrak{D}_\alpha}$ is injective (for all $\alpha\in\mathcal{I}$) imply that 
$$
\mathfrak{W}\ :=\ \faktor{\mathfrak{U}}{\sim_{\mathrm{w}}}
$$
has the structure of a compact, simply connected, nodal Riemann surface. 
By definition, the map $R:\mathfrak{U}\to\widehat{\C}$ descends to a map
$$
\widecheck{R}: \mathfrak{W}\longrightarrow\widehat{\C}.
$$ 
Note that $
\widecheck{R}$ is also a degree $np$ branched covering. Abusing notation, we denote the image of a set $X\subset\mathfrak{U}$ (respectively, a point $p\in\mathfrak{U}$) under the quotient map $\mathfrak{U}\longrightarrow\mathfrak{W}$ by $X$ (respectively, $p$).

\begin{lemma}\label{eta_descends_lem}
The homeomorphism $\eta_\ast:\mathfrak{U}\to\mathfrak{U}$ descends to a homeomorphism $\widecheck{\eta}:\mathfrak{W}\longrightarrow\mathfrak{W}$.
\end{lemma}
\begin{proof}
Let us suppose that $(z,\alpha)\sim_{\mathrm{w}} (w,\beta)$; i.e., $z\in\widetilde{S}_\alpha, w\in\widetilde{S}_\beta,$ and $R_\alpha(z)=R_\beta(w)$. We need to show that $\eta_\ast(z,\alpha)\sim_{\mathrm{w}} \eta_\ast(w,\beta)$.

To this end, note that 
$$
\eta_\ast(z,\alpha)=(\eta(z),\kappa(\alpha))\in\widetilde{S}_{\kappa(\alpha)},\quad \textrm{and}\quad \eta_\ast(w,\beta)=(\eta(w),\kappa(\beta))\in\widetilde{S}_{\kappa(\beta)}
$$ 
(by Lemma~\ref{S_F_lem}). Now, $R_{\kappa(\alpha)}(\eta(z))=F(R_\alpha(z))=F(R_\beta(w))=R_{\kappa(\beta)}(\eta(w))$, and hence $\eta_\ast(z,\alpha)\sim_{\mathrm{w}} \eta_\ast(w,\beta)$.
\end{proof} 

Thus, the correspondence $\mathfrak{C}_\ast$ on $\mathfrak{U}$ also descends to a correspondence on $\mathfrak{W}$. We denote this correspondence by $\mathfrak{C}$.

\begin{remark}\label{examples_eta_rem}
Recall that for each of the two conformal matings $F$ considered in  Section~\ref{examples_subsec}, the domain of definition $\mathrm{Dom}(F)$ is a union of two closed Jordan disks touching at a single point. Thus, the associated correspondence $\mathfrak{C}$ is defined on a nodal sphere with a unique node; i.e., the nodal surface $\mathfrak{W}$ has precisely two non-singular components, and both of them are Riemann spheres. 

In the quadratic example of Section~\ref{quad_example_subsubsec}, the mating $F$ preserves the boundary of each of the two Jordan disks comprising $\Int{\mathrm{Dom}(F)}$, and hence the involution $\widecheck{\eta}$ of Lemma~\ref{eta_descends_lem} leaves each of the two Riemann spheres in $\mathfrak{W}$ invariant. 

In the cubic example of Section~\ref{cubic_example_subsubsec}, the mating $F$ swaps the boundaries of the two Jordan disks forming $\Int{\mathrm{Dom}(F)}$, so the involution $\widecheck{\eta}$ carries the two Riemann spheres in $\mathfrak{W}$ to each other. 
\end{remark}

\subsubsection{Dynamical partition for $\mathfrak{C}$}\label{inv_partition_corr_2_subsubsec}

As in Section~\ref{inv_partition_corr_1_subsubsec}, we define
$$
\widetilde{\mathcal{K}}:=\widecheck{R}^{-1}(\mathcal{K}),\quad \widetilde{\mathcal{T}}:= \widecheck{R}^{-1}(\mathcal{T}).
$$
The proof of Proposition~\ref{corr_partition_1_prop} applies mutatis mutandis to the current setting and implies the following result.

\begin{proposition}\label{corr_partition_2_prop}
\noindent\begin{enumerate}
\item $\widecheck{\eta}(\widetilde{\mathcal{T}})=\widetilde{\mathcal{T}}$, and $\widecheck{\eta}(\widetilde{\cK})=\widetilde{\cK}$.

\item Let $(\mathfrak{u}_1,\mathfrak{u}_2)\in\mathfrak{C}$. Then $\mathfrak{u}_1\in\widetilde{\mathcal{T}}$ (respectively, $\mathfrak{u}_1\in\widetilde{\mathcal{K}}$) if and only if $\mathfrak{u}_2\in\widetilde{\mathcal{T}}$ (respectively, $\mathfrak{u}_2\in\widetilde{\mathcal{K}}$).
\end{enumerate}
\end{proposition}

\subsubsection{Group structure in $\mathfrak{C}$}\label{corr_group_2_subsubsec}

Thanks to the description of the critical points of $\widecheck{R}$ in $\widecheck{R}^{-1}(\cT)\setminus\overline{\mathfrak{D}}$ given in Proposition~\ref{crit_pnt_r_sharp_prop}, the arguments of Subsections~\ref{corr_group_1_subsubsec} and~\ref{lifted_tiling_quotient_subsubsec} apply mutatis mutandis to the general situation and imply the following results.
 
\begin{proposition}\label{grand_orbit_group_gen_prop}
\noindent\begin{enumerate}
\item $\widetilde{\cT}$ is the union of $p$ disjoint topological disks $U_0,\cdots,U_{p-1}$, where each $U_i$ contains a unique critical point (of multiplicity $n-1$) of $\widecheck{R}$ and is mapped onto $\cT$ with degree $n$. 

\item There exists a conformal automorphism $\tau$ of $\widetilde{\mathcal{T}}$ such that 
$$
\tau^{np}=\mathrm{id},\ \textrm{and}\ \widecheck{R}^{-1}(\widecheck{R}(z))=\{z,\tau(z),\cdots,\tau^{np-1}(z)\}\ \forall\ z\in \widetilde{\mathcal{T}}.
$$ 
Hence, the forward branches of $\mathfrak{C}$ on $\widetilde{\mathcal{T}}$ are given by the conformal automorphisms $\tau\circ\widecheck{\eta},\cdots,\tau^{np-1}\circ\widecheck{\eta}$.

\item The dynamics of $\mathfrak{C}$ on $\widetilde{\cT}$ is equivalent to the action of the group 
$$
\langle\widecheck{\eta}\rangle\ast\langle\tau\rangle\ \cong\ \Z/2\Z\ast\Z/(np)\Z
$$ 
of conformal automorphisms of $\widetilde{\cT}$. 

\item The group $\langle\widecheck{\eta}\rangle\ast\langle\tau\rangle$ acts properly discontinuously on $\widetilde{\cT}$. Moreover, the quotient orbifold $\faktor{\widetilde{\cT}}{\langle\widecheck{\eta}\rangle\ast\langle\tau\rangle}$ is biholomorphic to $\Sigma=\faktor{\D}{\widehat{\Gamma}}$.
\end{enumerate}
\end{proposition}

\subsubsection{Polynomial structure in $\mathfrak{C}$}\label{corr_poly_2_subsubsec}

We set
$$
\widetilde{S}:=\bigsqcup_{\alpha\in\mathcal{I}}\widetilde{S}_{\alpha}.
$$ 
Note that $\widecheck{\eta}$ maps $\overline{\mathfrak{D}}$ onto $\mathfrak{W}\ \setminus\ \mathfrak{D}$, and preserves $\widetilde{S}$. 
As in Section~\ref{corr_poly_1_subsubsec}, we set $\widetilde{\mathcal{K}}_1:=\widetilde{\mathcal{K}}\cap\overline{\mathfrak{D}}$ and $\widetilde{\mathcal{K}}_2:=\widetilde{\mathcal{K}}\ \setminus\ \mathfrak{D}$.

\begin{lemma}\label{lifted_ne_top_1_lem}
We have that $\widetilde{\mathcal{K}}_2=\widecheck{\eta}(\widetilde{\mathcal{K}}_1)$, and $\widetilde{\mathcal{K}}_1\cap\widetilde{\mathcal{K}}_2=\widetilde{S}$.
\end{lemma}
\begin{proof}
The first statement follows from the fact that $\widecheck{\eta}(\widetilde{\mathcal{K}})=\widetilde{\mathcal{K}}$. 
For the second statement, first observe that
$$
\widetilde{\mathcal{K}}_1\cap\widetilde{\mathcal{K}}_2=\{\mathfrak{u}\in\partial\mathfrak{D}: \widecheck{R}(\mathfrak{u})\in\mathcal{K}\}.
$$
The result is now a consequence of the fact that $\widecheck{R}(\widetilde{S})=S_F\subset\mathcal{K}$ and $\widecheck{R}(\partial\mathfrak{D}\ \setminus\ \widetilde{S})=\partial\mathrm{Dom}(F) \setminus S_F\subset\mathcal{T}$.
\end{proof}

Finally, Proposition~\ref{lifted_ne_dyn_prop} naturally generalizes to the current setting.

\begin{proposition}\label{lifted_ne_dyn_gen_prop}
\noindent\begin{enumerate}
\item $\widetilde{\mathcal{K}}_2$ is forward invariant, and hence, $\widetilde{\mathcal{K}}_1$ is backward invariant under $\mathfrak{C}$.

\item $\mathfrak{C}$ has a forward branch carrying $\widetilde{\mathcal{K}}_1$ onto itself with degree $np-1$, and this branch is conformally conjugate to $P:\cK(P)\to\cK(P)$.

\item $\mathfrak{C}$ has a backward branch carrying $\widetilde{\mathcal{K}}_2$ onto itself with degree $np-1$, and this branch is also conformally conjugate to $P:\cK(P)\to\cK(P)$.
\end{enumerate}
\end{proposition}
\begin{proof}
The proof is similar to that of Proposition~\ref{lifted_ne_dyn_prop}. We only give a proof of the second statement.

The forward branch of $\mathfrak{C}$ carrying $\widetilde{\mathcal{K}}_1$ onto itself (with degree $np-1$) acts as 
$$
\mathfrak{b}:\ \left(z,\alpha\right)\mapsto \left(\left(R_{\beta}\vert_{\overline{\mathfrak{D}_{\beta}}}\right)^{-1}\left(R_{\kappa(\alpha)}(\eta(z))\right),\beta\right),
$$
where $(z,\alpha)\in\widetilde{\mathcal{K}}_1$, and $R_{\kappa(\alpha)}(\eta(z))\in\overline{\Omega_\beta}$.
It is easy to see from the construction that $\widecheck{R}:\widetilde{\mathcal{K}}_1\longrightarrow \mathcal{K}$ is a homeomorphism. We claim that $\widecheck{R}\vert_{\widetilde{\mathcal{K}}_1}$ is a conjugating map between $\mathfrak{b}$ and $F\vert_{\cK}$. To this end, note that 
$$
F(\widecheck{R}(z,\alpha))=F(R_\alpha(z))=R_{\kappa(\alpha)}(\eta(z)),
$$
and 
$$
\widecheck{R}\left(\mathfrak{b}(z,\alpha)\right)=\widecheck{R}\left(\left(R_{\beta}\vert_{\overline{\mathfrak{D}_{\beta}}}\right)^{-1}\left(R_{\kappa(\alpha)}(\eta(z))\right),\beta\right)=R_{\kappa(\alpha)}(\eta(z)).
$$
It follows that $\widecheck{R}\vert_{\widetilde{\mathcal{K}}_1}\circ\mathfrak{b}=F\circ\widecheck{R}\vert_{\widetilde{\mathcal{K}}_1}$. To complete the proof, we note that $F\vert_{\cK}$ is conformally conjugate to $P\vert_{\cK(P)}$ via the mating conjugacy $\mathfrak{X}_P$.
\end{proof}

We summarize the above results in the following theorem.

\begin{theorem}\label{corr_mating_thm_2}
The correspondence $\mathfrak{C}$ on $\mathfrak{W}$ defined by Equation~\eqref{corr_gen_eqn} is a mating of $P$ and $\Sigma:=\faktor{\D}{\widehat{\Gamma}}$ in the following sense.
\begin{enumerate}
\item The dynamics of $\mathfrak{C}$ on $\widetilde{\cT}$ is equivalent to the action of a group 
$$
\langle\widecheck{\eta}\rangle\ast\langle\tau\rangle\ \cong\ \Z/2\Z\ast\Z/(np)\Z
$$ 
of conformal automorphisms of $\widetilde{\cT}$. Here, $\tau$ is a conformal automorphism of $\widetilde{\cT}$ of order $np$ such that $\tau^p$ induces an order $n$ conformal automorphism on each component of $\widetilde{\cT}$.

Moreover, the above group action is properly discontinuous, and the quotient orbifold $\faktor{\widetilde{\cT}}{\mathfrak{C}}$ is biholomorphic to $\Sigma$.

\item The correspondence $\mathfrak{C}$ has a forward branch carrying $\widetilde{\mathcal{K}}\cap\overline{\mathfrak{D}}$ onto itself with degree $np-1$, and this branch is topologically conjugate to $P:\mathcal{K}(P)\to \mathcal{K}(P)$, with the conjugacy being conformal on the interior. On the other hand, $\mathfrak{C}$ has a backward branch carrying $\widetilde{\mathcal{K}}\ \setminus\ \mathfrak{D}$ onto itself with degree $np-1$, and this branch is also topologically conjugate to $P:\mathcal{K}(P)\to \mathcal{K}(P)$, with the conjugacy being conformal on the interior.
\end{enumerate}
\end{theorem}

\begin{proof}
The first statement is the content of Proposition~\ref{grand_orbit_group_gen_prop} and the second one is the content of Proposition~\ref{lifted_ne_dyn_gen_prop}.
\end{proof}

We are now ready to prove a slightly more general version of Theorem~\ref{corr_mating_intro_thm} announced in the introduction.

\begin{theorem}\label{corr_main_thm}
Let $\Sigma$ be a hyperbolic orbifold of genus zero with arbitrarily many (at least one, but finite) punctures, at most one order two orbifold point, and at most one order $\nu\geq 3$ orbifold point.
Further, let $P$ be a polynomial in a real-symmetric hyperbolic component of degree $1-2\nu\cdot\chi_{\mathrm{orb}}(\Sigma)$ (respectively, $1-2\chi_{\mathrm{orb}}(\Sigma)$) polynomials if $\Sigma$ has (respectively, does not have) an order $\nu$ orbifold point.

Then, there exist a holomorphic, hence algebraic correspondence $\mathfrak{C}$ on a compact, simply connected, (possibly noded) Riemann surface $\mathfrak{W}$ and a $\mathfrak{C}-$invariant partition $\mathfrak{W}=\widetilde{\cT}\sqcup\widetilde{\cK}$ such that the following hold.
\begin{enumerate}
	\item On $\widetilde{\cT}$, the dynamics of $\mathfrak{C}$ is orbit-equivalent to the action of a  group of conformal automorphisms acting properly discontinuously. Further, $\faktor{\widetilde{\cT}}{\mathfrak{C}}$ is biholomorphic to $\Sigma$.
	
	\item $\widetilde{\cK}$ can be written as the union of two copies $\widetilde{\cK}_1, \widetilde{\cK}_2$ of $\cK(P)$ (where $\cK(P)$ is the filled Julia set of $P$), such that $\widetilde{\cK}_1$ and $\widetilde{\cK}_2$ intersect in finitely many points. Furthermore, $\mathfrak{C}$ has a forward (respectively, backward) branch carrying $\widetilde{\cK}_1$ (respectively, $\widetilde{\cK}_2$) onto itself with degree $np-1$, and this branch is conformally conjugate to $P:\mathcal{K}(P)\to \mathcal{K}(P)$. 
\end{enumerate}
In particular, if $P$ lies in a principal hyperbolic component, then $\mathfrak{W}=\widehat{\C}$; i.e., $\mathfrak{C}$ is an algebraic correspondence on the Riemann sphere.
\end{theorem}

\begin{proof}
Let us assume that $\Sigma$ has $\delta_1\geq 1$ punctures, $\delta_2\in\{0,1\}$ order two orbifold points, and $\delta_3\in\{0,1\}$ order $\nu\geq 3$ orbifold points. The condition that $\Sigma$ is hyperbolic is equivalent to the requirement that
$$
\chi_{\mathrm{orb}}(\Sigma)=2-\delta_1-\frac{\delta_2}{2}-\delta_3(1-\frac{1}{\nu})<0.
$$
We set
$$
n=\begin{cases}
1 \quad \mathrm{if}\quad \delta_3=0,\\
\nu \quad \mathrm{if}\quad \delta_3=1.
\end{cases}
$$
Further, we set
$$
p=\begin{cases}
2(\delta_1-1) \quad \mathrm{if}\quad \delta_2=0,\\
2\delta_1-1 \quad \mathrm{if}\quad \delta_2=1.
\end{cases}
$$
Note that when $\delta_3=0$, then 
$$
1-2\chi_{\mathrm{orb}}(\Sigma)=1-2(2-\delta_1-\frac{\delta_2}{2})=2\delta_1+\delta_2-3=p-1=np-1,
$$
(as $n=1$ in this case). On the other hand, when $\delta_3=1$, then
$$
1-2\nu\cdot\chi_{\mathrm{orb}}(\Sigma)=1-2\nu(1-\delta_1-\frac{\delta_2}{2}+\frac{1}{\nu})=2\nu\delta_1+\nu\delta_2-2\nu-1=\nu p-1=np-1
$$
(as $n=\nu$ in this case). Moreover, the restriction on $\chi_{\mathrm{orb}}(\Sigma)$ implies that $np\geq 3$. 

By construction, $\faktor{\D}{\widehat{\pmb{\Gamma}}_{n,p}}$ is homeomorphic to $\Sigma$ (as orbifolds). It follows that there exists $\Gamma_\Sigma\in\mathrm{Teich}(\widehat{\pmb{\Gamma}}_{n,p})$ such that $\faktor{\D}{\Gamma_\Sigma}$ is biholomorphic to $\Sigma$.
The result now follows by applying Theorems~\ref{corr_mating_thm_1} and~\ref{corr_mating_thm_2} on the pair $\Gamma_\Sigma, P$.
\end{proof}

\section{A character variety and a simultaneous uniformization locus}\label{char_var_sec}

In this section, we will put the results of the previous sections together to justify the diagram (Figure~\ref{intro_fig}) furnished in the introduction.
Along the way, we will put an algebraic structure on the moduli space of our correspondences in terms of the coefficients of the uniformizing rational maps. The construction of this space of correspondences will lay the foundation for the proof of Theorem~\ref{punc_sphere_teich_corr_intro_thm} (see Section~\ref{punc_sphere_bers_sec}).

We recall that $n,p$ are positive integers with $np\geq 3$, and $d:=np-1$. For $(\rho:\pmb{\Gamma_{n,p}}\to\Gamma)\in\mathrm{Teich}^\omega(\pmb{\Gamma_{n,p}})$, the conformal mating of $A_{\Gamma}^{\mathrm{fBS}}$ and $P\in\mathcal{H}_{d}$ is denoted by $F:\overline{\Omega}\to\widehat{\C}$. The associated mating semi-conjugacies are denoted by $\mathfrak{X}_P$ and $\mathfrak{X}_\Gamma$ (see Definition~\ref{conf_mat_def}). Further, let $R, \mathfrak{D}$ be as in Corollary~\ref{main_hyp_comp_mating_class_cor}.

\subsection{Moduli space of marked matings}\label{mating_moduli_subsec} Recall from Theorem~\ref{conf_mat_thm} that the conformal mating $F:\overline{\Omega}\to\widehat{\C}$ of $A_{\Gamma}^{\mathrm{fBS}}$ and $P$ is unique up to M{\"o}bius conjugacy. A \emph{marked conformal mating} is a pair $(F= A_{\Gamma}^{\mathrm{fBS}}\mate P, \mathfrak{X}_\Gamma(1))$. Two such pairs are equivalent if there is a M{\"o}bius map that conjugates the conformal matings respecting the marked fixed points. The collection of equivalence classes of marked conformal matings will be referred to as the \emph{moduli space of marked matings} associated with $\mathrm{Teich}^\omega(\pmb{\Gamma_{n,p}})$ and $\mathcal{H}_{d}$. We denote this space by
$$
\cM\equiv \cM\left(\mathrm{Teich}^\omega(\pmb{\Gamma_{n,p}}),\mathcal{H}_{d}\right).
$$ 
We have a natural map
\begin{equation*}
\begin{split}
\Xi_1:\ \mathrm{Teich}^\omega(\pmb{\Gamma_{n,p}})\times\mathcal{H}_{d}\quad \longrightarrow \quad \cM\ \\
(\Gamma,P) \quad \mapsto \quad [F:= A_{\Gamma}^{\mathrm{fBS}}\mate P, \mathfrak{X}_\Gamma(1)].
\end{split}
\label{teich_hyp_to_mating}
\end{equation*}

Let us now fix a conformal mating $F= A_{\Gamma}^{\mathrm{fBS}}\mate P:\overline{\Omega}\to\widehat{\C}$. By Corollary~\ref{main_hyp_comp_mating_class_cor}, there exist a Jordan domain $\mathfrak{D}$ (with $\eta(\partial\mathfrak{D})=\partial\mathfrak{D}$) and a degree $(d+1)$ rational map $R$ of $\widehat{\C}$ that maps $\overline{\mathfrak{D}}$ injectively onto $\overline{\Omega}$, such that $F\vert_{\overline{\Omega}}\equiv R\circ\eta\circ (R\vert_{\overline{\mathfrak{D}}})^{-1}$. Clearly, conjugating $F$ by a M{\"o}bius map amounts to post-composing $R$ with the same M{\"o}bius map.
We will now show that when a particular $F$ is chosen, the associated rational map $R$ is essentially unique.

We denote the centralizer of $\eta$ in $\mathrm{PSL}_2(\C)$ by $C(\eta)$.

\begin{proposition}\label{unique_unif_prop}
Let $F:\overline{\Omega}\to\widehat{\C}$ be a conformal mating of $A_{\Gamma}^{\mathrm{fBS}}$ and $P$. Suppose further that there exist pairs $(R_1,\mathfrak{D}_1), (R_2,\mathfrak{D}_2)$ with the following properties.
\begin{enumerate}
\item $\mathfrak{D}_i$ a Jordan domain with $\eta(\partial\mathfrak{D}_i)=\partial\mathfrak{D}_i$, 
\item $R_i\vert_{\overline{\mathfrak{D}_i}}$ is injective,
\item $R_i(\overline{\mathfrak{D}_i})=\overline{\Omega}$, and
\item $F\vert_{\overline{\Omega}}\equiv R_i\circ\eta\circ (R_i\vert_{\overline{\mathfrak{D}_i}})^{-1}$, for $i\in\{1,2\}$.
\end{enumerate}
Then, there exists a M{\"o}bius map $M\in C(\eta)$ such that $M(\mathfrak{D}_1)=\mathfrak{D}_2$ and $R_1\equiv R_2\circ M$.
\end{proposition}
\begin{proof}
We define
$$
M:\widehat{\C}\to\widehat{\C},\quad z\mapsto
\begin{cases}
(R_2\vert_{\overline{\mathfrak{D}_2}})^{-1}\circ R_1(z),\hspace{1.2cm} \mathrm{if}\ z\in \overline{\mathfrak{D}_1},\\
\eta\circ(R_2\vert_{\mathfrak{D}_2})^{-1}\circ R_1\circ\eta(z),\hspace{1.5mm} \mathrm{if}\ z\in \widehat{\C}\setminus\overline{\mathfrak{D}_1}.
\end{cases} 
$$
Since $(R_2\vert_{\partial\mathfrak{D}_2})^{-1}\circ R_1:\partial\mathfrak{D}_1\to\partial\mathfrak{D}_2$ conjugates $\eta\vert_{\partial\mathfrak{D}_1}$ to $\eta\vert_{\partial\mathfrak{D}_2}$, it follows that the piecewise definitions of $M$ agree continuously, and hence $M$ is a homeomorphism of the Riemann sphere that commutes with $\eta$. Moreover, $M$ is conformal away from the Jordan curve $\partial\mathfrak{D}_1$. 

The facts that $\partial\Omega\setminus S_F$ is a union of finitely many non-singular analytic arcs (see Section~\ref{unif_rat_crit_pnt_subsec}) and that $R_1$ has no critical point on $\partial\mathfrak{D}_1\setminus (R_1\vert_{\partial\mathfrak{D}_1})^{-1}(S_F)$ (by Corollary~\ref{crit_pnt_r_cor}) together imply that $\partial\mathfrak{D}_1$ is a piecewise non-singular analytic curve. In particular, $\partial\mathfrak{D}_1$ is conformally removable. It now follows that $M$ is a M{\"o}bius map commuting with $\eta$. 
Moreover, the definition of $M$ implies that $M(\mathfrak{D}_1)=\mathfrak{D}_2$ and $R_1\equiv R_2\circ M$.
\end{proof}

\noindent After possibly pre-composing with $z\mapsto -z$, we can and will assume that $R(1)=\mathfrak{X}_\Gamma(1)$.

\subsection{Space of correspondences as character variety}\label{corr_mod_space_subsec} 

Let us consider the space $\mathscr{C}$ of all correspondences of the form
\begin{equation}
(z,w)\in\mathfrak{C}\iff \frac{R(w)-R(\eta(z))}{w-\eta(z)}=0,
\label{corr_eqn_1}
\end{equation}
where $R\in\mathrm{Rat}_{d+1}(\C)$. Such a correspondence $\mathfrak{C}$ has bi-degree $d$:$d$; the $d$ forward (respectively, backward) branches of $\mathfrak{C}$ send a point $z$ to the $d$ points in the set $R^{-1}(R(\eta(z)))\setminus\{\eta(z)\}$ (respectively, to the $d$ points in the set $\eta(R^{-1}(R(z)))\setminus\{\eta(z)\}$). Note that the space $\mathscr{C}$, which is parametrized by the quasi-projective variety $\mathrm{Rat}_{d+1}(\C)$, defines an ambient space in which the correspondences produced by Theorem~\ref{corr_mating_thm_1} live. 

\begin{definition}(cf.\cite[\S 2]{BP94})\label{equiv_corr_def}
We say that two correspondences $\mathfrak{C}_1, \mathfrak{C}_2$ in $\mathscr{C}$ are \emph{equivalent} if there exists $M\in C(\eta)$ (where $C(\eta)$ is the centralizer of $\eta$ in $\mathrm{PSL}_2(\C)$) such that
$$
(z,w)\in\mathfrak{C}_1\ \iff\ (Mz,Mw)\in\mathfrak{C}_2.
$$ 
\end{definition}
\begin{remark}
Suppose that the correspondences $\mathfrak{C}_1, \mathfrak{C}_2$ are equivalent in the sense of Definition~\ref{equiv_corr_def}. If $\phi$ is a local holomorphic branch of $\mathfrak{C}_1$, then $M\circ\phi\circ M^{-1}$ is a local holomorphic branch of $\mathfrak{C}_2$. Thus, the branches of two equivalent correspondences are M{\"o}bius conjugate.
\end{remark}

A routine computation using Equation~\eqref{corr_eqn_1} and Definition~\ref{equiv_corr_def} shows that two distinct correspondences $\mathfrak{C}_1, \mathfrak{C}_2\in\mathscr{C}$ defined by $R_1, R_2\in\mathrm{Rat}_{d+1}(\C)$ are equivalent if and only if $R_1\equiv R_2\circ M$, for some $M\in C(\eta)$. On the other hand, replacing $R$ by $M\circ R$, for $M\in\pslc$, produces the same correspondence $\mathfrak{C}$.

Therefore, the space of equivalence classes of correspondences in $\mathscr{C}$ is parametrized by the quotient $\ \faktor{\mathrm{Rat}_{d+1}(\C)}{\sim}$ under the equivalence relation
	$$
	R\sim M_2\circ R\circ M_1,
	$$
where $R\in \mathrm{Rat}_{d+1}(\C), M_2\in\pslc$, and $M_1\in C(\eta)$.
The space $\ \faktor{\mathrm{Rat}_{d+1}(\C)}{\sim}$, with its algebraic structure, can be regarded as an analog of the  \emph{character variety} 
for surface groups (the algebraic structure comes from a GIT quotient construction, see \cite{Mum65}; compare \cite[\S 4.3]{Kap01}, \cite{LM85} for related constructions of moduli spaces).

\subsection{A simultaneous uniformization locus of correspondences}\label{simult_unif_corr_subsec}

According to Section~\ref{mating_moduli_subsec}, the rational maps $R$ associated with the conformal matings in $\cM$ are well-defined only up to pre-composition with M{\"o}bius maps in $C(\eta)$ and post-composition with arbitrary M{\"o}bius maps. In light of the discussion in Section~\ref{corr_mod_space_subsec}, each marked conformal mating in the moduli space $\cM$ defines an equivalence class of correspondences in $\ \faktor{\mathrm{Rat}_{d+1}(\C)}{\sim}$ via Equation~\eqref{corr_eqn_1}, where $R$ is the rational uniformizing map of Corollary~\ref{main_hyp_comp_mating_class_cor} normalized so that  $R(1)~=~\mathfrak{X}_\Gamma(1)$. Thus, we have a well-defined map
\begin{equation*}
\begin{split}
\Xi_2:\  \cM \quad \longrightarrow \quad \faktor{\mathrm{Rat}_{d+1}(\C)}{\sim}\ \\
[F= A_{\Gamma}^{\mathrm{fBS}}\mate P, \mathfrak{X}_\Gamma(1)] \quad \mapsto  \quad [\mathfrak{C}].
\end{split}
\label{mating_to_corr}
\end{equation*}
We denote the image of $\Xi_2$ in the `character variety' $\ \faktor{\mathrm{Rat}_{d+1}(\C)}{\sim}$ by
$$
\cC\equiv \cC\left(\mathrm{Teich}^\omega(\pmb{\Gamma_{n,p}}),\mathcal{H}_{d}\right),
$$
and call it the \emph{moduli space of correspondences} associated with $\mathrm{Teich}^\omega(\pmb{\Gamma_{n,p}})$ and $\mathcal{H}_{d}$. Note that the space $\cC$ can be seen as a locus of simultaneous uniformizations of marked groups in $\mathrm{Teich}^\omega(\pmb{\Gamma_{n,p}})$ and polynomials in $\mathcal{H}_{d}$.

\subsection{Intrinsic description of the mating structure of correspondences}\label{corr_mating_struct_intrinsic_subsec}

Let $\mathfrak{C}\in\cC$. By construction, there exists $(\Gamma,P)\in \mathrm{Teich}^\omega(\pmb{\Gamma_{n,p}})\times\mathcal{H}_{d}$ such that $\Xi_2\circ\Xi_1(\Gamma,P)=\mathfrak{C}$. Let $R$ be a rational map generating the correspondence $\mathfrak{C}$. By our normalization, $R(1)=\mathfrak{X}_\Gamma(1)$. 

Recall that while the correspondence $\mathfrak{C}$ was defined solely in terms of the rational map $R$ (via Equation~\eqref{corr_eqn_1}), the dynamical partition for $\mathfrak{C}$ was given in terms of $F=A_{\Gamma}^{\mathrm{fBS}}\mate P$, or equivalently, in terms of the rational map $R$ and the Jordan domain~$\mathfrak{D}$ (recall the relation $F\vert_{\overline{\Omega}}\equiv R\circ\eta\circ (R\vert_{\overline{\mathfrak{D}}})^{-1}$ from Corollary~\ref{main_hyp_comp_mating_class_cor}). We will now expound how the complete dynamical structure of $\mathfrak{C}$ (including the limit, tiling, and non-escaping sets of $\mathfrak{C}$ and the domain $\mathfrak{D}$) can be recovered directly from $R$.  

Since the iterated $F-$preimages of $\mathfrak{X}_\Gamma(1)$ are dense in the limit set $\Lambda$ of $F$  (this follows from the fact that the iterated $P-$preimages of any point on $\mathcal{J}(P)$ are dense in $\mathcal{J}(P)$), it follows that the grand orbit of $1$ under the correspondence $\mathfrak{C}$ is dense in the limit set $\widetilde{\Lambda}$ of $\mathfrak{C}$ (see Figure~\ref{corr_fig}). Hence, the limit set of $\mathfrak{C}$ can be recovered from $R$ (without knowledge of the domain $\mathfrak{D}$). 

The tiling set $\widetilde{\mathcal{T}}$ of $\mathfrak{C}$ can now be recognized as the union of the connected components of $\widehat{\C}\setminus\widetilde{\Lambda}$ on which $\mathfrak{C}$ acts properly discontinuously (with torsion points, when $n>1$). The closures of the other two components of $\widehat{\C}\setminus\widetilde{\Lambda}$ comprise $\widetilde{\mathcal{K}}$. On one of these two components, the map $R$ is injective, while $R$ maps the other component with degree $d$. The closure of the former (respectively, the latter) component is $\widetilde{\mathcal{K}_1}$ (respectively, $\widetilde{\mathcal{K}_2}$).

Thanks to the description of the closure of $\widetilde{\mathcal{T}}$ given in Lemma~\ref{lifted_tiling_top_lem}, we know that the components $U_0,\cdots,U_{p-1}$ of $\widetilde{\mathcal{T}}$ are Jordan domains, and they form a chain such that neighboring components touch at critical points of $R$ that lie on $\widetilde{\Lambda}$. We now consider the Jordan curve $\mathfrak{J}$ obtained by connecting the critical points of $R$ on $\widetilde{\Lambda}$ consecutively by hyperbolic geodesics in the components $U_i$. By the proof of Proposition~\ref{lifted_tiling_quotient_prop}, the map $R$ is injective on one of the complementary components of $\mathfrak{J}$, and this component coincides with $\mathfrak{D}$ (see Figure~\ref{corr_fig}). 

Thus, we can reconstruct $\mathfrak{D}, \widetilde{\mathcal{K}}, \widetilde{\mathcal{T}},$ and $\widetilde{\Lambda}$ from the rational map $R$. Clearly, the set $\widetilde{T^0}\subset\widetilde{\cT}$ (which is the union of the rank zero tiles in the tiling set of $\mathfrak{C}$) and hence the set $\widetilde{T^0_{U_0}}= \widetilde{T^0}\cap U_0$ can also be reconstructed from the above data.

The proof Proposition~\ref{lifted_tiling_quotient_prop} also shows that when the topological disk $U_0$ is uniformized by the unit disk, the set $\widetilde{T^0_{U_0}}$ corresponds to an ideal $np-$gon $\mathfrak{P}$ in $\D$ that admits the rotation $M_\omega$ as a symmetry. Moreover, pulling back a sector of angle $2\pi/n$ in $\mathfrak{P}$ (with geodesic boundary) under this uniformization yields a fundamental domain $\mathfrak{S}$ for the $\mathfrak{C}-$action on $\widetilde{\cT}$ equipped with side-pairing transformations. This defines a marking on the quotient $\faktor{\widetilde{\cT}}{\mathfrak{C}}$. This marked Riemann surface is biholomorphic to
\smallskip

\noindent$\bullet$ a sphere with $\frac{p}{2}+1$ punctures and an order $n$ orbifold point for $p$ even, and
\smallskip

\noindent$\bullet$ a sphere with $\frac{p+1}{2}$ punctures, an order two orbifold point and an order $n$ orbifold point for $p$ odd.
\smallskip

In other words, the correspondence $\mathfrak{C}$ determines a unique element of $\mathrm{Teich}(\widehat{\pmb{\Gamma}}_{n,p})\cong\mathrm{Teich}^\omega(\pmb{\Gamma_{n,p}})$. 

Finally, by Proposition~\ref{lifted_ne_dyn_prop}, an appropriate branch of $\mathfrak{C}$ on $\widetilde{\mathcal{K}}_1$ is conformally conjugate to the action of a polynomial in $\mathcal{H}_{d}$ on its filled Julia set. In fact, such a polynomial is uniquely determined when we require that the conjugacy sends the fixed point $1$ of this correspondence branch to the  landing point of the external dynamical ray at angle $0$ for the polynomial.

The above recipe defines a map 
\begin{equation*}
\begin{split}
\Xi_3:\cC \longrightarrow \mathrm{Teich}^\omega(\pmb{\Gamma_{n,p}})\times\mathcal{H}_{d}
\label{corr_to_teich_hyp}
\end{split}
\end{equation*}
that is, by construction, the inverse of the map $\Xi_2\circ\Xi_1$. This completes the justification of the commutative diagram (Figure~\ref{intro_fig}) presented in the introduction.

\section{A Bers slice for genus zero orbifolds}\label{punc_sphere_bers_sec}
We continue to use the notation of Section~\ref{char_var_sec}. Recall that in that section, we constructed a simultaneous uniformization locus $\cC\equiv\cC\left(\mathrm{Teich}^\omega(\pmb{\Gamma_{n,p}}),\mathcal{H}_{d}\right)$ in the `character variety' $\ \faktor{\mathrm{Rat}_{d+1}(\C)}{\sim}\ $ of bi-degree $d$:$d$ algebraic correspondences on $\widehat{\C}$ defined by Equation~\eqref{corr_eqn_1}. The space $\cC$ is the analog of the quasi-Fuchsian space in our setup. Our next goal is to manufacture a complex-analytic slice in this simultaneous uniformization locus such that the polynomial component is frozen to be $\pmb{P}(z):=z^d$ (in $\mathcal{H}_d$), while the marked groups run through $\mathrm{Teich}^\omega(\pmb{\Gamma_{n,p}})$. This is akin to Bers' original construction of the Bers slice in the quasi-Fuchsian locus (cf. \cite[\S 5.10]{Mar16}).

\subsection{The Bers embedding}\label{bers_embed_subsec}

With the natural identification of $\mathrm{Teich}^\omega(\pmb{\Gamma_{n,p}})$ with $\mathrm{Teich}^\omega(\pmb{\Gamma_{n,p}}) \times\{\pmb{P}\}$, the map
$$
\Xi_2\circ\Xi_1:\mathrm{Teich}^\omega(\pmb{\Gamma_{n,p}}) \times\{\pmb{P}\}\longrightarrow \cC
$$
gives rise to a map
$$
\mathfrak{B}:\mathrm{Teich}^\omega(\pmb{\Gamma_{n,p}})\longrightarrow \cC
$$
(See Subsections~\ref{mating_moduli_subsec},~\ref{corr_mod_space_subsec} for the definitions of $\Xi_1,\Xi_2$.)
 
\begin{remark}
From the discussion in this section, it will follow that the map $\mathfrak{B}$ can be thought of as an analog  of the `Bers embedding' of $\mathrm{Teich}(\widehat{\pmb{\Gamma}}_{n,p})$ into the (analog of the) `quasi-Fuchsian space' $\cC$, where the latter sits inside the 
(analog of the) `character variety' $\ \faktor{\mathrm{Rat}_{d+1}(\C)}{\sim}\ $.
\end{remark}

We will now show that the image of the map $\mathfrak{B}$ can be identified with a subset of $\C^L$, where 
$$
L:=\mathrm{dim}_{\C}\left(\mathrm{Teich}^\omega(\pmb{\Gamma_{n,p}})\right)=\mathrm{dim}_{\C}\left(\mathrm{Teich}(\widehat{\pmb{\Gamma}}_{n,p})\right).
$$

Suppose that a correspondence $\mathfrak{C}$ defined by a degree $(d+1)$ rational map $R$ (via Equation~\eqref{corr_eqn_1}) lies in the image of $\mathfrak{B}$. Then, by Corollary~\ref{crit_pnt_r_cor}, the map $R$ has $p$ critical points on $\partial\mathfrak{D}$, a critical point of multiplicity $np-2$ in $\Int{\widetilde{\mathcal{K}_2}}$, and $p$ distinct critical points, each of multiplicity $n-1$, in $\widetilde{\cT}\setminus\overline{\mathfrak{D}}$. 

Possibly after pre and post-composing $R$ with elements of $C(\eta)$ and $\mathrm{PSL}_2(\C)$ (respectively), we can assume the following.
\begin{enumerate}
\item $\infty$ is the unique superattracting fixed point (of local degree $d$) of the corresponding conformal mating $F$,
\item $\infty\in\mathfrak{D}$ with $R(\infty)=\infty$,
\item $R'(\infty)=1$, and
\item $R'(1)=0$ with $R(1)=\mathfrak{X}_\Gamma(1)$.
\end{enumerate}
As $F$ maps $\infty$ to itself with local degree $np-1$, it follows that $R$ has an order $np-1$ pole at the origin. The conditions $R(\infty)=\infty$ and $R'(\infty)=1$ now imply that $R$ is of the form 
$$
R(z)=\frac{R_1(z)}{z^{np-1}},
$$ 
where $R_1$ is a monic polynomial of degree $np$. To obtain an explicit form of $R_1$, we need to consider various cases.

\subsubsection*{Punctured spheres without orbifold points}
In this case, $n=1$ and $p$ is an even integer. We set $p=2q$, for some $q\geq 2$. 
We first post-compose $R$ with a translation to write it as
$$
R(z)=z+\frac{a_1}{z}+\cdots+\frac{a_{2q-1}}{z^{2q-1}},
$$
for $a_1,\cdots,a_{2q-1}\in\C$. Note also that Corollary~\ref{crit_pnt_r_cor} forces the $2q$ critical points of $R$ on $\partial\mathfrak{D}$ to be of the form 
$$
\{1, -1, c_1,\frac{1}{c_1},\cdots, c_{q-1},\frac{1}{c_{q-1}}\},
$$
for some $c_1,\cdots,c_{q-1}\in\C^*$. 
Differentiating $R$, one sees that the degree $2q$ polynomial
$$
Q(z):= z^{2q}-\sum_{j=1}^{2q-1} j a_j z^{2q-1-j}
$$
has $\{1, -1, c_1,\frac{1}{c_1},\cdots, c_{q-1},\frac{1}{c_{q-1}}\}$ as its roots.
A routine application of Vieta's formula now shows that 
\begin{equation}
R(z)=z+\frac{a_1}{z}+\cdots+\frac{a_{q-2}}{z^{q-2}}+\frac{a_q}{z^q}+\cdots+\frac{a_{2q-3}}{z^{2q-3}}+\frac{1}{(2q-1)\cdot z^{2q-1}},
\label{bers_slice_1}
\end{equation}
where
\begin{equation}
a_{2q-j-1}=-\ \frac{(j-1)}{(2q-j-1)} a_{j-1},\quad j\in\{2,\cdots,q-1\}.
\label{bers_slice_2}
\end{equation}
We identify the rational maps in the image of $\mathfrak{B}$ (where the normalization of these rational maps is given by Equations~\eqref{bers_slice_1} and~\eqref{bers_slice_2}) with their $q-2$ independent complex coefficients $a_1,\cdots,a_{q-2}$. Thus, the image of $\mathfrak{B}$ can be identified with a subset of $\C^{q-2}$. 

We also note that as $\ \faktor{\D}{\widehat{\pmb{\Gamma}}_{1,2q}}$ is a $(q+1)-$times punctured sphere, its Teichm{\"u}ller space has complex dimension $q-2$.

\subsubsection*{Genus zero orbifolds with exactly one orbifold point of order $2$ and no orbifold point  of order $\nu\geq 3$}
In this case, $n=1$ and $p$ is an odd integer. We set $p=2q+1$, for some $q\geq 2$. 
As in the previous case, we can post-compose $R$ with a translation to write it as
$$
R(z)=z+\frac{a_1}{z}+\cdots+\frac{a_{2q}}{z^{2q}},
$$
for $a_1,\cdots,a_{2q}\in\C$. Moreover, Corollary~\ref{crit_pnt_r_cor} implies that the $2q+1$ critical points of $R$ on $\partial\mathfrak{D}$ are the form 
$$
\{1, c_1,\frac{1}{c_1},\cdots, c_{q},\frac{1}{c_{q}}\},
$$
for some $c_1,\cdots,c_{q}\in\C^*$. 
Differentiating $R$, one sees that the degree $2q$ polynomial
$$
Q(z):= z^{2q+1}-\sum_{j=1}^{2q} j a_j z^{2q-j}
$$
has $\{1, c_1,\frac{1}{c_1},\cdots, c_{q},\frac{1}{c_{q}}\}$ as its roots.
Once again, a straightforward computation using Vieta's formula shows that 
\begin{equation}
R(z)=z+\frac{a_1}{z}+\cdots+\frac{a_{q-1}}{z^{q-1}}+\frac{a_q}{z^q}+\cdots+\frac{a_{2q-2}}{z^{2q-2}}+\frac{1}{2q\cdot z^{2q}},
\label{bers_slice_3}
\end{equation}
where
\begin{equation}
a_{2q-j}= -\ \frac{(j-1)}{(2q-j)} a_{j-1},\quad j\in\{2,\cdots,q\}.
\label{bers_slice_4}
\end{equation}
Thus, with the identification of the rational maps in the image of $\mathfrak{B}$ (normalized by Equations~\eqref{bers_slice_3} and~\eqref{bers_slice_4}) with their $q-1$ independent complex coefficients $a_1,\cdots,a_{q-1}$, the map $\mathfrak{B}$ can be thought of as taking values in $\C^{q-1}$. 

Further, as $\faktor{\D}{\widehat{\pmb{\Gamma}}_{1,2q+1}}$ is a genus zero orbifold with $(q+1)$ punctures and one order two orbifold point, its Teichm{\"u}ller space has complex dimension $q-1$.

\subsubsection*{Genus zero orbifolds with exactly one orbifold point  of order $\nu\geq 3$ and at most one orbifold point of order $2$}
In this case, $n=\nu\geq 3$ and $p$ is odd (respectively, even) depending on whether the orbifold has (respectively, does not have) an order two orbifold point. Recall that by Corollary~\ref{crit_pnt_r_cor}, in addition to the $(np-2)-$fold critical point at the origin, the map $R$ has 
\begin{itemize}
\item $p$ distinct critical points on $\partial\mathfrak{D}$, of which one/two are fixed by $\eta$ (depending on whether $p$ is odd/even) and the others form $2-$cycles under $\eta$, and
\item $p$ distinct critical points, each of multiplicity $n-1$, in $\widetilde{\cT}\setminus\overline{\mathfrak{D}}$, and all these critical points are mapped to a common critical value in $\cT$.
\end{itemize} 
It will be convenient to post-compose $R$ with a translation such that the critical value of $R$ in $\cT$ is at the origin. Then $R$ has precisely $p$ distinct zeroes at $a_1,\cdots, a_p$, each of  multiplicity $n-1$. Therefore, 
$$
R(z)=\frac{(z-a_1)^n\cdots (z-a_p)^n}{z^{np-1}}.
$$
In particular, the coefficients of $R$ can be written in terms of the elementary symmetric polynomials $e_1,\cdots, e_p$ in $a_1,\cdots, a_p$. Using (logarithmic) differentiation, one now easily sees that the $p$ critical points of $R$ on $\partial\mathfrak{D}$ are roots of the equation
\begin{equation}
\begin{split}
n\left(\sum_{j=1}^p \frac{a_j}{z-a_j}\right) +1 =0,\\
\iff\ Q(z):= z^p+\sum_{j=1}^p (-1)^j e_j (1-nj) z^{p-j}=0.
\end{split}
\label{log_diff_eqn}
\end{equation}
Since the roots of $Q$ are of the form 
$$
\{1, -1, c_1,\frac{1}{c_1},\cdots, c_{q-1},\frac{1}{c_{q-1}}\},
$$
when $p=2q$, and of the form
$$
\{1, c_1,\frac{1}{c_1},\cdots, c_{q},\frac{1}{c_{q}}\},
$$
when $p=2q+1$, we are now reduced to the computations carried out in the previous two cases. In particular, it follows that if $p=2q$ (respectively, $p=2q+1$), the polynomial $Q$ has $q-1$ (respectively, $q$) independent coefficients. Since the coefficients of $Q$ are multiples of the elementary symmetric polynomials $e_1,\cdots, e_p$ (see Equation~\eqref{log_diff_eqn}), we conclude that only $q-1$ (respectively, $q$) of these elementary symmetric polynomials are unconstrained. Hence, the rational map $R$ also has $q-1$ (respectively, $q$) independent complex coefficients. As in the previous cases, the image of $\mathfrak{B}$ can therefore be identified with a subset of $\C^{q-1}$ (respectively, of $\C^q$). 
Finally, we remark that $\mathrm{Teich}(\widehat{\pmb{\Gamma}}_{n,2q})$ (respectively, $\mathrm{Teich}(\widehat{\pmb{\Gamma}}_{n,2q+1})$) has complex dimension $q-1$ (respectively, $q$).

\subsubsection*{Complex-analyticity of $\mathfrak{B}$}

The preceding analysis shows that the image of the map $\mathfrak{B}$ can be identified with a subset of $\C^L$, where 
$L=\mathrm{dim}_{\C}\left(\mathrm{Teich}(\widehat{\pmb{\Gamma}}_{n,p})\right).$

Recall that the Teichm{\"u}ller space of an orbifold (or a Fuchsian group) can be endowed with a complex structure via the Bers simultaneous uniformization theorem. Specifically, in the statement below, we identify $\mathrm{Teich}(\widehat{\pmb{\Gamma}}_{n,p})$ with the Bers slice $\mathcal{B}(\widehat{\pmb{\Gamma}}_{n,p})$ in the space of quasi-Fuchsian representations of $\widehat{\pmb{\Gamma}}_{n,p}$ (see \cite[\S 5.10]{Mar16}).

\begin{proposition}\label{bers_embedding_corr_prop}
$\mathfrak{B}:\mathrm{Teich}(\widehat{\pmb{\Gamma}}_{n,p})\longrightarrow \C^{L}$ is a biholomorphism onto its image.
\end{proposition}
\begin{proof}
We recall the notation $\pmb{P}(z)=z^{np-1}$. Let $\mathfrak{X}_{\pmb{P}}:\overline{\D}\to\widehat{\C}$ and $\mathfrak{X}_{\pmb{\Gamma_{n,p}}}:\overline{\D}\to\widehat{\C}$ be the mating conjugacies associated with the conformal mating $\pmb{F}$ of $\pmb{P}$ and $A_{\pmb{\Gamma_{n,p}}}^{\mathrm{fBS}}$ (see Definition~\ref{conf_mat_def}).

Each representation $(\widehat{\rho}:\widehat{\pmb{\Gamma}}_{n,p}\to\widehat{\Gamma})\in\mathcal{B}(\widehat{\pmb{\Gamma}}_{n,p})$ (see Section~\ref{factor_bs_gen_subsec}) is given by 
$$
\widehat{\rho}(g)=\psi_\rho\circ g\circ\psi_\rho^{-1},\ g\in\widehat{\pmb{\Gamma}}_{n,p},
$$ 
where $\psi_\rho$ is a quasiconformal homeomorphism of $\widehat{\C}$ that is conformal on $\D^*$.
Moreover, the quasiconformal maps $\psi_\rho$ depend complex-analytically on representations $\widehat{\rho}\in\mathcal{B}(\widehat{\pmb{\Gamma}}_{n,p})$. We define the $\widehat{\pmb{\Gamma}}_{n,p}-$invariant Beltrami coefficient $\mu_\rho:=\psi_\rho^*(\mu_0)$ (where $\mu_0$ is the trivial Beltrami coefficient), and note that $\mu_\rho$ also depends complex-analytically on $\widehat{\rho}$. We further push $\mu_\rho$ forward to the dynamical plane of $A_{\pmb{\Gamma_{n,p}}}^{\mathrm{fBS}}$, and continue to call it $\mu_\rho$.

It follows that the $\pmb{F}-$invariant Beltrami coefficients 
$$
\mu_{\pmb{F},\rho}:=\begin{cases}
\left(\mathfrak{X}_{\pmb{\Gamma_{n,p}}}\right)_*(\mu_\rho)\quad \textrm{on}\quad \mathfrak{X}_{\pmb{\Gamma_{n,p}}}(\D),\\
\quad 0\hspace{2cm} \textrm{elsewhere},
\end{cases}
$$ 
depend complex-analytically on $\widehat{\rho}\in\mathcal{B}(\widehat{\pmb{\Gamma}}_{n,p})$. Consequently, the (normalized) quasiconformal maps $\phi_\rho$ solving the Beltrami equation with coefficient $\mu_{\pmb{F},\rho}$ depend complex-analytically on $\widehat{\rho}$. Furthermore, the map $\phi_\rho\circ\pmb{F}\circ\phi_\rho^{-1}$ is the conformal mating of $\pmb{P}$ and $A_{\Gamma}^{\mathrm{fBS}}$ with mating conjugacies $\phi_\rho\circ\mathfrak{X}_{\pmb{P}}$ and $\phi_\rho\circ\mathfrak{X}_{\pmb{\Gamma_{n,p}}}\circ\widehat{\psi}_\rho^{-1}$, where $\widehat{\psi}_\rho$ is the quasiconformal conjugacy between $A_{\pmb{\Gamma_{n,p}}}^{\mathrm{fBS}}$ and $A_{\Gamma}^{\mathrm{fBS}}$ induced by $\psi_\rho$.

Let $\pmb{R}$ be the normalized rational map associated with the conformal mating $\pmb{F}$. Let $\widehat{\phi}_\rho$ be a quasiconformal map solving the Beltrami equation with coefficient $\pmb{R}^*(\mu_{\pmb{F},\rho})$. As the Beltrami coefficients $\pmb{R}^*(\mu_{\pmb{F},\rho})$ depend complex-analytically on $\widehat{\rho}$, the same is true for the maps $\widehat{\phi}_\rho$.
By the proof of Proposition~\ref{mating_class_prop},
$$
R_\rho:=\phi_\rho\circ\pmb{R}\circ\widehat{\phi}_\rho^{-1}
$$
is a rational map associated with the conformal mating $\phi_\rho\circ\pmb{F}\circ\phi_\rho^{-1}$. Since both families of quasiconformal maps $\{\phi_\rho\}_{\widehat{\rho}\in\mathcal{B}(\widehat{\pmb{\Gamma}}_{n,p})}$ and $\{\widehat{\phi}_\rho\}_{\widehat{\rho}\in\mathcal{B}(\widehat{\pmb{\Gamma}}_{n,p})}$ depend complex-analytically on $\widehat{\rho}$, it follows that the coefficients of $R_\rho$ also depend complex-analytically on $\widehat{\rho}$.
Hence, the map $\mathfrak{B}:\mathcal{B}(\widehat{\pmb{\Gamma}}_{n,p})\longrightarrow \C^L$ is complex-analytic. 

The existence of the inverse map $\Xi_3$ in Section~\ref{corr_mating_struct_intrinsic_subsec} shows that the map $\mathfrak{B}$ is injective. Since the complex dimension of $\mathrm{Teich}(\widehat{\pmb{\Gamma}}_{n,p})\cong\mathcal{B}(\widehat{\pmb{\Gamma}}_{n,p})$ is $L$, it follows that $\mathfrak{B}$ is a biholomorphism onto its image (cf. \cite[Theorem~2.14]{range}).
\end{proof}

\subsection{Proof of Theorem~\ref{punc_sphere_teich_corr_intro_thm}}\label{thm_c_prf_subsec}

\begin{proof}[Proof of Theorem~\ref{punc_sphere_teich_corr_intro_thm}]
Note that by construction, the map $\mathfrak{B}$ sends each representation $\widehat{\rho}:\widehat{\pmb{\Gamma}}_{n,p}\to\widehat{\Gamma}$ in the Bers slice $\mathcal{B}(\widehat{\pmb{\Gamma}}_{n,p})$ to a bi-degree $(np-1)$:$(np-1)$ algebraic correspondence $\mathfrak{C}$ on $\widehat{\C}$ that is a mating of $z^{np-1}$ and $\faktor{\D}{\widehat{\Gamma}}$ in the sense of Theorem~\ref{corr_mating_thm_1}. The $L$ complex coefficients of the normalized rational maps $R$ defining these correspondences $\mathfrak{C}$ endow the resulting space of correspondences with a complex manifold structure. By Proposition~\ref{bers_embedding_corr_prop}, the map $\mathfrak{B}$ yields a biholomorphism  between the above complex manifold and the Bers slice $\mathcal{B}(\widehat{\pmb{\Gamma}}_{n,p})$.
\end{proof}

We conclude this section with the following question.

\begin{question}
Let $\Sigma\in\cS$, and $L:=\dim_{\C}(\mathrm{Teich}(\Sigma))$.
\noindent\begin{enumerate}
\item Is the image $\mathfrak{B}(\mathrm{Teich}(\Sigma))$ pre-compact in $\C^{L}$?
\item Describe the dynamics of the correspondences that lie on the boundary of $\mathfrak{B}(\mathrm{Teich}(\Sigma))$. In particular, do Bers boundary groups not treated in \cite[Section 7]{MM1} arise?
\end{enumerate}
\end{question}

\section{Bullett-Penrose-Lomonaco correspondences}\label{bp_corr_sec}

As mentioned in the introduction, in the special case $n=3$ and $p=1$, the correspondences produced by Theorem~\ref{corr_mating_intro_thm} belong to the family of bi-degree $2$:$2$ correspondences studied by Bullett-Penrose-Lomonaco \cite{BP94,BL20a,BL20b,BL22}. In this section, we will derive explicit formulae for these correspondences using our conformal matings framework, and show that they can indeed be brought to the Bullett-Penrose normal form (see \cite[Equation~1.1]{BP94}).

Recall from Section~\ref{hecke_factor_bs_subsec} that the index three extension $\widehat{\pmb{\Gamma}}_{3,1}$ of $\pmb{\Gamma_{3,1}}$ is M{\"o}bius conjugate to the standard modular group $\textrm{PSL}_2(\Z)$. In particular, $\mathrm{Teich}^\omega(\pmb{\Gamma_{3,1}})\cong\mathrm{Teich}(\widehat{\pmb{\Gamma}}_{3,1})$ is a singleton. Further let $P$ be a polynomial lying in a real-symmetric hyperbolic component in the connectedness locus of degree $np-1=2$ polynomials; i.e., $P$ is a quadratic polynomial in a real-symmetric hyperbolic component of the Mandelbrot set. We denote the conformal mating of $A_{\pmb{\Gamma_{3,1}}}^{\mathrm{fBS}}$ and $P$ by $F$.

Since $p=1$, it follows that the set $\mathcal{A}_p$ is a singleton, and hence the lamination $\mathcal{L}_P$ is empty (see Section~\ref{lami_model_subsec}). Therefore, $\Omega:=\Int{\mathrm{Dom}(F)}$ is a Jordan domain. By Proposition~\ref{mating_class_prop}, there exist a cubic rational map $R$ and a Jordan domain $\mathfrak{D}$ with $\eta(\partial\mathfrak{D})=\partial\mathfrak{D}$ such that $R$ carries $\mathfrak{D}$ injectively onto $\Omega$ and $F\vert_{\Omega}\equiv R\circ\eta\circ (R\vert_{\mathfrak{D}})^{-1}$. (The Jordan domain $\mathfrak{D}$ is depicted as the disk $\Delta_J^{st}$ in \cite[Figure~4]{BL20b}.) By Corollary~\ref{crit_pnt_r_cor}, the map $R$ has a critical point $c_1\in\partial\mathfrak{D}$ that is fixed under $\eta$. Moreover, the same corollary says that $R$ has a simple critical point $c_2\in R^{-1}(\cK)\setminus\overline{\mathfrak{D}}$ and a double critical point $c_3\in R^{-1}(\cT)\setminus\overline{\mathfrak{D}}$.

Also note that $\widetilde{\cT}$ is a simply connected domain, and $R:\widetilde{\cT}\to\cT$ is a degree three branched covering with a double critical point at $c_3$. By Theorem~\ref{corr_mating_thm_2}, the action of the associated correspondence $\mathfrak{C}$ on $\widetilde{\cT}$ is conformally conjugate to the action of the modular group $\widehat{\pmb{\Gamma}}_{3,1}$ on $\D$, and $\mathfrak{C}$ is a mating of $P$ and the modular surface $\faktor{\D}{\widehat{\pmb{\Gamma}}_{3,1}}$.

We will now bring the correspondence $\mathfrak{C}$ to the Bullett-Penrose-Lomonaco normal form. Let $M_1, M_2$ be M{\"o}bius maps such that
$$
M_1(c_1)=1,\ M_1(c_2)=-1,\ M_1(c_3)=\infty,
$$
and
$$
M_2(R(c_1))= -2,\ M_2(R(c_2)=2,\ M_2(R(c_3)=\infty.
$$
We set $R_1:= M_2\circ R\circ M_1^{-1}$. Then, $R_1$ has a double critical point at $\infty$ with the associated critical value also at $\infty$, and hence $R_1$ is a cubic polynomial. An elementary calculation using the facts that the two finite critical points of $R_1$ are at $\pm 1$ and the associated critical values are at $\mp 2$ now shows that $R_1(u)=u^3-3u$. We also set 
$$
\Omega_1=M_2(\Omega),\quad \eta_1:=M_1\circ\eta\circ M_1^{-1},\quad F_1\equiv M_2\circ F\circ M_2^{-1},
$$
and observe that
$$
F_1\vert_{\Omega_1}\equiv M_2\circ R\circ\eta\circ\left(R\vert_{\mathfrak{D}}\right)^{-1}\circ M_2^{-1}\equiv R_1\circ \eta_1\circ(R_1\vert_{M_1(\mathfrak{D})})^{-1}.
$$
Note that the involution $\eta_1$ fixes $1$ and $a:=M_1(-1)$, and hence can be written as $\eta_1(u)=\frac{(a+1)u-2a}{2u-(a+1)}$. We will change coordinates so that $\eta_1$ becomes the involution $z\mapsto -z$. To this end, we define $R_2:=R_1\circ M_3^{-1}$, where $M_3(u)=\frac{u-1}{a-u}$ sends the fixed points $1,a$ of $\eta_1$ to $0,\infty$, respectively. The conjugated involution $\eta_2:=M_3\circ\eta_1\circ M_3^{-1}$ fixes $0,\infty$, and thus can be written as $\eta_2(z)=-z$. Finally,
$$
F_1\vert_{\Omega_1}\equiv R_2\circ M_3\circ \eta_1\circ M_3^{-1} \circ (R_2\vert_{M_3\circ M_1(\mathfrak{D})})^{-1}\equiv R_2\circ\eta_2\circ (R_2\vert_{M_3\circ M_1(\mathfrak{D})})^{-1}.
$$
The associated correspondence (which is obtained by lifting $F_1$ and its backward branches by $R_2$) is given by
\begin{align*}
(X,Y)\in\mathfrak{C} & \iff R_2(Y)-R_2(\eta_2(X))=0,\ Y\neq\eta_2(X)\\
& \iff R_1(M_3^{-1}(Y))-R_1(M_3^{-1}(-X))=0,\ Y\neq -X\\
& \iff \left(\frac{aY+1}{Y+1}\right)^3-3\left(\frac{aY+1}{Y+1}\right)=\left(\frac{-aX+1}{-X+1}\right)^3-3\left(\frac{-aX+1}{-X+1}\right),\\ 
& \hspace{9cm} Y\neq-X,\\
& \iff  \left(\frac{aY+1}{Y+1}\right)^2+ \left(\frac{aY+1}{Y+1}\right) \left(\frac{aX-1}{X-1}\right)+ \left(\frac{aX-1}{X-1}\right)^2=3.
\end{align*}

Thus, the correspondence $\mathfrak{C}$ belongs to the family of bi-degree $2$:$2$ correspondences a la Bullett-Penrose-Lomonaco \cite{BP94,BL20a,BL20b}.

\begin{remark}
More generally, when $p=1$ and $n\geq 3$, the uniformizing rational maps $R$ can be chosen as degree $n$ polynomials. The associated correspondences $\mathfrak{C}$ are matings of degree $(n-1)$ polynomials $P$ and the genus zero orbifold $\Sigma=\faktor{\D}{\widehat{\pmb{\Gamma}}_{n,1}}$ with exactly one puncture, exactly one order two orbifold point, and exactly one order $n\geq 3$ orbifold point. Note that $\widehat{\pmb{\Gamma}}_{n,1}$ has an index two subgroup $\pmb{\Gamma}^{\ast}_{n,1}$ that uniformizes the genus zero orbifold $\Sigma^*$ with  exactly one puncture, exactly two order $n\geq 3$ orbifold points, and no other orbifold point.
The correspondences $\mathfrak{C}$ admit index two subcorrespondences that are matings of $P^{\circ 2}$ (polynomials of degree $(n-1)^2$) and orbifolds $\Sigma^*$ double covering $\Sigma$ (cf. \cite[\S 4.3]{B00}).
\end{remark}

\end{document}